\documentclass{article}
\usepackage{
amsmath,
amsfonts,
latexsym, 
amssymb
}

\usepackage[non-sorted-cites]{amsrefs}
\usepackage{hyperref}
\usepackage{bm}

\usepackage{tocloft}
\usepackage{pstricks}
\usepackage{subcaption}

\usepackage{tikz}
\usepackage{tikz,fullpage}
\usetikzlibrary{arrows,%
                petri,%
                topaths, automata}%
\usepackage{floatrow}

\usepackage{graphicx}
\DeclareMathOperator{\Prob}{\mathbb{P}}   
\DeclareMathOperator{\cov}{Cov}             
\DeclareMathOperator{\Cov}{Cov}             

\newcommand{\1}{\mathds{1}}
\usepackage{enumerate}
\usepackage{mathrsfs}
\usepackage{wrapfig}

\usepackage{color}

\newcommand{\Id}{{\mathrm{Id}}}

\numberwithin{equation}{section}

\newcommand{\btheta}{\bm{\theta}}

\newcommand{\eps}{\varepsilon}

\newcommand{\rd}{{\rm d}}

\newcommand{\mc}{\mathcal}

\newcommand{\be}{\begin{equation}}
\newcommand{\ee}{\end{equation}}

\newcommand{\e}{{\varepsilon}}

\newcommand{\T}{\mathbb T}

\usepackage{amsmath} 
\usepackage{amssymb}
\usepackage{amsthm}

\newcommand{\wt}{\widetilde}

\newcommand{\ii}{\mathrm{i}} 

\renewcommand{\epsilon}{\varepsilon}
\renewcommand{\leq}{\leqslant}
\renewcommand{\geq}{\geqslant}

\newcommand{\E}{\mathbb{E}}
\newcommand{\R}{\mathbb{R}}

\newcommand{\N}{\mathbb{N}}
\newcommand{\Z}{\mathbb{Z}}

\def\Xint#1{\mathchoice
{\XXint\displaystyle\textstyle{#1}}%
{\XXint\textstyle\scriptstyle{#1}}%
{\XXint\scriptstyle\scriptscriptstyle{#1}}%
{\XXint\scriptscriptstyle\scriptscriptstyle{#1}}%
\!\int}
\def\XXint#1#2#3{{\setbox0=\hbox{$#1{#2#3}{\int}$ }
\vcenter{\hbox{$#2#3$ }}\kern-.6\wd0}}

\def\dashint{\Xint-}

\newcommand{\mb}{\mathbb}

\DeclareMathOperator{\tr}{Tr}
\DeclareMathOperator{\Tr}{Tr}

\DeclareMathOperator{\Var}{Var}

\DeclareMathOperator{\re}{Re}
\DeclareMathOperator{\im}{Im}

\DeclareMathOperator{\sgn}{sgn}

\DeclareMathOperator{\OO}{O}
\DeclareMathOperator{\oo}{o}

\theoremstyle{plain} 
\newtheorem{theorem}{Theorem}[section]
\newtheorem*{theorem*}{Theorem}
\newtheorem{lemma}[theorem]{Lemma}
\newtheorem*{lemma*}{Lemma}
\newtheorem{corollary}[theorem]{Corollary}
\newtheorem*{corollary*}{Corollary}
\newtheorem{proposition}[theorem]{Proposition}
\newtheorem*{proposition*}{Proposition}

\newtheorem{definition}[theorem]{Definition}
\newtheorem*{definition*}{Definition}

\newtheorem*{example*}{Example}
\newtheorem{remark}[theorem]{Remark}

\newtheorem*{remark*}{Remark}
\newtheorem*{remarks*}{Remarks}

\makeatletter
\renewcommand{\subsection}{\@startsection
{subsection}
{2}
{0mm}
{-\baselineskip}
{0 \baselineskip}
{\normalfont\bf\itshape}} 
\makeatother

\usepackage{dsfont}
\usepackage{stmaryrd}

\newcommand{\cor}{\color{red}}

\newcommand{\nc}{\normalcolor}

\def\@empty{}

\def\author#1{\par
    {\centering{\authorfont#1}\par\vspace*{0.05in}}
}

\def\titlefont{\fontsize{13}{15}\bfseries\boldmath\selectfont\centering{}}
\def\authorfont{\fontsize{13}{15}}
\def\abstractfont{\fontsize{8}{10}}

\let\affiliationfont\rhfont

\def\address#1{\par
    {\centering{\affiliationfont#1\par}}\par\vspace*{11pt}
}

\def\body{
\setcounter{footnote}{0}
\def\thefootnote{\alph{footnote}}
\def\@makefnmark{{$^{\rm \@thefnmark}$}}
}

\def\title#1{
    \thispagestyle{plain}
    \vspace*{-14pt}
    \vskip 79pt
    {\centering{\titlefont #1\par}}%
    \vskip 1em
}

\renewenvironment{abstract}{\par%
    \vspace*{6pt}\noindent 
    \abstractfont
    \noindent\leftskip10pt\rightskip10pt
}{%
  \par}

\newcommand{\f}[1]{\boldsymbol{\mathrm{#1}}}

\usepackage{tocloft}

\newcommand{\mf}{\mathfrak}
\newcommand{\mscr}{\mathscr}

\makeatletter
\renewcommand{\section}{\@startsection
{section}
{1}
{0mm}
{-2\baselineskip}
{1\baselineskip}
{\normalfont\large\scshape\centering}} 
\makeatother

\usepackage{verbatim}

\begin{document}

~\vspace{-1cm}

~\hspace{-0.3cm}

\title{
Liouville quantum gravity from random matrix dynamics}

\vspace{0.6cm}

\noindent\begin{minipage}[b]{0.5\textwidth}

 \author{Paul Bourgade}

\address{Courant Institute,\\ New York University\\
   bourgade@cims.nyu.edu}

 \end{minipage}
\noindent\begin{minipage}[b]{0.5\textwidth}
 \author{Hugo Falconet}

\address{Courant Institute,\\ New York University\\
  hugo.falconet@cims.nyu.edu}

 \end{minipage}

\begin{abstract}
We establish the first connection between $2d$ Liouville quantum gravity and natural dynamics of random matrices.  In particular, we show that if $(U_t)$ is a Brownian motion on the unitary group at equilibrium, then the measures
$$
|\det(U_t - e^{\ii \theta})|^{\gamma} \rd t\rd \theta
$$
converge in the limit of large dimension to the $2d$ LQG measure, a properly normalized exponential of the $2d$ Gaussian free field. Gaussian free field type fluctuations associated with these dynamics were first established by Spohn (1998) and convergence to the LQG measure in $2d$ settings was conjectured since the work of Webb (2014), who proved the convergence of related one dimensional measures by using inputs from Riemann-Hilbert theory.  

The convergence follows from the first multi-time extension of the result by Widom (1973) on Fisher-Hartwig asymptotics of Toeplitz determinants with real symbols. To prove these, we develop a general surgery argument and combine determinantal point processes estimates with stochastic analysis on Lie group, providing in passing a probabilistic proof of Webb's $1d$ result.  We believe the techniques will be more broadly applicable to matrix dynamics out of equilibrium,  joint moments of determinants for classes of correlated random matrices, and the characteristic polynomial of non-Hermitian random matrices.

\noindent 
\end{abstract}

\vspace{-0.5cm}

\tableofcontents

\vspace{-0.5cm}
\noindent

\begin{figure}[h]
\begin{tikzpicture}
\draw[white,line width=0pt] (-.1\textwidth, -.1\textheight) rectangle (.1\textwidth, .1\textheight);
\pgftext[center,at={\pgfpoint{-95}{0}}]{\includegraphics[width=3cm,height=4cm]{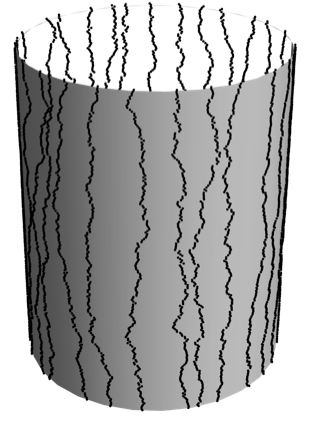}}
\pgftext[center,at={\pgfpoint{95}{0}}]{\includegraphics[width=4cm,height=5cm]{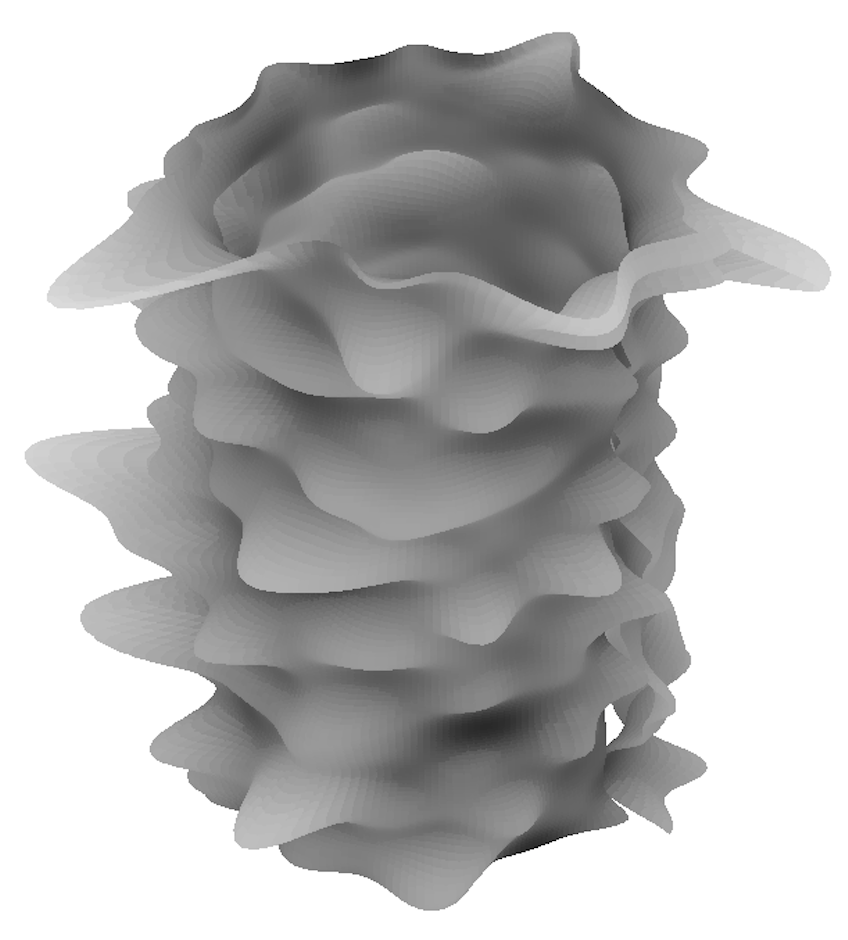}}
\draw [-stealth](-1.2,0) -- (1,0);
\node[rectangle, draw=none, minimum size=1pt] at (-0.1, 0.5) {$e^{\gamma H(\theta, t)}$};
\node[rectangle, draw=none, minimum size=1pt] at (-0.3, -1) {$H=\sum\log|z_k(t)-e^{\ii\theta}|$};
\draw [-stealth](-5,-0.6) -- (-5,0.6);
\node[rectangle, draw=none, minimum size=1pt] at (-5.3, 0) {$t$};
\end{tikzpicture}
\caption*{The (random) Gibbs measure $e^{\gamma H}$ converges to the Gaussian multiplicative chaos on the cylinder,  with $H$ the electric field associated to non-intersecting Brownian motions  on the circle, $\bold{z}(t)$.}
\end{figure}
\newpage
\section{Introduction}

The Gaussian multiplicative chaos (GMC), introduced by Kahane in \cite{Kah1985},  is the fractal measure
$$
e^{\gamma \phi(z)}  \rd z := \lim_{\eps \to 0} e^{\gamma \phi_{\eps}(z) - \frac{\gamma^2}{2} \E(\phi_\eps(z)^2)}  \rd z,
$$
where $\phi_{\eps}$ is a mollification of a $\log$-correlated Gaussian field $\phi$ on a domain $D \subset \mb{R}^d$ and $\rd z$ denotes the Lebesgue measure on $\mathbb{R}^d$. The regularization and renormalization are necessary because of the negative Sobolev regularity of the field. The convergence holds in probability with respect to the topology of weak convergence and the parameter $\gamma \in (0, \sqrt{2d})$ since the limit is zero above this range \cite{RobertVargas,  Sha2016, Ber2017, RhoVar2014}. The specific case where $\phi$ is a two dimensional Gaussian free field (GFF) (a Gaussian field whose covariance function is  the inverse of the Laplacian) or a one dimensional restriction thereof, has proved to be connected with many different domains in mathematical physics. To name a few, it is the volume form in Liouville quantum gravity (LQG), a metric measure space corresponding to the formal Riemannian metric tensor $``e^{\gamma \phi} (\rd x^2 +\rd y^2)"$  \cite{Pol1981, DupShe2011, DDDF2020, GwyMil2021}; appears in the scaling limit of random planar maps \cite{LeGall2013, Mie2013, MilShe2020, HolSun2019};  interplays through conformal welding with Schramm Loewner Evolutions and the Conformal Loop Ensemble, the scaling limit of interfaces in critical spins and percolation models \cite{AstJonKupSak2011, DupMilShe2014, She2016, MilShe2016, AngSun2021}; played a central role in the rigorous formulation and the resolution of Liouville Conformal Field Theory \cite{DavKupRhoVar2016, DOZZ, GKRVboostrap}; and appears in the construction of a stochastic version of the Ricci flow \cite{DubShe2019}. The literature on this topic is abundant and we refer to the survey \cite{She2022} and references therein.

\medskip

The Brownian motion on the unitary group $U(N)$ is a rich object in random matrix theory. It preserves the Haar measure and, under this initial condition, its eigenvalues have Circular Unitary Ensemble distribution at each fixed time. They satisfy the Dyson dynamics \cite{Dyson62} on the circle and, by the Karlin-McGregor formula \cite{KarMcGreg1959}, can be seen as Brownian motions on the unit circle conditioned not to intersect. As ubiquitous in random matrix theory, we are concerned with the large $N$ limit of observables of this process. The large $N$ limit of the unitary Brownian motion itself is the free unitary Brownian motion \cite{Bia1997II, ColDahKem2018} and this has applications to the large $N$ limit of the Yang-Mills measure on the Euclidean plane with unitary structure group as observed in \cite{Levy2017}. In this paper, we prove the following 

\begin{theorem}
\label{thm:gmc-cv} Let $(U_t)$ be a unitary Brownian motion at equilibrium, as defined in (\ref{eqn:Unitbis}). Then
for every $\gamma \in (0,2\sqrt{2})$,
\begin{equation}
\label{eq:LQGconvergence}
\lim_{N \to \infty }\frac{|\det(U_{t}-e^{\ii \theta})|^\gamma}{\E( |\det(U_{t}-e^{\ii \theta})|^{\gamma} )} \rd t \rd\theta = e^{\gamma h(z)} \rd z
\end{equation}
where $h$ is the Gaussian free field on the cylinder $\mb{R} \times \mb{R}/2\pi \mb{Z}$, $\E(h(z)h(w))= \pi (-\Delta_{\mc{C}})^{-1}(z,w)$, where $\Delta_{\mc{C}} = \partial_t^2 + \partial_\theta^2$.  Moreover,  the convergence is in distribution with respect to the weak topology.
\end{theorem}
The usual parametrization $\gamma \in (0,\sqrt{2d})$ in GMC theory corresponds to $\log$-correlated fields. Here, the field is $\frac{1}{2}\log$-correlated and by a change of parametrization our result covers this entire range (see \eqref{eq:cov-field} below for an exact formula of the covariance of this free field,  and background). \\

In \cite{Web2015}, Webb opened a connection between Gaussian multiplicative chaos and random matrix theory by linking the characteristic polynomial of the Circular Unitary Ensemble (CUE) to a one-dimensional GMC and conjectured that similar results also hold for the Gaussian Unitary Ensemble, one-dimensional $\beta$-ensembles, and more generally for random matrix models presenting log-correlations, including in dimension two.  His proof and the ones of the following works \cite{BerWebWon2017, NikSakWeb20} relied on existing results for Fisher-Hartwig asymptotics based on the Riemann-Hilbert approach (or adaptations thereof). 
Another approach  appeared in \cite{ChaNaj}, still for $d=1$,  which showed that the limit of an object different from the characteristic polynomial,  the spectral measure of circular $\beta$-ensembles, coincides with a Gaussian multiplicative chaos.
In our paper, as an application of our main theorem below, we provide the first  convergence to the $2d$ LQG measure, taking a new angle in viewing this problem as one in random matrix dynamics. \\

By considering the unit disk instead of a semi-infinite cylinder (i.e., replacing $\rd t \rd\theta$ by $e^{-2t} \rd t \rd\theta$ with $z = e^{-t} e^{\ii \theta}$, $t\in (0,\infty)$ in the limit \eqref{eq:LQGconvergence}),  Theorem \ref{eq:LQGconvergence} translates into convergence towards the measure $e^{\gamma \bar{h}(z)} \rd z$ on the unit disk where $\bar{h}$ is the lateral part in the polar decomposition of the 2d whole plane GFF $h$, i.e. $\bar{h}(z) = h(z) - \dashint_{|z| \mb{U}} h$; the subtracted process $r \mapsto  \dashint_{r \mb{U}} h$ being a Brownian motion independent of $\bar{h}$. The field $\bar{h}$ and the associated chaos measure were introduced in the mating of trees \cite{DupMilShe2014}, used in the proof of the DOZZ formula \cite{DOZZ}  and the dynamics of the restriction of $\bar{h}$ on concentric circles played a crucial role in the proof of the conformal bootstrap in Liouville theory \cite{GKRVboostrap}.  The unitary Brownian motion is the most natural model among random matrix dynamics that induce the field $\bar{h}$ and its own dynamics.

\subsection{Multi-time Fisher-Hartwig asymptotics}.\ The main contribution of this paper is the dynamical extension of asymptotics of Toeplitz determinants with singularities.  In the following discussion, the Fourier transform is normalized as $\hat f_k=\int_0^{2\pi} f(e^{\ii\theta})e^{-\ii k\theta}\frac{\rd\theta}{2\pi}$ and we
let 
\[(f,g)_{\rm H}=(f,g)_{{\rm H}^{1/2}}=\sum_{k\in\mathbb{Z}} |k|\,\hat f_k\hat g_{-k}.\]
The Toeplitz determinant  $D_N(f)=\det(\hat f_{j-k})_{j,k=0}^{N-1}$ has been the subject
of many investigations. For example,  a simple version of the strong Szeg{\H o} theorem states that if $f=e^V$ with $V$ real-valued and smooth enough, $D_N(f)\sim \exp(N\hat V_0+\frac{1}{2}\|V\|^2_{{\rm H}})$ for large dimension.

For a wide class of irregular functions $f$, Fisher and Hartwig \cite{FisHar} made a seminal general conjecture about the asymptotic form of $D_N(f)$, which has been corrected by Basor and Tracy \cite{BasTra} and is settled in full generality \cite{DeiItsKra2011} by Riemann-Hilbert methods, after multiple important contributions, e.g.\cite{Bas,Wid,Ehr}.
For example, in the special case where $f(z)=e^{V(z)}\prod_{j=1}^m|z-z_j|^{2\alpha_j}$ with $m\geq 1$ fixed singularities $z_j$ on the unit circle, $\alpha_j>-1/2$,  and smooth centered real $V$,
the Fisher-Hartwig asymptotics  states that
\begin{equation}\label{eqn:Widom}
D_N(f)=e^{\frac{1}{2}\|V\|^2_{{\rm H}}-\sum_{j=1}^m\alpha_jV(z_j)}\, N^{\sum_{j=1}^m \alpha_j^2}
\prod_{1\leq j<k\leq m}|z_j-z_k|^{-2\alpha_j\alpha_k}\prod_{j=1}^m\frac{{\rm G}(1+\alpha_j)^2}{{\rm G}(1+2\alpha_j)}\ (1+\oo(1)),
\end{equation}
where the Barnes function ${\rm G}$ is defined in Subsection \ref{subsec:notations}.
Motivations in statistical physics for general Fisher-Hartwig asymptotics are multiple, see in particular the beautiful exposition of applications to the phase transition of the 2d Ising model  in \cite{DeiItsKraII}.

Such Toeplitz determinant asymptotics are related to random matrix theory as they correspond to moments of characteristic polynomials of random matrices. For example, the Heine formula implies that the left-hand side of (\ref{eqn:Widom}) coincides with
$\E[\prod_{j=1}^m|P_N(z_j)|^{2\alpha_j}e^{{\rm Tr}V(U)}]$, where $P_N(z)=\det(z-U)$ and $U$ is a $N\times N$ Haar-distributed unitary matrix. 
The main contribution of our paper is the first Fisher-Hartwig asymptotics for singularities in space and time. More precisely,  Theorem \ref{thm:FH} below is a multi-time extension of (\ref{eqn:Widom}),  a formula due to Harold Widom in 1973.\\

To state this main result,  we first denote $\mathscr{A}$ (resp. $\mathscr{B}$) a finite subset of $\{z=t+\ii \theta,t\in\mathbb{R},0\leq \theta \leq 2\pi\}$ (resp.  $\mathbb{R}$), with fixed cardinality but  possibly $N$-dependent points.
The functions $f_s$ in the statement below are of regularity $\mathscr{C}^{3}$ on an arbitrary mesoscopic scale $N^{-1+\delta}$, $\delta\in(0,1]$.  We also remind the definition of the Poisson kernel ${\rm P}_t$ in (\ref{eqn:PoissonK}).

\begin{theorem}\label{thm:FH}
Let $(U_t)$ be a unitary Brownian motion at equilibrium,  as defined in (\ref{eqn:Unitbis}). 
Let $0<\delta\leq 1$, $C$ be fixed constants.  There exists $\varepsilon>0$ such that
uniformly in $\max_{\mathscr{B}}|s|+\max_{\mathscr{A}}|z|\leq C$,  $\min_{(z,z')\in\mathscr{A}^2,z\neq z'}|e^z-e^{z'}|>N^{-1+\delta}$, 
$\gamma_z\in[0,C]$,  $f_s\in\mathscr{S}_{\delta,C}$ (see Definition \ref{subsec:fctSpaces}), we have
\begin{multline}
\label{eq:asymptotics}
\E\Big[e^{\sum_{s \in \mathscr{B}} \Tr f_s(U_s)  } \prod_{z = t + \ii \theta\in\mathscr{A}} | \det (U_t - e^{\ii \theta}) |^{\gamma_z} \Big]= e^{N \sum_{\mathscr{B}}  \dashint f_s + \frac{1}{2} \sum_{\mathscr{B}^2} (f_s, {\rm P}_{|s-s'|} f_{s'})_{\rm H} - \sum_{z \in \mathscr{A}, s\in\mathscr{B}} \frac{\gamma_{z}}{2} ( {\rm P}_{|t-s|} - {\rm P}_{\infty}) f_s (e^{\ii \theta})} \\
\times \prod_{\mathscr{A}} N^{ \frac{\gamma_z^2}{4}} \frac{{\rm G}(1+\frac{\gamma_z}{2})^2}{{\rm G}(1+\gamma_z)} \prod_{z,w \in \mathscr{A}, z \neq w} \left( \frac{\max(|e^z|,|e^w|)}{|e^z-e^w|} \right)^{\frac{1}{4} \gamma_z \gamma_w} (1+{\rm O}( N^{-\eps}))
\end{multline}
where the multiplicative constant in $\OO$ depends on $|\mathscr{A}|$,  $|\mathscr{B}|$.
\end{theorem}

When there is no singularity ($\mathscr{A}=\varnothing$),  this formula is a dynamical generalization of the 
strong Szeg{\H o} theorem.  It can also be thought of as 
an upgrade to any mesoscopic scale and to exponential generating functions of Spohn's convergence of the Dyson Brownian motion dynamics to the free field (see section \ref{subsec:chardyn}).

However the main originality and applications of Theorem \ref{thm:FH} are due to the logarithmic insertions, see for example Remarks \ref{rem:logcor} and \ref{rem:max} on straightforward corollaries on logarithmically correlated fields, their maximum and optimal eigenvalues deviations along Dyson Brownian motion.
Based on (\ref{eq:asymptotics})  it is also not hard to obtain
that for any smooth space-time curve $\mathscr{C}$ in $(e^{\ii\theta},t)$ with Lebesgue measure $\lambda_{\mathscr{C}}$,
$| \det (U_t - e^{\ii \theta})|^\gamma \rd\lambda_{\mathscr{C}}$ converges up to normalization to a one dimensional Gaussian multiplicative chaos  in the ${\rm L}^1$ phase (i.e. $\gamma<2$ for $d=1$).  In particular this recovers the fixed time  results from \cite{Web2015,NikSakWeb20}.

The proof of Theorem \ref{thm:FH}  applies to other singularities: the discontinuities from $\im \log$. 
We only stated our results for  $\re\log$ for the sake of conciseness,  but one can easily state a consequence of the  discontinuous case ${\rm Im}\log$.
Indeed,  define $\im\log\det(1-e^{-\ii\theta}U_t)=\sum_k\im\log (1-e^{\ii(\theta_k(t)-\theta)})$, 
with the branch choice $\im \log (1-e^{\ii\varphi})=(\varphi-\pi)/2$ if $\varphi\in[0,\pi)$,  $(\varphi+\pi)/2$ if $\varphi\in(-\pi,0)$.  As
$\im\log\det(1-e^{-\ii\theta}U_t)-\im\log\det(1-U_t)=\pi({\rm N}_t(0,\theta)-\E {\rm N}_t(0,\theta))$,  where 
${\rm N}_t(0,\theta)=\big|\{\theta_k(t)\in(0,\theta]\}\big|$,
we have
\begin{equation}
\label{eq:LQGconvergence2}
\lim_{N \to \infty } Z_{N,\gamma}^{-1}\,  e^{\gamma\, \pi({\rm N}_t(0,\theta)-\E {\rm N}_t(0,\theta))}  \rd t \rd\theta = e^{\gamma h(z)} \rd z
\end{equation}
for every $\gamma \in (0,2\sqrt{2})$ and some constants $Z_{N,\gamma}$.  The above result relies on an analogue of Theorem \ref{thm:FH},  see Equation
(\ref{eq:asymptotics_withjump}) in the Appendix,  which also gives the necessary, straightforward changes for a proof.

Although an extension of Theorem \ref{thm:FH}  to include $\im\log$ and complex-valued $f_s$ is straightforward,  a generalization to 
complex-valued $\gamma_z$ is not.  In the static case,  the most general version of Fisher-Hartwig asymptotics \cite{DeiItsKra2011} allows general complex exponents,  with asymptotics involving a subtle variational problem.  It is not even clear how to formulate a related conjecture in our multi-time setting.\\

More generally,  moments of characteristic polynomials of wide classes of random matrices have been a topic of major interest, see e.g.
 \cite{BaiKea, BorStr,BumGam,FyoSup} to name a few in the case of integer exponents by algebraic and supersymmetric  methods, and \cite{BerWebWon2017,Fah2019,ChaFahWebWon,ClaFahLamWeb, WebWon} for fractional exponents by Riemann-Hilbert methods.  Theorem \ref{thm:FH} initiates joint (fractional) moments for correlated random matrices, a topic connected to the quenched complexity of high dimensional landscapes \cite{AufBenCer,Fyo2004}.

Our paper considers random matrices from the canonical setting, the unitary group, but we expect the convergence to LQG will remain in other settings (and the proof method through surgery as described below will apply,  although major technical obstacles remain). Such settings include dynamics on other Lie groups,  out of equilibrium or with a Dyson Brownian motion at arbitrary temperature.
In fact,  the upcoming work \cite{BouDubHarKel2025} on a non-Hermitian analogue of Fisher-Hartwig asymptotics will 
follow a scheme similar to the surgery that we now explain.

\subsection{Outline}.\ \label{outline}
To prove our main result, we develop a general surgery argument that allows us to go beyond the usual free field limit and which works very roughly speaking as follows: 1) we ``cut" the long range non-singular part of the determinants  in \eqref{eq:asymptotics} and prove a (space-time) decoupling of the resulting product of localized singularities 2) we carry out a general ``gluing operation" for non-singular terms 3) we evaluate asymptotics of \textit{one} localized singularity by gluing the opposite of the associated long range non-singular part to the determinant itself, together with the Selberg integral formula 4) with these in hands, it remains to glue back the non-singular parts and the additional smooth functions to the localized singularities.\\

{\noindent \it Decoupling.} The first ingredient consists in a space-time decoupling of the truncated singularities. Usual techniques to prove decorrelation for linear statistics or extrema of eigenvalues do not seem to work for the product of local singularities,  either because our functions are not in ${\rm H}^{1/2}$ or because such decouplings give additive error terms.  We find a new general multiplicative decorrelation of local linear statistics which can apply to a large class of determinantal point processes. We prove in Section \ref{subsec:kernel}, by using the Eynard-Mehta machinery,  that the process of the eigenvalues at different times is a determinantal point process. Despite the simplicity of the expression of the kernel we find, it seems that this stationary case has not been derived before (nor with arbitrary initial condition), because there is no canonical ordering as the particles are winding around the circle. As a second step, to work out the decoupling, the starting point of our proof is an infinite dimensional version of the Hoffman-Wielandt inequality, applied to a related self-adjoint operator, from which we then extract the sought decoupling of our observable. This is the content of Section \ref{subsec:decoupling}. \\

{\noindent \it Matrix dynamics.}
Our``gluing" operation starts with the usual method (initiated in random matrix theory in \cite{Joh1998}) of Hamiltonian perturbation  and then we  a) perform an integration by parts, b) obtain asymptotics.

As explained at the beginning of Section \ref{sec:loop}, due to our combined multitime and singular settings,
 step a) requires an original approach: 
the integration by parts formula from Proposition \ref{prop:Girsanov} encodes information about eigenvalues but also eigenvectors, while loop equations traditionally correspond to hierarchies only for particles/eigenvalues.
For the proof of Proposition \ref{prop:Girsanov}, we use the Girsanov theorem on the Lie algebra $\mathfrak{u}_N$ of the unitary group (the unitary Brownian motion $(U)$ is the solution of a matrix SDE driven by a Brownian motion $(B)$ on  $\mathfrak{u}_N$). This entails characterizing the Fr\'echet derivatives of the UBM, $D_F U_t := \lim_{\eps \to 0} \eps^{-1} ( U(B+\eps F)_t - U(B)_t )$ (shifting $B$ in a progressively measurable direction $(F) = \int_0^{\cdot} f_s ds$),  as solutions of matrix SDEs, and solving explicitly these. We exploit the stationarity of the process to consider long times so that observables of the UBM are well encoded by the noise driving the process and in particular by its associated integration by parts formula. 

 To control the error terms from step a),  we prove an averaged (over projections) and multi-time local law (Proposition \ref{prop:avIso}, the main result of Section 4), which is new including in the context of Hermitian random matrices. 
Moreover,  to control submiscroscopic errors due to logarithmic singularities,   in the key Lemma \ref{lem:loopEqn}, we impose an algebraic cancellation property (\ref{eqn:assumed}),  which holds if we add {\it compensator} functions to our observables of interest (see Section \ref{subsec:proofmain}).  These compensators are local functions,  so they can be included from the start in the decoupling step,  completing the outline of our surgery argument.\\

In Section \ref{sec:proofmain}, by applying the general surgery introduced above, we prove Theorem \ref{thm:FH} first, and then use it for our main application,  i.e.  the convergence to the Liouville quantum gravity measure.

\subsection*{Acknowledgements}. We wish to thank Jiaoyang Huang and Ofer Zeitouni for their useful feedback on the first version of the paper, and Xin Sun for suggesting the extension of our result to some other surfaces.  
We are especially grateful to Ahmet Keles for his many suggestions which helped improve this work.
P.B. was supported by the NSF grant DMS 2054851 and a Simons fellowship.

\section{Preliminaries}

\subsection*{Basic notations}.
\label{subsec:notations}
In this paper, $\rd\lambda$ denotes the Lebesgue measure on the unit circle $\mb{U}$, and $\rd m$ the Lebesgue measure on $\mathbb{C}$.  We remind that the 
Fourier coefficients of $f$ are defined as $\hat{f}_k = \frac{1}{2\pi} \int_0^{2\pi} e^{-\ii k \theta} f(\theta) d \theta$.
The Poisson kernel plays an important role and is normalized as follows:
\begin{equation}\label{eqn:PoissonK}
{\rm P}_t f(z)=\int_0^{2\pi}f(e^{\ii\theta}){\rm Re}\frac{1+ze^{-\ii\theta-t}}{1-ze^{-\ii\theta-t}}\frac{\rd\theta}{2\pi}
=
\sum_{k\in\mathbb{Z}}\hat f_k e^{-|k|t}z^k.
\end{equation}
Its restriction to $\mb{U}$ is given by ${\rm P}_t f(e^{\ii \theta}) = \sum_{k} \hat{f}_k e^{-|k| t} e^{\ii k \theta}$.

The Barnes G-function is defined as the Weierstrass product
$$
{\rm G}(z+1)=(2\pi)^{z/2}e^{-\frac{z+z^2(1+\gamma)}{2}}\prod_{k=1}^\infty \left(1+\frac{z}{k}\right)^ke^{\frac{z^2}{2k}-z}.
$$
Here, and only here, $\gamma$ is the Euler constant.
The Barnes function satisfies the functional equation ${\rm G} (z+1)=\Gamma(z) {\rm G}(z)$ where $\Gamma$ is the Gamma function.

Moreover,  for a matrix $A$, $\Tr(A) = \sum_i A_{i,i}$ and we denote by $A^T$ the transpose of $A$. $A^* = \overline{A^T}$. If $M, N$ are two complex valued matrices, $\langle M , N \rangle = \Tr ( \overline{M}^T N )$  and  $\langle M, N \rangle_{\mf{R}} = \re  \langle M, N \rangle$. 

Finally,   the statement of Theorem \ref{thm:FH}  makes use the following functional space $\mathscr{S}_{\delta,C}$ described below.

\begin{definition}\label{subsec:fctSpaces}
For $0<\kappa\leq 1$
and $k\in\mathbb{N}$ we introduce the norm on $\{f:\mathbb{U}\to\mathbb{R}\}$
$$
\|f\|_{\infty,k,\kappa}=\sum_{j=0}^k N^{j(\kappa-1)}\|f^{(j)}\|_\infty.
$$
We define ${\rm A}_{\kappa,C}$ as the set of functions $g:\mathbb{U}\to\mathbb{R}$ supported on an arc of length  $N^{-1+\kappa}$ and smooth on that scale in the sense that 
$\|g\|_{\infty,3,\kappa}\leq C$. 
For $0<\delta\leq 1$, let $\mathscr{S}_{\delta,C}$ be the set of functions $f:\mathbb{U}\to\mathbb{R}$ which can be written as
$$
f=\sum_{i=1}^m f_i,\ m\leq C\log N,  f_i\in {\rm A}_{\kappa,C}\ (\kappa\in[\delta,1]).$$
\end{definition}
Note that for any $g\in {\rm A}_{\kappa,C}$ we have $\hat{g}_k\leq C\min(N^{-1+\kappa},\frac{N^{1-\kappa}}{k^2})$, so $\|g\|_{\rm H}^2\leq C$. Thus, $f\in\mathscr{S}_{\delta,C}$ implies $\|f\|_{\rm H}=\OO(\log N)$.

Two examples of particular interest are as follows. First,  functions of type $f(e^{\ii\theta})=g(N^{1-\kappa}(\theta-\varphi))$ with $g$ compactly supported and $\mathscr{C}^3$ are in ${\rm A}_{\kappa,C}\subset \mathscr{S}_{\delta,C}$ for any $C>0$ and $\kappa\in[\delta,1]$.  Second,  any regularization of the function $f(e^{\ii\theta})=\log|e^{\ii\theta}-e^{\ii\varphi}|$ on scale $N^{-1+\delta}$ is in $\mathscr{S}_{\delta,C}$ for fixed, large enough $C$ (e.g.  $\theta\mapsto N^{1-\delta}\,\int f(e^{\ii\psi})\chi(N^{1-\delta}(\theta-\psi))\rd\psi$ with $\chi\geq 0$ smooth, compactly supported, $\int\chi=1$).

\subsection{Unitary Brownian motion}.  With its most common normalization,  the Brownian motion on the unitary group $U(N)$ satisfies the following stochastic differential equation (SDE)
\begin{equation}
\label{eq:UBM}
\rd \tilde{U}_t = \tilde{U}_t \rd B_t - \frac{1}{2} \tilde{U}_t \rd t
\end{equation}
where $dB_t$ is a Brownian motion on the space of skew Hermitian matrices. We consider an orthogonal basis of skew Hermitian matrices for $\langle \cdot, \cdot \rangle_{\mf{R}} $ given by matrices of the form 
$
\frac{1}{\sqrt{2N}} (E_{k,\ell}-E_{\ell,k}),  \frac{\ii}{\sqrt{2N}} (E_{k,\ell}+E_{\ell,k}),\frac{\ii}{ \sqrt{N}} E_{k,k}.
$  
Here, $E_{k, \ell}$ is the matrix whose $k,\ell$ entry is $1$ and other entries are $0$. Note that this is an orthonormal basis for $N \langle \cdot, \cdot \rangle_{\mf{R}} $. We write this basis $\{X_1, \dots, X_{N^2} \rbrace$. The Brownian motion $(B_t)$ can be realized as
\begin{equation}
\label{def:shbm}
B_t = \sum_k X_k \tilde{B}_t^k
\end{equation}
where the $(\tilde{B}^k)$'s are independent standard Brownian motions. 
It goes back to Dyson \cite{Dyson62} that the eigenvalues $\tilde z_k$ of $(\tilde{U}_t)$ satisfy
\begin{equation}
\label{eq:Dyson-2}
\rd \tilde z_k = \frac{1}{\sqrt{N}} \ii \tilde z_k \rd B_k - \frac{1}{N} \sum_{j \neq k} \frac{\tilde z_k \tilde z_j}{\tilde z_k - \tilde z_j} \rd t - \frac{1}{2} \tilde z_k \rd t.
\end{equation}

In this paper,  it will be more natural to consider a small time change in the unitary Brownian motion: the normalization
\begin{equation}
\label{eq:UBM-conv}
U_t := \tilde{U}_{2t},
\end{equation}
in other words  the dynamics
\begin{equation}\label{eqn:Unitbis}
\rd U_t = \sqrt{2}U_t \rd B_t -  U_t \rd t,
\end{equation}
will provide convergence to the free field on the cylinder with its canonical,  locally isotropic, covariance function $\E(h(z)h(w))= \pi (-\Delta_{\mc{C}})^{-1}(z,w)$,  as in Theorem \ref{thm:gmc-cv}. Moreover,  (\ref{eqn:Unitbis})
 corresponds to the normalization  in \cite{Spo1998},  the first result on convergence of dynamics of random matrix type to the free field, as explained in Subsection \ref{subsec:chardyn}.
Indeed Spohn considers the $\beta$-Dyson Brownian motion on the unit circle,  i.e.  the time evolution of $N$ particles on the unit circle  $\{e^{\ii \theta_1(t)}, \dots, e^{\ii \theta_N(t)} \}$ satisfying
\begin{equation}
\label{eqn:Dyson}
\rd \theta_j = \frac{\beta}{2N} \sum_{i \neq j} \cot \left(  \frac{\theta_j-\theta_i}{2}\right)  \rd t + \sqrt{\frac{2}{N}} \rd B_j(t)
\end{equation}
where the $(B_j)$'s are independent standard Brownian motions.  For the unitary Brownian motions strong solutions exist as \cite[Theorem 3.1]{CepLep2001} proves more generally that for $\beta \geq 1$, the particles almost surely do not collide but almost surely do when $\beta \in (0,1)$. 
With $z_k = e^{\ii \theta_k}$, the  dynamics  \eqref{eqn:Dyson} reads
\begin{equation}
\label{eqn:dynz}
\rd z_k=\ii z_k \sqrt{\frac{2}{N}}\rd B_k  - \frac{\beta}{N} \sum_{ j \neq k} \frac{z_k z_j }{z_k - z_j} \rd t + \frac{z_k}{N} (\frac{\beta}{2} - 1) \rd t - \frac{\beta}{2} z_k \rd t.
\end{equation}
By comparing (\ref{eq:Dyson-2}) and (\ref{eqn:dynz}),  the dynamics of the eigenvalues of the unitary Brownian motion as normalized in (\ref{eq:UBM-conv})
coincide with the $\beta$-Dyson Brownian motion from \cite{Spo1998} when $\beta=2$.

Finally,  we will use the It\^o formula for the considered dynamics (\ref{eqn:Unitbis}):
\begin{equation}
\label{eq:Ito-UBM}
\rd f(U_t) =\sqrt{2} \sum_k \mc{L}_{X_k} f(U_t) \rd \tilde{B}_t^k +  \Delta_{U(N)} f(U_t) \rd t,
\end{equation}
where $\mc{L}_X f(U) = \frac{d}{dt}_{|t=0} f( U e^{t X}) $ and $\Delta_{U(N)} f (U) = \sum_k \frac{d^2}{dt^2}_{|t=0} f(U e^{t X_k} )$ is the Laplacian on $U(N)$.

\subsection{The characteristic polynomial process and the free field}. \ \label{subsec:chardyn}
In the paragraphs below, starting from a formal application of Spohn's result \cite{Spo1998}, we explain how the large dimension limit of the logarithm of the characteristic polynomial process is naturally related with dynamics associated with the GFF. These explanations are not necessary for proving our theorems, but they shed some lights on the structure of the main objects we consider. We also use this as an opportunity to set some notations and record covariance identities that we will use, in particular when stating the convergence to the chaos measures $e^{\gamma h}$ and $e^{\gamma \bar{h}}$.

\smallskip

\noindent {\bf Characteristic polynomial process induced by the Dyson dynamics.} Given the dynamics (\ref{eqn:Dyson}),
Spohn \cite{Spo1998} considered the stochastic process (indexed by functions $f$) given by
$$
\xi_N(f,t) := \sum_{j=1}^N f(\theta_j(t)).
$$
As $\E \xi_N(f,t) = N \hat{f}_0 = N \dashint f$, it is natural to restrict to functions $f$ with zero mean and Spohn proved that the limiting dynamics are given, with $\Delta_{\mb{U}} = (\partial/\partial \theta)^2$,  by
\begin{equation}
\label{eq:spohn-dynamics}
\rd \xi (f,t) = \xi ( -(\beta/2) \sqrt{-  \Delta_{\mb{U}}} f ,t )\rd t+ \rd \mc{W}(f',t),
\end{equation}
where $d \mc{W}(f,t)$ is a Gaussian noise characterized by $\E( d\mc{W}(f,t) d\mc{W}(g,s)) = 2\delta(t-s) dt ds \frac{1}{2\pi} \int_0^{2\pi} f(x) g(x) dx$.

Now, we discuss the characteristic polynomial process induced by these dynamics, namely 
\begin{equation}
h_N(t,x) := \xi_N(f_x,t),
\end{equation}
where $f_x(\theta) := \log | e^{\ii \theta} - e^{\ii x} | =  - \re  \sum_{k\geq 1} \frac{1}{k} e^{\ii k \theta} e^{-ik x} = - \sum_{k \geq 1} \frac{1}{k} \cos(k(\theta-x))$. This field has zero mean in the sense that for every $N,t$, $\int_{\mb{U}} h_N(t,\cdot) = 0$. We formally take $f = f_x$ in \eqref{eq:spohn-dynamics} and look for the induced dynamics. Note first that $\sqrt{-(\partial/\partial \theta)^2} f_x(\theta) = \sqrt{-(\partial/\partial x)^2} f_x(\theta)$ so the drift is given by $- \frac{\beta}{2} (- \Delta_{\mb{U}})^{1/2}$. Concerning the noise part,  it is clearly white in time,  and when $t=s$ an elementary calculation gives
$$
\E (\mc{W}(f_x',t) \mc{W}(f_y',t) ) = 2 \frac{1}{2\pi} \int_0^{2\pi} f_x'(\theta) f_y'(\theta) d\theta  =\pi \delta(x-y).
$$
With $W$  an $L^2(\lambda)$ space-time white noise with zero mean (see below \eqref{eq:dynamics-gff-circles} for a representation with Brownian motions),  it is natural to expect from Spohn's result that
\begin{equation}
\label{eq:dyson-free-field}
\rd h_t = - \frac{\beta}{2} (- \Delta_{\mb{U}})^{1/2}  h_t \rd t+ \sqrt{\pi} W(\rd x,\rd t).
\end{equation}
Note also that $\int_{\mb{U}} h_t(x) \rd x = 0$ for every $t \in \mb{R}$ since $h_t(x) = \lim_N \sum \log | e^{\ii\theta_k^N(t)} - e^{\ii x}|$. 

\medskip

\noindent {\bf Dynamics of the averaged trace of the 2$d$ GFF on Euclidean circles.} We consider here the trace of the whole-plane GFF on Euclidean circles and explain in which sense the dynamics \eqref{eq:dyson-free-field} are related to it. The whole-plane GFF can be seen as a $\sigma$-finite measure (with Lebesgue measure on the zero mode) or as a random field modulo constant. Recalling that in the context of characteristic  polynomials $\int_{\mb{U}} h_N(t,\cdot) = 0$, we are here therefore only interested in $h_t = \Phi(e^{-t} \cdot ) - \dashint \Phi(e^{-t} \cdot )$, where $\Phi$ is a whole plane GFF, and this doesn't depend on the zero mode of the free field (so, for instance one can take $\Phi$ to have zero mean on $\mb{U}$ for which the covariance is given in \cite[Section 2.1.1]{Var2017}, for more on the GFF, see \cite{She2007, Dub2009}). From the $\log$-covariance of the whole-plane GFF, one has (see, e.g., \cite[Section 3] {DOZZ}),
\begin{equation}
\label{eq:cov-trace-field}
\E( h_s(e^{\ii x}) h_t(e^{\ii y}))) = \log \frac{\max(|e^{-s}|,|e^{-t}|)}{|e^{-s} e^{\ii x} - e^{-t} e^{\ii y} | }.
\end{equation}
In particular, $\E ( h_0(e^{\ii x}) h_0(e^{\ii y})) = - \log | e^{\ii x} - e^{\ii y} |$ and $h_0$ can be realized as $h_0 = \sum_k A_k(0)  \cos( k \cdot) + B_k(0) \sin( k \cdot)$ where $(A_k(0))$ and $(B_k(0))$ are independent Gaussian variables, with $A_k(0) \sim B_k(0) \sim \mc{N}(0, \frac{1}{k})$.  ${\rm H} = {\rm H}^{1/2}$ is exactly the Cameron-Martin space of $h_0$.

The Gaussian field given by \eqref{eq:cov-trace-field} has the same distribution as the one given by the following dynamics
\begin{equation}
\label{eq:dynamics-gff-circles}
\rd h_t = - (- \Delta_{\mb{U}})^{1/2} h_t\rd t + \sqrt{2 \pi} W(\rd t,\rd w),
\end{equation}
where $W$ is an $L^2( \lambda)$ space-time white noise on the unit circle and $h_0$ has the distribution of a centered Gaussian field with covariance given by $\E ( h_0(e^{\ii x}) h_0(e^{\ii y})) = - \log | e^{\ii x} - e^{\ii y} |$. The space-time white noise $W(\rd t, \rd w)$ can be realized as $\sum_{k \geq 1} \frac{\cos( k \cdot)}{\sqrt{\pi}} \rd V_k(t) + \frac{\sin( k \cdot)}{\sqrt{\pi}} \rd W_k(t)$ for some independent standard Brownian motions $(V_k)$, $(W_k)$. Therefore, with $h_t = \sum_k A_k(t)  \cos( k \cdot) + B_k(t) \sin( k \cdot)$, the above dynamics can be written as $\rd A_k(t) = - k A_k(t) \rd t + \sqrt{2}  \rd V_k(t)$ and, similarly, $\rd B_k(t) = - k B_k(t) \rd t + \sqrt{2} \rd W_t$. This is an infinite dimensional Ornstein-Uhlenbeck process and $A_k(t) = e^{-k t} A_k(0) + \sqrt{2} \int_0^t e^{-k(t-s)} \rd V_k(s)$ (similarly for $B_k$).

The identification in law of these two processes follows by a covariance calculation since both fields are Gaussian. Indeed, using the coordinates $z = t + \ii x$, $w=s+\ii y$ so $ \max(t,s) = \log \max(|e^z|,|e^w| )$, this follows from
$$
\sum_{k \geq 1} \frac{\cos(k(x-y))}{k} e^{- k|t-s|}  = - \log | 1- e^{-|t-s|} e^{\ii (y-x)} |    = \log \frac{\max( |e^z|,|e^w|) }{|e^z-e^w|}.
$$
Note that if $(h_t)$ solves \eqref{eq:dynamics-gff-circles}, $\tilde{h}_t = a h_{bt}$ solves $\rd \tilde{h}_t = - b (- \Delta_{\mb{U}})^{1/2}  \tilde{h}_t \rd t+ a \sqrt{b} \sqrt{2\pi} \widetilde{W}(\rd x,\rd t)$.

\eqref{eq:dyson-free-field} is natural from the point of view of the characteristic polynomial process. From the GFF point of view, the explicit form of \eqref{eq:dynamics-gff-circles} naturally arises from the Markov property of the free field. Indeed, instead of viewing $(h_t)$ as the trace of the free field on $e^{-t}\mb{U}$, it is equivalent to view it as the harmonic part of the Markov decomposition of $\Phi$ on $e^{-t} \mb{D}$, $h_t(z) = H h_{t | \mb{U}} (z)$ where $H$ denotes the harmonic extension. Then, writing $\Phi = h_0 + \phi_0$ on $\mb{D}$, where $\phi_0$ is an independent GFF with zero boundary values, it follows that
\begin{equation}
\label{eq:harmonic-decompo}
h_t(z) = h_0(e^{-t} z) + H_t(\phi_0)(e^{-t}z),
\end{equation}
where $H_t$ denotes the harmonic projection on $e^{- t} \mb{D}$. \eqref{eq:harmonic-decompo} readily implies that $(h_t)$ is a Markov process.  On the circle $w \in \mb{U}$, formally, $\frac{\rd}{\rd t}_{|t=0} h_0(e^{-t} w) = \frac{\rd}{\rd t}_{|t=0} H h_0(e^{-t} w) = \partial_n H h_0 $ where $\partial_n$ is the inward pointing normal derivative and $\partial_n H$ is the Dirichlet-to-Neumann operator, which here coincides with $- (- \Delta_{\mb{U}})^{1/2}$. This is a formal way for retrieving the drift part of \eqref{eq:dynamics-gff-circles}. In fact, from \eqref{eq:harmonic-decompo} and using the martingale problem approach, one can rigorously prove that the dynamics of $(h_t)$ are given by \eqref{eq:dynamics-gff-circles}. This approach is more robust and avoids having to guess the exact dynamics. For more  details, a generalization can be found in \cite{LQG-growth} which considers instead of Euclidean growth the metric growth associated with the LQG metric.

\medskip

\noindent {\bf Free field on the cylinder.}  When $\beta =2$,  the covariance of the limiting field associated with \eqref{eq:dyson-free-field} is 
\begin{equation}
\label{eq:cov-field}
\E (h(z) h(w)) = \frac{1}{2} \log \frac{ \max(|e^z|,|e^w| )}{|e^z-e^w|} = \frac{1}{2} \sum_{k \geq 1} \frac{\cos(k(x-y))}{k} e^{- k|t-s|} = {\rm P}_{|t-s|} C(x-y)
\end{equation}
where 
\begin{equation}
\label{def:cov-C}
C(x,y) = C(x-y) =  - \frac{1}{2} \log |e^{\ii x} - e^{\ii y}|.
\end{equation} 
This is an expression of the Green function associated with the Laplacian on $\mc{C} := \mb{R} \times \mb{U}$, $\Delta_{\mc{C}} =\partial_t^2 + \partial_\theta^2$.  Indeed, with $\hat{F}(\xi,k) := \frac{1}{2\pi}  \int_{\mb{R}} \int_{\mb{U}} F(t,x) e^{-\ii t \xi} e^{-\ii k x} \rd t \rd x$, we have $F(t,x) = \frac{1}{2\pi} \sum_{k \neq 0} \int_{\mb{R}} \hat{F}(\xi,k) e^{\ii k x} e^{\ii t \xi} \rd\xi$ so $-\Delta_{\mc{C}} F (t,x) =  \frac{1}{2\pi}  \sum_{k \neq 0} \int_{\mb{R}} (k^2+\xi^2) \hat{F}(\xi,k) e^{\ii k x} e^{\ii t \xi} \rd\xi$ and $(-\Delta_{\mc{C}})^{-1}$ has symbol given by $\frac{1}{k^2 + \xi^2}$. We retrieve the covariance kernel
$$
(-\Delta_{\mc{C}})^{-1} F(t,x) = \frac{1}{2\pi} \sum_{k \neq 0} \int_{\mb{R}} \hat{F}(\xi,k) \frac{e^{\ii kx} e^{\ii t \xi}}{k^2+\xi^2} \rd\xi = \int_{\mb{R} \times \mb{U}} F(s,y)  (-\Delta_{\mc{C}}^{-1})(s,x;t,y)\rd s\rd y
$$
where $(-\Delta_{\mc{C}})^{-1}(s,x;t,y)$ is given by
$$
\frac{1}{(2\pi)^2} \sum_{k \neq 0} \int_{\mb{R}} \frac{e^{\ii k(x-y)} e^{\ii \xi (t-s)}}{k^2+\xi^2} \rd\xi = \frac{1}{(2\pi)^2} \sum_{k \neq 0} \int_{\mb{R}} \frac{1}{k^2} \frac{e^{\ii k(x-y)} e^{\ii \xi (t-s)}}{1+(\xi/k)^2} \rd\xi =\frac{1}{(2\pi)^2} \sum_{k \neq 0} \int_{\mb{R}} \frac{1}{|k|} \frac{e^{\ii k(x-y)} e^{\ii \omega k (t-s)}}{1+\omega^2}  \rd\omega.
$$
By using $\frac{1}{\pi} \int_{\mb{R}} \frac{e^{\ii \omega x}}{1+\omega^2} \rd \omega = e^{-|x|}$, we get $\frac{1}{4\pi} \sum_{k \neq 0} \frac{1}{|k|} e^{ik(x-y)} e^{-|k| |t-s|} = \frac{1}{2\pi} \sum_{ k \geq 1} \frac{\cos(k(x-y))}{k} e^{-|k| |t-s|} $ hence  
\begin{equation}
\E(h(z)h(w))=\E(h(s,x) h(t,y)) = \pi (-\Delta_{\mc{C}})^{-1}(s,x;t,y).
\end{equation}

\subsection{Submicroscopic smoothing.}\ \label{sec:subsmooth}
In this section,  we explain that in our main result Theorem \ref{thm:FH}, without loss of generality we can assume that the logarithmic singularities from the determinants are smoothed on a submicroscopic scale.
More precisely, given a fixed small parameter $\alpha>0$, we define in the following logarithms smoothed on scale
\[
\rho=N^{-1-\alpha}.
\]
The regularization below  of the logarithm on scale $\rho$ is harmless as typically there are no particles below that scale, but it will help for the proof of the key Lemma \ref{lem:loopEqn}.

\begin{definition}\label{def:SmoothLog}
For any $z\in\partial\mathbb{D}$  consider the following functions  $\ell^z_+,\ell^z_-:\partial\mathbb{D}\to\mathbb{R}$.
\begin{itemize}
\item[--] $\ell^z_+(w)\geq \log |z-w|$,   $\ell_+^z(ze^{\ii\varphi})=\ell_+^z(ze^{-\ii\varphi})$, $\ell^z_+(w)=\log|z-w|$ when $|z-w|> 2\rho$,   $\ell^z_+(w)=\log\rho$ when $0<|z-w|<\rho$, and 
$|\partial_{\psi}^k\ell^{z}_+(e^{\ii\psi})|\leq C_k\min(\rho,|z-e^{\ii\psi}|)^{-k}$ for any $k\geq 1$.
\item[--] $\ell^z_-(w)=\int\chi_\rho(\varphi)\log|e^{\ii\varphi}w-z|\rd \varphi$ where $\chi\geq 0$ is smooth, even, supported on $[-1,1]$,  $\int\chi=1$, and
$\chi_\rho(x)=\rho^{-1}\chi(x/\rho)$. Note that $\ell_-$ also satisfies  $|\partial_{\psi}^k\ell^{z}_-(e^{\ii\psi})|\leq  C_k\min(\rho,|z-e^{\ii\psi}|)^{-k}$.
\end{itemize}
\end{definition}

The following lemmas show that Theorem \ref{thm:FH} only needs to be proved for $\ell_+,\ell_-$.

\begin{lemma}\label{lem:smoothing}
With the notations from Theorem \ref{thm:FH},  for some fixed $\tilde c>0$ we have
\begin{align*}
\E\Big[e^{\sum_{s \in \mathscr{B}} \Tr f_s(U_s)  } \prod_{z = t + \ii \theta\in\mathscr{A}} | \det (U_t - e^{\ii \theta}) |^{\gamma_z} \Big]\leq\, & \E\Big[e^{\sum_{s \in \mathscr{B}} \Tr f_s(U_s)  +\sum_{z = t + \ii \theta\in\mathscr{A}}\gamma_z\Tr\ell^{e^{\ii\theta}}_+(U_t)}\Big],\\
\E\Big[e^{\sum_{s \in \mathscr{B}} \Tr f_s(U_s)  } \prod_{z = t + \ii \theta\in\mathscr{A}} | \det (U_t - e^{\ii \theta}) |^{\gamma_z} \Big]\geq\, &\E\Big[e^{\sum_{s \in \mathscr{B}} \Tr f_s(U_s)  +\sum_{z = t + \ii \theta\in\mathscr{A}}\gamma_z\Tr\ell^{e^{\ii\theta}}_-(U_t)}\Big]\cdot (1-N^{-\alpha/2})\\
&-e^{N \sum_{\mathscr{B}}  \dashint f_s -\tilde c(\log N)^2}.
\end{align*}
\end{lemma}

\begin{proof}
The first inequality is trivial because $\ell_+^{e^{\ii\theta}}(e^{\ii\psi})\geq \log|e^{\ii\theta}-e^{\ii\psi}|$ for any real $\theta,\psi$.

The second inequality relies on invariance by rotation and Jensen's inequality: Denoting  $X$ a random variable with density $\chi_\rho$, we have
\begin{align*}
\E\Big[e^{\sum_{s \in \mathscr{B}} \Tr f_s(U_s)  +\sum_{z = t + \ii \theta\in\mathscr{A}}\gamma_z\Tr\log|U_t-e^{\ii\theta}|}\Big]&=\int\E\Big[e^{\sum_{s \in \mathscr{B}} \Tr f_s(e^{\ii\varphi}U_s)  +\sum_{z = t + \ii \theta\in\mathscr{A}}\gamma_z\Tr\log|e^{\ii\varphi}U_t-e^{\ii\theta}|}\Big]\chi_\rho(\varphi)\rd\varphi\\
&=\E\E_{X}\Big[e^{\sum_{s \in \mathscr{B}} \Tr f_s(e^{\ii X}U_s)  +\sum_{z = t + \ii \theta\in\mathscr{A}}\gamma_z\Tr\log|e^{\ii X}U_t-e^{\ii\theta}|}\Big]\\
&\geq \E\Big[e^{\sum_{s \in \mathscr{B}} \Tr \E_Xf_s(e^{\ii X}U_s)  +\sum_{z = t + \ii \theta\in\mathscr{A}}\gamma_z\E_X\Tr\log|e^{\ii X}U_t-e^{\ii\theta}|}\Big]\\
&= \E\Big[e^{\sum_{s \in \mathscr{B}} \Tr \tilde f_s(U_s)  +\sum_{z = t + \ii \theta\in\mathscr{A}}\gamma_z\Tr\ell^{e^{\ii\theta}}_-(U_t)}\Big],
\end{align*}
where
\begin{equation}\label{eqn:tildef}
\tilde f_s(w)=\int\chi_\rho(\varphi)f_s(we^{\ii\varphi})\rd\varphi.
\end{equation}
After ordering $0\leq \theta_2(s)-\theta_1(s)\leq\dots \leq \theta_N(s)-\theta_1(s)\leq 2\pi$, we now consider the rigidity event 
\[
\mathscr{R}=\bigcap_{s\in\mathscr{B},1\leq i<j\leq N}\{|\theta_j(s)-\theta_i(s)-\frac{2\pi(j-i)}{N}|\leq \frac{N^{c}}{N}\}
\]
where $c$ can be any constant chosen in $(0,\delta)$ and $\delta$ is the regularity scale from the assumptions of Theorem \ref{thm:FH}.
From (\ref{eqn:rigidity}) and a union bound we know that $\mathbb{P}(\mathscr{R}^{\rm c})\leq e^{-c(\log N)^2}$. Together with the Cauchy-Schwarz inequality,  Lemma \ref{lem:regularity} below and Lemma \ref{lem:Johansson}, this implies
\[
\E\Big[e^{\sum_{s \in \mathscr{B}} \Tr \tilde f_s(U_s)  +\sum_{z = t + \ii \theta\in\mathscr{A}}\gamma_z\Tr\ell^{e^{\ii\theta}}_-(U_t)}\mathds{1}_{\mathscr{R}^{\rm c}}\Big]\leq 
e^{N \sum_{\mathscr{B}}  \dashint f_s -\tilde c(\log N)^2}.
\]
From Definiton \ref{subsec:fctSpaces},  $\|f_s'\|_\infty\leq (\log N) N^{1-\delta}$, and on $\mathscr{R}$ we have $\#\{\theta_i\in{\rm supp} f_s\}\lesssim \log N$,  so that on $\mathscr{R}$ we have 
\[
|{\rm Tr}\tilde f_s(U_s)-{\rm Tr} f_s(U_s)|\leq N^\delta\cdot (\log N)^2N^{1-\delta} N^{-1-\alpha}\leq N^{-\alpha/2}.
\]
The conclusion immediately follows.
\end{proof}

Next, we establish the reverse inequalities to recover the log-singularities from their smooth approximations.  For our purpose, it suffices to prove these reverse inequalities in the single-time, single log-singularity case. The argument relies on a simple asymptotic result for the Hua–Pickrell kernel, $K^{{\rm HP}(\gamma)}$, which is the correlation kernel of the determinantal point process associated with the biased measure on the particles defined by
\begin{align*} 
\E^{{\rm HP}(\gamma)}[f(U)]:=\E\Big[f(U)\frac{|\det(U-\Id)|^{\gamma}}{\E[|\det(U-\Id)|^{\gamma}]}\Big]
\end{align*}
where $\gamma\geq 0$. The following key observation about this kernel follows directly from Theorem 3.6 and Proposition 3.7 in \cite{bourgade2009random}.
\begin{lemma}\label{lemma:local_HP_kernel}
Fix some $C>0$ and let $\gamma\in[0,C]$ be a constant. Then, there exists a $c>0$ depending only on $C$ such that for every $N\in\N$ and $\theta\in[-\frac{1}{N},\frac{1}{N}]$,
\begin{align*} 
|K^{{\rm HP}(\gamma)}(\theta,\theta)|\leq c |\theta|^{\gamma}N^{1+\gamma}.
\end{align*}
As a corollary, by Hadamard's inequality,
\begin{align}\label{eqn:HP_corr}
\rho_k^{{\rm HP}(\gamma)}(\theta_1,\dots,\theta_k)\leq \prod_{i=1}^{k} K^{{\rm HP}(\gamma)}(\theta_i,\theta_i)\leq c^k|\theta_1\theta_2\cdots \theta_k|^{\gamma} N^{k(1+\gamma)}
\end{align}
where $\rho_{k}^{{\rm HP}(\gamma)}$ is the $k$-point correlation function.
\end{lemma}

\begin{lemma}\label{lem:single_smoothing} Given $\alpha,C>0$ and $\rho=N^{-1-\alpha}$, for any $\gamma\in[0,C]$ we have,
\begin{align} \label{eqn:log_delta_reg}
&\E[e^{\gamma\Tr\ell_{+}^{1}(U)}]\leq \E[e^{\gamma\Tr\log|U-\Id|}](1+\OO(N^{-\alpha})),\\
\label{eqn:log_conv_delta_reg}
&\E[e^{\gamma\Tr\ell_{-}^{1}(U)}]\geq \E[e^{\gamma\Tr\log|U-\Id|}](1+\OO(N^{-\alpha/2})),
\end{align}
where the implicit constants depend only on $C$.
\end{lemma}
\begin{proof} We omit the superscript $1$ in $\ell_{+}$ and $\ell_-$ for simplicity and define a counting set $$A_k:=\{\textnormal{number of eigenangles in $[-3\rho,3\rho]$ is exactly $k$}\}.$$
Then
\begin{align*}
\E[e^{\gamma\Tr\ell_{+}(U)}]&=\E[\mathds{1}_{A_0}\cdot e^{\gamma\Tr\log|U-\Id|}]+\sum_{k=1}^{N}\E[\mathds{1}_{A_k}\cdot e^{\gamma\Tr\ell_{+}(U)}]
\\
&\leq \E[e^{\gamma\Tr\log|U-\Id|}]\Big(1+\sum_{k=1}^{N}\E^{{\rm HP}(\gamma)}[\mathds{1}_{A_k}\cdot e^{\gamma\Tr(\ell_{+}-\log|1-\cdot|)(U)}] \Big).
\end{align*}
For any $k\geq 1$,
\begin{align*} 
\E^{{\rm HP}(\gamma)}[\mathds{1}_{A_k}\cdot e^{\gamma\Tr(\ell_{+}-\log|1-\cdot|)(U)}] &\leq  \mathbb{E}^{{\rm HP}(\gamma)}\Big[\sum_{\substack{i_1,\dots,i_k\in[\![N]\!]\\\textnormal{and distict}}} \delta_{(\theta_{i_1},\dots,\theta_{i_k})}([-3\rho,3\rho]^k) \cdot \prod_{j=1}^{k}(\frac{4\rho}{|\theta_{i_j}|})^{\gamma} \Big]]
\\
&\leq \int_{[-3\rho,3\rho]^k} \prod_{j=1}^{k}(\frac{4\rho}{|\theta_{j}|})^{\gamma}c^k|\theta_1\theta_2\cdots \theta_k|^{\gamma} N^{k(1+\gamma)}\rd\theta_{1}\dots\rd\theta_{k}
\leq (\frac{4c}{N^{\alpha}})^{k+k\gamma}
\end{align*}
where we have used equation \eqref{eqn:HP_corr}. Substituting this gives the necessary bound.

Moving onto the second inequality, \eqref{eqn:log_conv_delta_reg}, we introduce an intermediate function in order to evaluate long range and short range separately:
$$\ell:=(1-\chi_{\varepsilon})\cdot\log|1-\cdot\,|+\chi_{\varepsilon}\cdot\ell_{-}$$
where $\varepsilon:=N^{-1-\alpha/2}$ and $\chi_{\epsilon}(z)$ is a smooth bump function on unit circle that is $1$ if $|z-1|<\epsilon$, and $0$ if $|z-1|>2\epsilon$. Short range can be handled similarly to the proof of the inequality \eqref{eqn:log_delta_reg}. We define $$B_k:=\{\textnormal{number of eigenangles in $[-3\varepsilon,3\varepsilon]$ is exactly $k$}\}$$ and let $\ell_{+,\varepsilon}$ be defined analogously to $\ell_{+}$, except with $\varepsilon$ is used as the scaling parameter in place of $\rho$. Then,
\begin{align} 
\E[e^{\gamma\Tr\ell(U)}]&=\E[e^{\gamma\Tr\log|U-\Id|}]+\sum_{k=1}^{\infty}\Big(\E[\mathds{1}_{B_k}e^{\gamma\Tr\ell(U)}]-\E[\mathds{1}_{B_k}e^{\gamma\Tr\log|U-\Id|}]\Big) \nonumber
\\
&=\E[e^{\gamma\Tr\log|U-\Id|}]+O\Big(\sum_{k=1}^{\infty}\E[\mathds{1}_{B_k}e^{\gamma\Tr\ell_{+,\epsilon}(U)}]\Big)=\E[e^{\gamma\Tr\log|U-\Id|}](1+\OO(N^{-\alpha/2})) \label{eqn:ell_to_log}
\end{align}
where in the second inequality we used $\ell_{+,\epsilon}\geq\ell$ and $\ell_{+,\epsilon}\geq\log|1-\cdot\hspace{.07cm}|$ and the last equality is as shown during the proof of equation \eqref{eqn:log_delta_reg}.

Finally, we finish the proof by showing $\E[e^{\gamma\Tr\ell_{-}(U)}]=\E[e^{\gamma\Tr\ell(U)}](1+\OO(N^{-\alpha/2}))$. Consider the rigidity set, i.e. when the eigenangles are ordered, $0\leq\theta_1\leq\dots\leq \theta_N\leq 2\pi$, and 
$$\widetilde{\mathscr{G}}=\bigcap_{1\leq k\leq N}\left\{|\theta_k-\frac{2\pi k}{N}|\leq \frac{\varphi}{N}\right\}$$
where $\varphi=(\log N)^{\log\log N}$ for which we have $\mathbb{P}(\widetilde{\mathscr{G}}^c)\leq e^{-(\log N)^{\log \log N}}$ (see, e.g., equation (1.8) in \cite{lambert2021mesoscopic}). Simple calculations show that on $\widetilde{\mathscr{G}}$, $\Tr(\ell-\ell_{-})(U)\leq \varphi^2 N^{-\alpha}$. Hence,
\begin{align*} 
\E[e^{\gamma\Tr\ell_{-}(U)}]&=\E[\mathds{1}_{\widetilde{\mathscr{G}}}\cdot e^{\gamma\Tr\ell(U)}](1+\OO(N^{-\alpha/2}))+\OO\Big(\Prob(\widetilde{\mathscr{G}}^c)^{1/2}\E[e^{2\gamma\Tr\ell_{-}(U)}]^{1/2}\Big)
\\
&=\E[e^{\gamma\Tr\ell(U)}](1+\OO(N^{-\alpha/2}))+\Prob(\widetilde{\mathscr{G}}^c)^{1/2}\OO\Big(\E[e^{2\gamma\Tr\ell_{+,\epsilon}(U)}]^{1/2}\Big)
\\
&=\E[e^{\gamma\Tr\log|U-\Id|}](1+\OO(N^{-\alpha/2}))
\end{align*}
where we have used $\ell_{+,\epsilon}\geq\ell$ and $\ell_{+,\epsilon}\geq\ell_{-}$ in the second equality and equations \eqref{eqn:ell_to_log}, \eqref{eqn:log_delta_reg} with Lemma \ref{lem:one-sing} in the third.
\end{proof}

\begin{lemma}\label{lem:regularity}
Let 
$\ell^{\ii\theta}_+,\ell^{e^{\ii\theta}}_-$ be
as in Definition \ref{def:SmoothLog}. Then there exists $\tilde C=\tilde C(\alpha)$ such that
\[
\|\ell^{e^{\ii\theta}}_\pm\|_{{\rm H}}^2\leq \tilde C\log N.
\]
\end{lemma}
\begin{proof}
%
We prove the result for $\ell_+$ (the proof for $\ell_-$ is identical). Without loss of generality we consider $\ell(\varphi):=\ell_+^1(e^{\ii\varphi})$.
For $k\geq \rho^{-1}$ we have
$
|\hat \ell_k|\lesssim k^{-2}\int|\ell''|\lesssim k^{-2}\rho^{-1},
$
so that $\sum_{|k|\geq \rho^{-1}}|k|\cdot|\hat\ell_k|^2\lesssim 1$.

For $k\leq \rho^{-1}$,  we note that
$
\hat \ell_k=\int_{-\pi}^\pi\ell(\varphi)e^{\ii k\varphi}\rd\varphi=k^{-1}\int_0^\pi\ell'(\varphi)(e^{\ii k\varphi}-e^{-\ii k\varphi})\rd\varphi,
$
and
\begin{multline*}
|\int_0^\pi\ell'(\varphi)(e^{\ii k\varphi}-e^{-\ii k\varphi})\rd\varphi|\lesssim \int_0^{1/k}\min(\rho^{-1},\varphi^{-1})k\varphi\rd\varphi+|\int_{1/k}^\pi\ell'(\varphi)\sin(k\varphi)\rd\varphi|\\
\lesssim 1+k^{-1}|\int_{1/k}^\pi\ell''(\varphi)\cos(k\varphi)\rd\varphi|+k^{-1}\ell'(1/k)\lesssim 1.
\end{multline*}
We have proved  $\hat \ell_k\lesssim k^{-1}$ for $k\leq \rho^{-1}$, so $\sum_{k\leq \rho^{-1}}k|\hat \ell_k^2|\lesssim \log N$ and the proof is complete.
\end{proof}

\section{Multi-time determinantal point process}

\label{sec:DPP}

$1d$ Markov processes such as random walks or diffusions conditioned not to intersect arise in many statistical mechanics models. In the continuous setting, the Karlin-McGregor formula \cite{KarMcGreg1959} allows to understand the probability distribution of these non-intersecting paths by viewing them as measures defined by products of several determinants. The Eynard-Mehta theorem states that these are determinantal point processes (point processes for which the correlation functions can be expressed as determinants of an associated kernel), a large class that appears in random matrix theory, growth processes, directed polymers, tilings and combinatorics, to name a few. Nice introductions and more background can be found in \cite{Joh2006, Bor2011} and references therein.

\subsection{The extended kernel}.
\label{subsec:kernel}
Motivated by universality associated with nonequilibrium eigenvalue statistics, Pandey and Shukla  \cite{PanShu1991} studied in 1991 the Dyson dynamics with $\beta =2$ started from two initial conditions, COE and CSE, and expressed their correlation functions as determinants.  Below, we show that when started from equilibrium, namely CUE initial condition, the associated process is a determinantal point process and provide an expression of its kernel. We have a modern treatment, using the Eynard-Mehta theorem and we then discuss the case of arbitrary initial conditions. As a comparison, the stationary GUE case where the Brownian motions are on the real line instead of the circle can be found in \cite{Joh2005} (see, e.g., Equation (2.12)). Here, some extra care is needed, one of the reasons being that there is no canonical ordering of the particles since they are winding around the unit circle.

\begin{proposition}
\label{Prop:DPP-kernel}
The eigenvalues of the unitary Brownian motion $(U_t)_{t \geq 0}$ from 
(\ref{eqn:Unitbis}),  started from the Haar measure $\{ (j, e^{\ii \theta_k(t_j)})_{1 \leq k \leq N} ~ : ~ 1 \leq j \leq J \}$ form a determinantal point process with kernel given by
\begin{align}
\label{eq:kernel-diag}
K(i,x;j,y)  = & \frac{\mathds{1}_{i \leq j}}{2\pi}  e^{- (\frac{N-1}{2})^2  \frac{|t_i-t_j|}{N}} \sum_{1\leq k \leq N}  e^{(k-\frac{N+1}{2})^2  \frac{|t_i-t_j|}{N}} e^{\ii (x-y) (k-\frac{N+1}{2})}  \\
&  - \frac{\mathds{1}_{i > j}}{2\pi} e^{ (\frac{N-1}{2})^2  \frac{|t_i-t_j|}{N}} \sum_{k \in [1,N]^c}  e^{-(k-\frac{N+1}{2})^2  \frac{|t_i-t_j|}{N}} e^{\ii (x-y) (k-\frac{N+1}{2})},
\end{align}
where $x,y$ here are angles. Namely, for any bounded and measurable function $g : \llbracket 1,J\rrbracket  \times \mb{U} \to \mb{R}$, we have 
$$
 \E \left( \prod_{j=1}^{J} \prod_{i=1}^N  (1+g(j,z_i(t_j))    \right) = \sum_{k=0}^{\infty} \frac{1}{k!} \int_{( \llbracket 1,J\rrbracket \times [0,2\pi])^k} \left( \prod_{j=1}^k  g(j,x_j)\right) \det ( K((r_i,x_i);(r_j,x_j))_{i,j=1}^k  \lambda(dx) \#(dr).
$$
\end{proposition}

\textit{Sketch of the proof.} First, using \cite{Hobson-Werner}, we give an expression of the transition probability of Brownian motions on the circle conditioned on non-intersecting for all time. Then, using an argument from \cite{AriOCo2019}, we rewrite it as a product of determinants in order to apply the Eynard-Mehta theorem and compute thereby an expression of the associated extended kernel.

\begin{proof}
To lighten the notations, we prove it for $J=2$, the generalization to any fixed $J$ is straightforward. First, we need a result by Hobson and Werner \cite{Hobson-Werner}. In this paper, the authors consider Brownian motions on the circle killed when intersecting.  Conditioning on non-intersecting (for all times) corresponds to considering the dynamics
$$
\rd \theta_j = \frac{1}{2} \sum_{i \neq j} \cot \left(  \frac{\theta_j-\theta_i}{2}\right) \rd t +  \rd B_j(t),
$$
see (4.1) in their paper, where the $\theta_j$'s are the angles and the $B_j$'s are standard Brownian motions. This is the time change $t \to \frac{t N}{2}$ and $\beta =2$ in \eqref{eqn:Dyson} (so we will eventually take $t \to \frac{2t}{N}$ in our formula).

Let $A_{s,t}$ be the event that trajectories do not intersect between times $s$ and $t$, and $\mathbb{P}$ the distribution if independent BMs on the torus. With the notations from \cite{Hobson-Werner}, the transition probability $q_t$ of $Z$  ((4.1) in \cite{Hobson-Werner}) is
\begin{multline}
\label{eq:transition}
q_t(x,y) =\lim_{T\to\infty}\mathbb{P}((x,0)\to(y,t)\mid A_{0,T})  =\lim_{T\to\infty}\frac{\mathbb{P}((x,0)\to(y,t), A_{0,t})\mathbb{P}_y(A_{0,T-t})}{\mathbb{P}_x(A_{0,T})}  \\
=\lim_{T\to\infty}\frac{\mathbb{P}_y(A_{0,T-t})}{\mathbb{P}_x(A_{0,T})}  q_t^*(x,y) =e^{\lambda_N t}\frac{|(y)|}{|\Delta(x)|}q_t^*(x,y), 
\end{multline}
where we used the notation  $\Delta(x)=\prod_{k < \ell}(e^{\ii x_\ell}-e^{\ii x_k})$ and the result from \cite{Hobson-Werner}:
$$
\mathbb{P}_x(A_{0,T})\underset{T\to\infty}{\sim}c_N e^{-\lambda_N T}|\Delta(x)|, \qquad \lambda_N:=\frac{N(N-1)(N+1)}{24}.
$$

Here, $q_t^*$ denotes the transition density of $N$ Brownian motions on the circle killed when any two of them collide. \cite{Hobson-Werner} gives an expression of this term and we borrow an argument by Arista and O'Connell \cite[Section 5.1]{AriOCo2019} to rewrite it. When $x, y$ belong to the set  $\{ z_1 < \dots < z_N < z_1+ 2\pi \} \cap  \{ z_1 \in [-\pi, \pi) \}$,
$$
q_t^*( e^{\ii x}, e^{\ii y} ) = \frac{1}{N} \sum_{u=0}^{N-1} \det \left( \sum_{k \in \Z} \eta^{uk} p_t(x_i, y_j+2\pi k) \right)  
$$
where $\eta = e^{\ii \frac{2\pi}{N}}$. With $\nu_{[\ell]}$ the representative of $\nu$ shifted by $\ell$ in $\{ z_1 < \dots < z_N < z_1 + 2\pi \} \cap  \{z_1 \in [-\pi, \pi) \}$ (i.e., $z_i \to z_{i+\ell \mod N}$), it was remarked in  \cite{AriOCo2019}  that
\begin{equation}
\label{eq:AriOCo2019}
\sum_{\ell = 0}^{N-1} q_t^*(e^{\ii x}, e^{\ii y_{[\ell]}}) = \det \left( \sum_{k \in \mb{Z}}  (-1)^{k(N-1)} p_t( x_i, y_j + 2k\pi) \right).
\end{equation}

The point process induced by the (ordered) vector $Z_t = \{ e^{\ii \theta_1(t)}, \dots, e^{\ii \theta_N(t)} \}$ is associated with counting functions $M_U(Z_t)$ where $U$ is an open subset of $\mb{U}$. Using that $F(M_U(w))$ is invariant under permutation and that the application $y \mapsto y_{[\ell]}$ is measure preserving, we have
$$
\int_{\mathrm{ordered}} q_t(x,y_{[\ell]}) F(M_U(y)) \rd y = \int_{\mathrm{ordered}} q_t(x,y_{[\ell]}) F(M_U(y_{[\ell]}))) \rd y =  \int_{\mathrm{ordered}} q_t(x,y) F(M_U(y)) \rd y 
$$
where ``ordered" $= \{w = (e^{\ii z_i})_{1 \leq i \leq N} ~ : ~ z_1 < \dots < z_N < z_1 + 2\pi, z_1 \in [-\pi, \pi) \}$. So, by using that  $| \Delta(y_{[\ell]}) | = | \Delta(y) |$ and combining \eqref{eq:transition} and \eqref{eq:AriOCo2019}, we have
$$
\E_x(F(M_U(Z_t))) = \int_{\mathrm{ordered}}  \frac{1}{N} \sum_{\ell=0}^{N-1} q_t(x,y_{[\ell]}) F(M_U(y)) \rd y =  \int_{\mathrm{ordered}}  w^N_{t,x}(y) F(M_U(y)) \rd y 
$$
where
\begin{equation}
\label{eq:weight-conditioned}
w^N_{t,x}(y) := \frac{e^{\lambda_N t}}{N} \frac{|\Delta(y)|}{|\Delta(x)|}  \det \left( \sum_{k \in \mb{Z}} (-1)^{k(N-1)} p_t( x_i, y_j + 2k\pi) \right).
\end{equation}
Note that when $y_1 < \dots < y_N < y_1 + 2\pi$, for $k < \ell$, $|e^{\ii y_{\ell}} - e^{\ii y_k} | = 2 | \sin ( \frac{y_\ell-y_k}{2} ) |  = 2 \sin ( \frac{y_\ell-y_k}{2} )$ since $y_{\ell} - y_k \in (0, 2\pi)$ hence
$$
|\Delta(y)| =\prod_{
k < \ell} 2 (e^{\ii y_\ell}-e^{\ii y_k})\cdot \frac{e^{-\ii\frac{y_k+y_\ell}{2}}}{2\ii} = \ii^{-\frac{N(N-1)}{2}} \Delta(e^{\ii y_1},\dots,e^{\ii y_N}) e^{-\ii\frac{N-1}{2}\sum y_k}  =  \ii^{-\frac{N(N-1)}{2}} \det(e^{\ii y_i(j-1-\frac{N-1}{2})}).
$$
So, \eqref{eq:weight-conditioned} is invariant under permutation of the $x_i$'s and under permutation of the $y_i$'s.

Starting from the Haar measure, for symmetric functions $F$ and $G$, we have
\begin{align*}
\E( F(Z_0) G(Z_t)) & \propto \int_{\mb{U}^N} F(x) \E_x (G(Z_t)) | \Delta(x) |^2 \rd x  \\
& \propto \int_{\text{ordered}^2} F(x) G(y) w_{t,x}^N(y) | \Delta(x) |^2 \rd x \rd y.
\end{align*}
Furthermore, $|\Delta (x) |^2 = \prod_{k < \ell} (e^{\ii x_\ell}-e^{\ii x_k}) \prod_{k < \ell} (e^{-\ii x_\ell}-e^{-\ii x_k}) = \Delta(e^{\ii x_1},\dots,e^{\ii x_N}) \Delta(e^{-\ii x_1},\dots,e^{- \ii x_N})$, so
$$
\frac{|\Delta(y)| }{|\Delta(x) |} |\Delta (x) |^2  =  \det(e^{\ii y_i(j-\frac{N+1}{2})})  \Delta(e^{-\ii x_1},\dots,e^{- \ii x_N})  e^{+\ii\frac{N-1}{2}\sum x_k} = \det(e^{-\ii x_i(j-\frac{N+1}{2})})  \det(e^{\ii y_i(j-\frac{N+1}{2})}) 
$$
and the joint density is proportional to 
$$
\det(e^{-\ii x_i(j-\frac{N+1}{2})}) \det \left( \sum_{k \in \mb{Z}} (-1)^{k(N-1)} p_t( x_i, y_j + 2k\pi) \right)  \det(e^{\ii y_i(j-\frac{N+1}{2})}).
$$
We conclude that the weight function associated to our random point process is of the form a product of several determinants. By the Eynard-Mehta theorem \cite{Eynard-Mehta} (see \cite[Section 2]{Joh2006} or \cite[Theorem 4.2]{Bor2011}), this is a determinantal point process, with kernel given by
\begin{align*}
K(0,x;0,y) & = \sum_{1 \leq i,j \leq N} ((G^{-1})^T)_{i,j} \Phi_i(x) \int_{\mb{U}} T(y,z) \Psi_j(z) \lambda(\rd z)  \\
K(0,x;1,y) & =  \sum_{1 \leq i,j \leq N} ((G^{-1})^T)_{i,j} \Phi_i(x) \Psi_j(y) \\
K(1,x;0,y) & = -  T(y,x) + \sum_{1 \leq i,j \leq N} ((G^{-1})^T)_{i,j} \int_{\mb{U}}  \Phi_i(z) T(z,x) \lambda(\rd z) \int_{\mb{U}}  T(y,z) \Psi_j(z) \lambda(\rd z) \\
K(1,x;1,y) & = \sum_{1 \leq i,j \leq N} ((G^{-1})^T)_{i,j} \int_{\mb{U}}  \Phi_i(z) T(z,x)  \lambda(\rd z) \Psi_j(y)
\end{align*}
where $G_{i,j} = \int_{\mb{U}^2} \Phi_i(x)  T(x,y) \Psi_j(y) \lambda(\rd x) \lambda(\rd y)$ and
\begin{equation}
\label{def:ort-pol-trans}
\Phi_i(x) = e^{- \ii x (i-\frac{N+1}{2})}, \qquad T(x,y) =  \sum_{k \in \mb{Z}} (-1)^{k(N-1)} p_t(x,y+2k\pi), \qquad \Psi_j(x) = e^{\ii x(j- \frac{N+1}{2})}.
\end{equation}
By taking $\delta = - \frac{N+1}{2}$, $x \to x-y$ in the summation formula of Lemma \eqref{lem:Poisson-Summation-Formula}, we have 
\begin{equation}
\label{eq:T-Poisson-sum}
T(x,y) = \frac{1}{2\pi} \sum_{n \in \mb{Z}} e^{\ii (x-y) (n- \frac{N+1}{2}) } e^{- \frac{1}{2}(n- \frac{N+1}{2})^2 t}  = \frac{1}{2\pi} \sum_{n \in \mb{Z}} u_n \Psi_n(x) \Phi_n(y) 
\end{equation}
where $u_n = e^{- \frac{1}{2}(n- \frac{N+1}{2})^2 t} $. This and $\int_{\mb{U}} \Phi_i(x) \Psi_j(x) \lambda(dx) = 2\pi \delta_0(j-i)$ imply \begin{equation}
\label{eq:integ}
\int_{\mb{U}} T(x,y) \Psi_j(y) \lambda(\rd y) = u_j \Psi_j(x) 
\end{equation}
so $G_{i,j} = \int_{\mb{U}^2} \Phi_i(x)  T(x,y) \Psi_j(y)  \lambda(dx) \lambda(\rd y)   = \int_{\mb{U}} \Phi_i(x) u_j \Psi_j(x) \lambda(\rd x) = 2\pi u_j   \delta_{0}(j-i) 
$. We observe that $((G^{-1})^T)_{i,j} = (2\pi)^{-1} \delta_{0}(j-i) u_j^{-1} $, $T(x,y) = T(y,x)$ and $u_{N+1-i} = u_i$. Therefore, we obtain
\begin{align*}
K(0,x;0,y) & = \frac{1}{2\pi} \sum_{i=1}^N \sum_{j=1}^N  \delta_{0}(j-i) u_j^{-1}  \Phi_i(x) u_j \Psi_j(y)   = \frac{1}{2\pi} \sum_{k=1}^N e^{\ii  (x-y) (k- \frac{N+1}{2})} = K(1,x;1,y)  \\
K(0,x;1,y) & = \frac{1}{2\pi}  \sum_{k=1}^N  e^{\frac{1}{2}(k-\frac{N+1}{2})^2t} e^{\ii (x-y) (k-\frac{N+1}{2})}\\
K(1,x;0,y) & =  - \sum_{k \in \mb{Z}} (-1)^{k(N-1)} p_t(y,x+2k\pi) + \frac{1}{2\pi}   \sum_{k=1}^N  e^{-\frac{1}{2}(k-\frac{N+1}{2})^2t} e^{\ii (x-y) (k-\frac{N+1}{2})}.
\end{align*}
The result follows by using \eqref{eq:T-Poisson-sum}, taking $t \to \frac{2t}{N}$ and conjugating the kernel by $e^{(\frac{N-1}{2})^2 \frac{t}{N}}$.
\end{proof}

Proposition \ref{Prop:DPP-kernel} and elementary calculations lead to the following corollary, which expresses the multi-time covariance of linear statistics in a remarkably simple form, even though we won't make use of it.

\begin{corollary}[Covariance of linear statistics] Consider the dynamics (\ref{eqn:Unitbis}) and denote  $\sgn(x)=\mathds{1}_{x>0} - \mathds{1}_{x<0}$.
For $H^{1/2}$ functions $f$ and $g$,  we have  for every $N$, $t \geq 0$,
\begin{align*}
\cov \left( \sum_k f(z_k(0)), \sum_k g(z_k(t)) \right) =   \sum_{|j| \leq N-1} \hat{f}_j \hat{g}_{-j} \sgn(j) e^{- |j| t} \frac{\sinh ( \frac{j^2 t}{N} )}{\sinh( \frac{jt}{N} )} +  \sum_{|j| \geq N} \hat{f}_j \hat{g}_{-j} e^{- j^2 \frac{t}{N}} \frac{\sinh( j t)} {\sinh( \frac{j t}{N})}.
\end{align*}
\end{corollary}
Later on, we will use the following pointwise estimates on the off-diagonal terms of the kernel obtained in Proposition \ref{Prop:DPP-kernel}. In the following lemma, $\tau$ and $\mu$ are allowed to depend on $N$ and both are $\OO(N)$. In our applications, we will only need cases when they are $\OO(N^{\delta})$ for some $0<\delta <1$.
\begin{lemma}[Pointwise estimates]
\label{lem:off-diagonal}
With $x-y = \frac{\mu}{N}$ and $ t = \frac{\tau}{N}$, we have as $N \to \infty$,
\begin{align}
\label{eq:Riemann-sum-approx-1}
\frac{1}{N} K(0,x;1,y) & = \frac{1}{2\pi}  \int_{|z| < 1/2} e^{(z^2-\frac{1}{4}) \tau + \ii \mu z} \rd z +\OO( \frac{\tau + |\mu|}{N} ), \\
\label{eq:Riemann-sum-approx-2}
\frac{1}{N} K(1,x;0,y) & = \frac{1}{2\pi} \int_{|z| > 1/2} e^{(\frac{1}{4}-z^2) \tau + \ii \mu z} \rd z + \OO( \frac{\tau + |\mu|}{N} ).
\end{align}
Furthermore, when $\max(\tau, |\mu|) \gg 1$, $|K(0,x;1,y)| + |K(1,x;0,y) | = \OO( N/\max(\tau,|\mu|))$.
\end{lemma}
We won't need \eqref{eq:Riemann-sum-approx-1} and \eqref{eq:Riemann-sum-approx-2}. The interest stems from the fact that when $\tau$ and $|\mu|$ are $\OO(1)$, they describe the limiting kernel at the microscale.

\begin{proof}
The first assertion \eqref{eq:Riemann-sum-approx-1} follows by using a Riemann sum approximation. Indeed, with $t = \frac{\tau}{N}$ and $x-y = \frac{\mu}{N}$, we have
$$
\frac{1}{N} K(0,x;1,y)  = \frac{1}{2\pi N} \sum_{k=1}^N  e^{ (\frac{k}{N}-\frac{1}{N})(\frac{k}{N}-1) \tau + \ii \mu (\frac{k}{N} - \frac{1}{2} - \frac{1}{2N}) } = \frac{1}{2\pi} \int_0^1 e^{x(x-1) \tau + \ii \mu (x-\frac{1}{2})} \rd x +\OO( \frac{\tau+|\mu|}{N}  )
$$
and we note $\int_0^1 e^{z(z-1) \tau + \ii \mu (z-\frac{1}{2})} \rd z = \int_{|z| < 1/2} e^{(z^2-\frac{1}{4})\tau + \ii \mu z} \rd z$. Along the same lines, we obtain \eqref{eq:Riemann-sum-approx-2}.

The second assertion \eqref{eq:Riemann-sum-approx-2} and the third one follow from elementary calculation. We explain the main ideas. For \eqref{eq:Riemann-sum-approx-2}, we note that the main contribution to $\sum_{k=-N}^N e^{((\frac{k}{N})^2-1)\tau}$ is $\sum_{(1-\eps) N \leq |k| \leq N  }  e^{((\frac{k}{N})^2-1)\tau} \leq 2 \sum_{(1-\eps) N \leq |k| \leq N  }  e^{-(1-\frac{k}{N})\tau} \leq 2 (e^{\frac{\tau}{N}}-1)^{-1} = \OO( N/\tau)$. The last one follows by a discrete integration by parts. Set $v_k = e^{((\frac{k}{N})^2-1)\tau}$, $e_k = e^{\ii \frac{k \mu}{N}}$, $w_0 = 0$ and $w_k = e_k + w_{k-1}$ for $1 \leq k \leq N$. Then, the term of interest, $\sum_{k=1}^N v_k e_k$ is equal to $v_N w_N + \sum_{k=1}^{N-1} (v_k-v_{k+1}) w_k$. Finally, $w_k = \sum_{\ell=1}^k e_{\ell} = \OO( 1) (e^{\ii \frac{\mu}{N}}-1)^{-1} = \OO( N/\mu)$ and $v_k$ is increasing.
\end{proof}

\noindent {\bf Out of equilibrium.} In the case of non-stationary initial data, the point process of eigenvalues at a fixed time is also determinantal point process and we provide here an expression of an associated kernel. In the Hermitian case, a self contained proof can be found in \cite[Appendix]{DuiJoh2018}. As seen above, the density of unlabeled eigenvalues (e.g., use a test function which is invariant under permutation) is proportional to
$$
w^N_{t,x}(y) := \frac{e^{\lambda_N t}}{N} \frac{|\Delta(y)|}{|\Delta(x)|}  \det \left( \sum_{k \in \mb{Z}} (-1)^{k(N-1)} p_t( x_i, y_j + 2k\pi) \right).
$$
Note that when $y_1 < \dots < y_N < y_1 + 2\pi$, we saw that we can write $|\Delta(y)| =  \ii^{-\frac{N(N-1)}{2}} \det(e^{\ii y_i(j-1-\frac{N-1}{2})})$ so that, recalling the notation $T$ and $\Psi$ in \eqref{def:ort-pol-trans}, we can identify (up to multiplicative constant) the weights 
$\det T(x_i,y_j)  \det \Psi_{i}(y_j)$. This is a biorthogonal ensemble (see \cite[Section 4]{Bor2011}) so a determinantal point process whose correlation kernel is given by $K_{t,x}(z,y) = \sum_{i,j}^N A^{-T}_{i,j} T(x_i,z) \Psi_j(y)
$ where, using \eqref{eq:integ},
$A_{i,j} = \int_{\mb{U}} T(x_i,y) \Psi_j(y) \lambda(dy)  = u_j \Psi_j(x_i) $ and $u_n = e^{- \frac{1}{2}(n- \frac{N+1}{2})^2 t} $.  Now, by recalling Cramer's formula,
$$
 \sum_{j=1}^n (A^{-1})_{i,j} b_j =  (A^{-1} b)_i = \frac{\det ( \text{col }  i \text{ of } A \text{ is replaced by } b)}{\det (A)},
$$
and we find
\begin{equation}
\label{eq:generel-kernel}
K_{t,x}(z,y) = \sum_{i,j} \Psi_j(y) (A^{-1})_{j,i} T(x_i,z)  = \sum_i T(x_i,z)  \frac{\det ( \text{line }  i \text{ of } A \text{ is replaced by } \Psi(y))}{\det (A)}.
\end{equation}
We denote by $\bar{A}^i$ the matrix for which the line $i$ in $A$ is replaced by $\Psi(y)$, i.e., $\bar{A}^i_{i,j} = \Psi_j(y)$ for $j \leq N$. Recalling \eqref{def:ort-pol-trans}, we have
\begin{align*}
\det A&  =( \prod_j u_j ) \det \Psi_j(x_i) = \prod_j u_j  \prod_i e^{-\ii x_i (\frac{N-1}{2})} \prod_{i < j} (e^{\ii x_j} - e^{\ii x_i }), \\
\det A^i &  =( \prod_j u_j ) \det \left( \Psi_j (x_k) 1_{k \neq i} + u_j^{-1} \Psi_j (y) 1_{k=i} \right).
\end{align*}
On the line $i$, we use (with $B_t \sim \mc{N}(0,t)$), $u_j^{-1} e^{\ii y (j-1) } = \E e^{ -\frac{N-1}{2} B_t}   e^{(\ii y+B_t) (j-1) }$. So, with simplifications coming from the quotient of Vandermonde determinants
$$
\frac{\det \bar{A}^i}{\det A} = e^{-\ii (y-x_i) ( \frac{N-1}{2})} \E e^{ -\frac{N-1}{2} B_t} \prod_{j \neq i} \frac{e^{\ii y +B_t}- e^{\ii x_j}}{e^{\ii x_i}- e^{\ii x_j}} = \E \prod_{j\neq i} \frac{\sin (\frac{y-\ii B_t -x_j}{2})}{\sin (\frac{x_i-x_j}{2})},
$$
and, going back to \eqref{eq:generel-kernel}, we obtain
\begin{equation}
\label{eq:kernel-out-eq}
K_{t,x}(z,y) =  \sum_i T(x_i,z) \E \prod_{j\neq i} \frac{\sin (\frac{y-\ii B_t -x_j}{2})}{\sin (\frac{x_i-x_j}{2})}.
\end{equation}
This expression is the analog of the Hermitian one used, e.g., in \cite{Joh2001, DuiJoh2018}. By using the residue theorem and expressing $\E$ as an integral, it is possible to give a contour integral representation of \eqref{eq:kernel-out-eq}.

\subsection{Asymptotic space-time decoupling}. \label{subsec:decoupling}
In the context of random matrices, using the determinant point processes machinery to obtain correlation/decorrelation estimates is not uncommon. A good illustration of the typical techniques can be found in \cite{PaqZei2017} which exploits the kernel obtained for the GUE minor process in \cite{JohNor2006} to derive such estimates. The starting point is usually a norm estimate for the differences of Fredholm determinants such as $|\det(I + A) - \det(I+B)| \leq |A-B| e^{1+|A|+|B|}$ (see \cite[Section 6.3]{PaqZei2017})  or a similar inequality for 2-regularized determinants (see \cite[Section 10]{PaqZei2017}). In our problem, such inequalities do not seem to be adapted since they give an additive error term and we look for a multiplicative one. We introduce here a method adapted for such errors.

For $j \in \llbracket 1, J\rrbracket$ and $E_j$'s on the unit circle, we define
\begin{equation}
\label{eq:truncated-singularities}
f_j(z)=e^{\gamma\ell_+(z-E_j)}\, \chi(\frac{z-E_j}{\lambda})+e^{\gamma\log(2\lambda)}\, \left(1-\chi(\frac{z-E_j}{\lambda})\right), \qquad \lambda=N^{-1+\kappa},
\end{equation}
where $\chi$ is fixed,  smooth,  $\chi(z)\in[0,1]$,  $\chi(z)=1$ for $|z|\leq 1$ and $\chi(z)=0$ for $|z|\geq 2$ (and recall Definition \ref{def:SmoothLog} for the definition of $\ell_+$).
Here $\kappa\in(0,1)$ is fixed.  The main result of this section, concerning the decoupling of the eigenvalues of the unitary Brownian motion \eqref{eq:UBM-conv},  is the following.

\begin{proposition}[Decoupling]
\label{Prop:decoupling} Let $\delta>0$,  $C>1$ such that $\min_{i \neq j} |(E_i,t_i)-(E_j,t_j)| \geq N^{-1+\delta}$
and $\gamma_z\in[0,C]$ for any singularity $z$.  Then for any $0\leq \alpha,\kappa\leq \frac{\delta}{100C}$ we have
$$
\E\left[\prod_{j=1}^J \prod_{i=1}^N f_j(z_i(t_j)) \right]  = \prod_{j=1}^J \E\left[\prod_{i=1}^N f_j(z_i(t_j)) \right]  (1+ \OO(N^{-\delta/3})).
$$
\end{proposition}

\begin{proof}
 The proof is easier to follow and has simpler notations for $J =2$ but generalizes immediately to an arbitrary fixed $J$. In this case we write $d = \max (\mu,\tau)$ where $\mu =  N\nc|E_1-E_0|$ and $\tau = N\nc |t_1-t_0|$,  so $d \geq N^{\delta}$. Furthermore,  the proposition is equivalent to prove that
$$
\E\left[\prod_{i=1}^N h_1(z_i(t_1))\cdot \prod_{i=1}^N h_0(z_i(t_0)) \right]  = 
\E\left[\prod_{i=1}^N h_1(z_i(t_1))\right]\cdot \E\left[\prod_{i=1}^N h_0(z_i(t_0)) \right]\cdot (1+\OO(N^{-\delta/3})),
$$
where $h_j(x)=\frac{f_j(x)}{f_j(E_j e^{2\ii\lambda})}$,  since they are equal up to a multiplicative constant. 
\bigskip

\bigskip

\noindent {\it First step: operators, spectra.} 
We introduce $k_i = \sqrt{1-h_i}$. By definition of $f$, $h=1$ for $|z-E_i|> 2\lambda$,   so the support of $k$ is of order $\OO(\lambda)$.

Moreover, for any $z$ we have  $h_i(z)\geq \varepsilon$
where (remember that $\ell_+$ is a smoothing of $\log$ on scale $N^{-1-\alpha}$)
\[
\eps := h_i(E_i)  = \frac{c_{\gamma}}{ ( N^{\kappa+\alpha})^{\gamma}} (1+\oo(1)).
\]

Let $\widetilde K$ be the kernel for independence between times $0$ and $t$, and $K$ the kernel we are interested in. Let $\mathcal{K}$, $\widetilde{\mathcal{K}}$ be the corresponding convolution kernels, namely $\mc{K} (r,x;s,y) = k_r(x) K(r,x;s,y) k_s(y)$.

The spectrum of $\widetilde{\mathcal{K}}$ is the union of the spectra of $\widetilde{\mathcal{K}_0}$ and $\widetilde{\mathcal{K}_1}$, the corresponding fixed time operators.
We have
$$
\E\left[\prod (1+x(h_0-1))(z_i(0))\right]=\det(\Id-x \widetilde{\mathcal{K}_0}).
$$
As $h_0 \geq \eps$, $x \mapsto 1+x(h_0-1) > 0$ on $[0, \frac{1}{1-\eps})$, and the left-hand side is $>0$ for any $x\in[0, \frac{1}{1-\eps})$,  we have $1-x\mu_i\neq 0$ for any $x\in[0, \frac{1}{1-\eps})$ and eigenvalue $\mu_i$ of $\widetilde{\mathcal{K}_0}$. We observe that $\widetilde{\mathcal{K}_0}$ is nonnegative so the spectrum of $\widetilde{\mathcal{K}_0}$ is in $[0,1-\eps)$, and the same property holds for $\widetilde{\mathcal{K}_1}$ and  $\widetilde{\mathcal{K}}$.

We now consider the Fredholm determinant of interest, i.e.
$$
\E\left[\prod h_0(z_i(0))\prod h_1(z_i(t))\right]=\det(\Id-\mathcal{K}).
$$
Since the entries of $\mc{K}$ are real-valued, we also have
$$
\E\left[\prod h_0(z_i(0))\prod h_1(z_i(t))\right]=\det(\Id-\mathcal{K}^*),
$$
so
$$
\E\left[\prod h_0(z_i(0))\prod h_1(z_i(t))\right]^2=\det((\Id-\mathcal{K})(\Id-\mathcal{K}^*)).
$$
Indeed, the multiplication rule for determinants $\det (I + A) \det (I + B) = \det ( (I+A)(I+B) )$  is justified for trace class operators $A$ and $B$. Here, this follows by an argument similar to the one of \cite[Proposition 2.4]{Joh2003}

We introduce the operators $\mathbf{K} := \mathcal{K}+\mathcal{K}^*-\mathcal{K}\mathcal{K}^*$ and $ \wt{\mathbf{K}} := \widetilde{\mathcal{K}}+\widetilde{\mathcal{K}^*}-\widetilde{\mathcal{K}}\widetilde{\mathcal{K}^*}$. Since $\widetilde{\mathcal{K}}$ is self-adjoint, the eigenvalues of $\wt{\mathbf{K}}$ are of type $\mu+\mu-\mu^2$ for $\mu\in [0,1-\eps)$, so its spectrum is included in $[0,1-\eps^2)$. 

\bigskip

\noindent {\it Second step: Hoffman-Wielandt inequality.} We know that if we order the eigenvalues $\lambda_i$ (resp. $\tilde\lambda_i$) of $\mathcal{K}+\mathcal{K}^*-\mathcal{K}\mathcal{K}^*$ (resp. $\widetilde{\mathcal{K}}+\widetilde{\mathcal{K}^*}-\widetilde{\mathcal{K}}\widetilde{\mathcal{K}^*}$) properly, we have by the Hoffman-Wielandt inequality (more precisely an infinite dimension version from \cite{Kato}, 
\begin{align*}
\sum|\lambda_i-\tilde\lambda_i|^2 \leq \|\mathbf{K}-\widetilde{\mathbf{K}}\|_{\rm HS}^2 & = \|(\mathcal{K}+\mathcal{K}^*-\mathcal{K}\mathcal{K}^*)-(\widetilde{\mathcal{K}}+\widetilde{\mathcal{K}^*}-\widetilde{\mathcal{K}}\widetilde{\mathcal{K}^*})\|_{\rm HS}^2\\
& \leq C \| \mc{K}-\wt{\mc{K}} \|_{\rm HS}^2 + C (\| \wt{\mc{K}} \|_{\rm HS}^2+\| \mc{K} \|_{\rm HS}^2) \| \mc{K} - \wt{\mc{K}} \|_{\rm HS}^2.
\end{align*}
We will use this and the estimates on the kernel to prove the following inequalities,
\begin{equation}
\label{eq:3rd-step}
\|\mc{K} \|_{\rm{HS}} = \OO(N^{\kappa}), \qquad
\|\mc{K} - \widetilde{\mc{K}} \|_{\rm{HS}} = \OO (N^{\kappa}/d), \qquad
\sqrt{\sum|\lambda_i-\tilde\lambda_i|^2} \leq \|\mathbf{K}-\widetilde{\mathbf{K}}\|_{\rm HS} = \OO(N^{2\kappa}/d).
\end{equation}
Let us mention an important consequence for what follows: this implies, for $N$ large enough, for any $i$, 
\begin{equation}
\label{eq:spectral-gap}
-1/2 \leq \lambda_i \leq 1 - \eps^2+ \OO(N^{2\kappa}/d) = 1 -\eps^2 + \oo(\eps^2), \qquad \text{when } N^{2\kappa}/d = \oo(\eps^2).
\end{equation}

For the first inequality, since $\lambda_i \geq \tilde{\lambda}_i - |\lambda_i - \tilde{\lambda}_i| \geq - |\lambda_i - \tilde{\lambda}_i|$,  $\sum_{\lambda_i<0}|\lambda_i|^2 \leq \sum_i |\lambda_i - \tilde{\lambda}_i|^2=\oo(1)$ when $N^{2\kappa} = \oo(d)$, which will be the case. So for large enough $N$,  for any $i$ we have $\lambda_i > -\frac{1}{2}$.  For the second one, for any $i$, $\lambda_i \leq \tilde{\lambda}_i+ |\lambda_i- \tilde{\lambda}_i|   \leq 1-\eps^2+ \OO(N^{2\kappa}/d)$.

Now, we prove these estimates. From \eqref{eq:kernel-diag} and since the size of the support of $k_i$  is $\OO(N^{-1+\kappa})$, we have $\int | k_i(x) K(0,x;0,y) k_i(y) |^2 dxdy \leq C \frac{(N^{-1+\kappa})^2}{N^2} N^2$ so $\| \wt{\mc{K}} \|_{\rm HS} = \OO( N^{\kappa})$.   

Furthermore $\| \mc{K} - \wt{\mc{K}} \|_{\rm HS} = \OO( N^{\kappa}/d)$ since Lemma \ref{lem:off-diagonal} gives a pointwise upper bound  $\OO(N/d)$ on the off-diagonal terms of the kernel and the size of the support is $\OO(N^{-1+\kappa})$. Therefore, we find $\|\mathbf{K}-\widetilde{\mathbf{K}}\|_{\rm HS} = \OO( N^{2\kappa}/d)$.

\bigskip

\noindent {\it Third step:  expansion of eigenvalues.} We will conclude by proving
\begin{equation}
\label{eq:decoupling-final}
\left|\log\widetilde \E\left[\prod h_1(z_i(0))\prod h_2(z_i(t))\right]-\log\E\left[\prod h_1(z_i(0))\prod h_2(z_i(t))\right]\right| =  \OO(\frac{N^{4\kappa}}{\eps \sqrt{d} }).
\end{equation} 
We bound from above the left-hand side above by expressing it with Fredholm determinants and using the following expansion of the logarithm for $\mathbf{K}$ (and similarly for $\widetilde{\mathbf{K}}$),
$$
\log \det((\Id-\mathcal{K})(\Id-\mathcal{K}^*)) = \log \det (\Id - \mathbf{K}) =-\sum_{j,\ell\geq 1}\frac{\lambda_j^\ell}{\ell}.
$$
Thus, we obtain with $m$ to be chosen,
$$
| \log \det (\Id - \wt{\mathbf{K}} ) - \log \det (\Id - \mathbf{K} ) |  \leq  \sum_{\ell=1}^m \left|{\rm Tr}({\mathbf{K}}^\ell)-{\rm Tr}({\widetilde{\mathbf{K}}}^\ell)\right|/\ell
+ \sum_{j \geq 1, \ell \geq m}\frac{|\lambda_j|^\ell+|\tilde{\lambda}_j|^\ell} {\ell}.
$$

First, we bound from above the contribution for $\ell \geq m$,
$$
\sum_{j \geq 1,\ell\geq m}\frac{|\lambda_j|^\ell}{\ell}   \leq \frac{1}{m} \sum_j \frac{|\lambda_j|^2}{1-|\lambda_j|}  \leq \frac{1}{m \eps^2} \| \mathbf{K}  \|_{\rm{HS}}^2 \leq \frac{N^{4\kappa}}{m \eps^2}
$$
where we used $\lambda_j\in[-1/2,1-\eps^2/2]$ in the second inequality.  We proceed similarly to bound the contribution of the $\tilde{\lambda}_j$'s.

Now we prove that the remaining term $ \sum_{\ell=1}^m \left|{\rm Tr}({\mathbf{K}}^\ell)-{\rm Tr}({\widetilde{\mathbf{K}}}^\ell)\right|/\ell =  m \OO (N^{4\kappa}/d)$. For $\ell=1$, we use
$$
| \Tr ( \mathbf{K} - \wt{\mathbf{K}}) | = | \Tr (\mc{K} \mc{K}^* -\wt{\mc{K}} \wt{\mc{K}}^* )| =|  \| \mc{K} \|_{\mathrm{HS}}^2 -  \| \wt{\mc{K}} \|_{\mathrm{HS}}^2| \leq \| \mc{K}- \wt{\mc{K}} \|_{\mathrm{HS}} ( \| \mc{K} \|_{\mathrm{HS}} +  \| \wt{\mc{K}} \|_{\mathrm{HS}})  = \OO(N^{2\kappa}/d).
$$
For $\ell \geq 2$, we write
$$
\left|{\rm Tr}({\mathbf{K}}^\ell)-{\rm Tr}(\widetilde{{\mathbf{K}}}^\ell)\right|/\ell\leq
\sum_i|\lambda_i^\ell-\widetilde\lambda_i^\ell|/\ell
\leq
\sum_i|\lambda_i-\widetilde\lambda_i|(|\lambda_i|^{\ell-1}+|\widetilde\lambda_i|^{\ell-1}).
$$
Then, by Cauchy-Schwarz we have
$$
\sum_i|\lambda_i-\widetilde\lambda_i| |\lambda_i|^{\ell-1}
\leq
(\sum_i|\lambda_i-\widetilde\lambda_i|^2 )^{1/2}
(\sum_i |\lambda_i|^{2\ell-2})^{1/2} \leq \|\mathbf{K}-\widetilde{\mathbf{K}}\|_{\rm HS} \|\mathbf{K}\|_{\rm HS} = \OO(N^{4\kappa}/d),
$$
where we used the Hoffman-Wielandt inequality, the fact that $|\lambda_j| \leq 1$  and $\ell\geq 2$ in the second inequality. The term with $|\tilde{\lambda}_i|$ can be bounded similarly. 

Altogether, the left-hand side in \eqref{eq:decoupling-final} is bounded  from above (up to a multiplicative constant) by $\frac{N^{4\kappa}}{m \eps^2}  + m \frac{N^{4\kappa}}{d}$. By taking $m = \sqrt{d}/\eps$, we get $\frac{N^{4\kappa}}{\eps \sqrt{d} }$.

\smallskip

{\it Conclusion.} We explain how we choose the parameters $\alpha,\kappa$, given $\delta,C$. Recalling $d = \max(\mu,\tau) \geq N^{\delta}$, $\eps \asymp N^{-\gamma(\alpha+\kappa)}$,  for any choice  $0\leq \alpha,\kappa\leq   \frac{\delta}{100 C}$ we have   $\frac{N^{2\kappa}}{\eps^2 d} = \oo(1)$ in \eqref{eq:spectral-gap} and $\frac{N^{4\kappa}}{\eps \sqrt{d} } = N^{-\delta/3}$ in the right-hand side of \eqref{eq:decoupling-final}.
\end{proof}

\section{Resolvent estimates}\label{sec:resolvent}

This section proves quantitative limits for the unitary analogue of the resolvent.  Some of our intermediate results are similar to existing local laws proved for random self-adjoint matrices (see e.g.  results and references from \cite[Chapter 6]{ErdYau2017}).  These resolvent estimates are the source of the almost optimal scales  in Theorem \ref{thm:FH}, and follow from a family of stochastic advection equations.
As explained in the following subsection,  dynamical methods for rigidity of the eigenvalues or bounds on eigenvectors have been increasingly important in random matrix theory.  We obtain for the first time optimal resolvent estimates  in both a {\it multi-time} and {\it full rank} setting, in the Proposition \ref{prop:avIso}. This is made possible thanks to (1)  Lemma \ref{lem:StocAdv} below which covers arbitrary projections of the resolvent,  (2) an iterative method to obtain first estimates  on eigenvalues,  then finite rank diagonal projections of the resolvent, then finite rank off-diagonal projections,  and finally full rank. 

The methods in this section could apply to some initial conditions out of equilibrium.  For the sake of simplicity we only consider dynamics close to equilibrium,  as this paper's main goal is showing a connection between random matrix dynamics and Liouville quantum gravity, not proving its universality.

\subsection{Stochastic advection equation for general observables}.
This subsection proves the stochastic advection equation for a generalization of the Borel transform
$$
m_t(z) = \frac{1}{N} \sum_k \frac{z+e^{\ii \theta_k(t)}}{z-e^{\ii \theta_k(t)}} = \frac{1}{N}\tr \left( \frac{z+U_t}{z-U_t} \right),
$$
which is defined, for any $N\times N$ deterministic matrix $A$,  as
\begin{equation}\label{eqn:obs}
m_{t,A}(z) =\tr \left( \frac{z+U_t}{z-U_t}\cdot A \right).
\end{equation}
The lemma below is instrumental for all results of this section.
\begin{lemma}\label{lem:StocAdv}
Under the unitary Brownian motion dynamics (\ref{eqn:Unitbis}), we have
$$
\rd m_{t,A}(z)=z m_t(z)\partial_z m_{t,A}(z)\rd t+2z {\rm Tr}\left(\frac{1}{z-U}A\frac{U}{z-U}\sqrt{2}\rd B_t\right).$$
\end{lemma}

At equilibrium we have $\E(m_t(z))=\mathds{1}_{|z|>1}-\mathds{1}_{|z|<1}$, so from the above lemma  at leading order $m_{t,A}$ should be well approximated by the solution of the advection equation
\begin{equation}\label{eqn:sa}
\frac{\rd}{\rd t} f_t(z)=z(\mathds{1}_{|z|>1}-\mathds{1}_{|z|<1}) \partial_z f_t(z),
\end{equation}
which has characteristics
\begin{equation}\label{eqn:char}
z_t=ze^{t} \mathds{1}_{|z| > 1} + z e^{- t} \mathds{1}_{|z|<1}.
\end{equation}
In other words we expect $$m_{t,A}(z)\approx m_{0,A}(z_t).$$
It has been known since Pastur's work \cite{Pas1972} that the Stieltjes transform of the Hermitian Dyson Brownian motion satisfies an advection equation analogous to (\ref{eqn:sa}),  in the limit of large dimension.
More general resolvent dynamics corresponding to $A$ with rank one can be used for regularization and universality purpose, as proved first in
\cite{LeeSch2015}, for eigenvalues statistics at the edge of deformed Wigner matrices. For the same model,  \cite{Ben2017,SooWar2019} used stochastic advection equations and characteristics to understand the shape of bulk eigenvectors.  Moreover,  the stochastic complex Burgers equation for the Stieltjes transform extends to general $\beta$-ensembles and allows to prove rigidity of the particles \cite{HuaLan2019,AdhHua2018},  also through regularization along the characteristics.   
For a general class of discrete particle systems,  analogues of the Stieltjes transform were also recently shown to satisfy equations of type (\ref{eqn:sa}) \cite{GorHua2022}. 

More directly relevant to our model,  the unitary Brownian motion,  complex Burgers equation for the Borel transform were first shown by Biane \cite{Bia1997,Bia1997II},  and they are instrumental in Adhikari and Landon's recent result on optimal location of eigenvalues out of equilibrium,  starting at identity \cite{AdhLan2023}.

While most of these works focus on the trace of the resolvent,  Lemma \ref{lem:StocAdv} considers general full-rank projections observables: it covers the Stieltjes transform (i.e.  $A={\rm Id}$ below, used in Proposition \ref{prop:rig}),  one-dimensional projections (i.e.  $A=qq^*$, used in Proposition \ref{prop:ani1}),  and a full-rank $A$ is needed for the proof of Proposition \ref{prop:avIso},  a main estimate towards Theorem \ref{thm:FH}.

\begin{proof}[Proof of Lemma \ref{lem:StocAdv}] Recall the definition of the skew Hermitian Brownian motion in \eqref{def:shbm}.  From It{\^o}'s formula \eqref{eq:Ito-UBM},  we have
\begin{align*}
\rd \frac{z+U}{z-U} & = 2 z \rd \frac{1}{z-U} \\
& =  2 z \sum_k  \frac{1}{z- U } \sqrt{2}U X_k \rd \tilde{B}_t^k  \frac{1}{z- U }  + 2z \left( \sum_k \frac{1}{z-U} U X_k^2 \frac{1}{z-U} + 2 \frac{1}{z-U} U X_k \frac{1}{z-U} U X_k \frac{1}{z-U} \right)\rd t\\
& = 2 z \frac{1}{z- U } U \sqrt{2}\rd B_t  \frac{1}{z- U } - 2z \frac{U}{(z-U)^2} \rd t+ 4 z \sum_k \frac{1}{z-U} U X_k \frac{1}{z-U} U X_k \frac{1}{z-U}\rd t.
\end{align*}
We have used
$\sum_{k=1}^{N^2} X_k^2  = -{\rm Id}$.
This implies (we use that for any two complex valued matrices $P$ and $Q$,
$\sum_{k=1}^{N^2} \tr (P X_k Q X_k)   = -N^{-1} \tr (P) \tr (Q)$)
\begin{align*}
\rd{\rm Tr}\left(\frac{z+U}{z-U} A\right)&=2z {\rm Tr}\frac{1}{z-U}A\frac{U}{z-U}\sqrt{2}\rd B-2z{\rm Tr}\frac{U}{(z-U)^2}A\rd t-4zN^{-1}{\rm Tr}\frac{U}{(z-U)^2}A\,{\rm Tr}\frac{U}{z-U}\rd t\\
&=2z {\rm Tr}\frac{1}{z-U}A\frac{U}{z-U}\sqrt{2}\rd B+2zN^{-1}\partial_z{\rm Tr}\frac{z+U}{z-U}A\,{\rm Tr}\frac{{\rm Id}}{2}\rd t+2zN^{-1}\partial_z{\rm Tr}\frac{z+U}{z-U}A\,{\rm Tr}\frac{U}{z-U}\rd t.
\end{align*}
As $\frac{{\rm Id}}{2}+\frac{U}{z-U}=\frac{1}{2}\frac{z+U}{z-U}$, we obtain the expected result.
\end{proof}

\subsection{Rigidity}.\ 
\label{subsec:rig} 
The following parameters
\begin{equation}
\label{eq:phi-delta}
\varphi=e^{(\log \log N)^2}, \Delta=(\log N)^2
\end{equation}
will often be used in this section, and so will be the notation
$$
\eta_v=||v|-1|.
$$
We order $0\leq\theta_1(s)\leq\dots\leq \theta_N(s)\leq 2\pi$ and for any $t\geq s$ we define $\theta_1(t)\leq\dots\leq \theta_N(t)$ by continuity. 
We consider $\gamma_k=\frac{2\pi k}{N}$ and the following good sets, 
$$
\mathscr{G}=\bigcap_{1\leq k\leq N}\left\{|\theta_k-\gamma_k|\leq \frac{\varphi^{6}}{N}\right\},\ \widetilde{\mathscr{G}}=\bigcap_{1\leq k\leq N}\left\{|\theta_k-\gamma_k|\leq \frac{\varphi}{N}\right\}.
$$
We also denote $\btheta(t)=(\theta_1(t),\dots,\theta_N(t))$.  The proposition below is a unitary analogue of
classical rigidity results for Hermitian random matrices, see \cite{ErdYau2017} and references therein.

\begin{proposition}\label{prop:rig} For any $D>0$ there exists $N_0$ such that for any $N\geq N_0$ we have
$$
\mathbb{P}\left(\bigcap_{s\leq t\leq s+\Delta}\{\btheta(t)\in\mathscr{G}\}\mid \btheta(s)\in\widetilde{\mathscr{G}}\right)\geq 1-e^{-(\log N)^D}.
$$
\end{proposition}

\begin{proof} The proof will proceed through (1) resolvent estimate at fixed space and time,  (2) uniform extension to  any time and mesoscopic scales,  (3) extension to submicroscopic scales,  (4) rigidity of gaps between eigenvalues, (5) rigidity of positions.\\

\noindent {\it First step: resolvent estimate.}\
We choose $A={\rm Id}/N$ in Lemma  \ref{lem:StocAdv},  which gives (in this section we define new independent Brownian motions through $\rd\hat B_{jk}(s)=(P_s\rd B(s)P_s^*)_{jk}$  with $P_s$ unitary diagonalizing $U_s$,  $U_s=P_sD_sP_s^*$, and abbreviate $\hat B_{j}=\hat B_{jj}$)
\begin{equation}\label{eqn:sae}
\rd m_t(z) =m_t(z) z \partial_z m_t(z)\rd t+ \frac{2\sqrt{2}\ii z}{N^{3/2}}\sum_{k=1}^N\frac{z_k(t)}{(z-z_k(t))^2}\rd\hat B_k(t).
\end{equation}
The following implementation of invariance along characteristics in this first step is similar to the proof of \cite[Theorem 1.2]{AdhLan2023}.

Without loss of generality we assume $s=0$ and we first consider some $|z|\in[1+\varphi^{8/5}/N,2]$.  Equations (\ref{eqn:sae}) and (\ref{eqn:char}) imply
\begin{equation}\label{eqn:afterChar}
\rd m_u(z_{t-u})= (m_u(z_{t-u})-1)z_{t-u}\partial_z m_{ u}(z_{t-u})\rd u+ \frac{2\sqrt{2}\ii z_{t-u}}{N^{3/2}}\sum_{k=1}^N\frac{z_k(u)}{(z_{t-u}-z_k(u))^2}\rd\hat B_k(u).
\end{equation}
We consider the stopping time (with respect to the filtration generated by $\hat B_1,\dots,\hat B_N$)
\begin{align}\label{eqn:tau}
\tau=\inf\left\{u\in[0,t]: |m_{u}(z_{t-u})-1|\geq \frac{\varphi^{3/2}}{N\eta_{z_{t-u}}} \right\}\wedge t
\end{align}
with the convention $\inf\varnothing=+\infty$.
We also abbreviate
$$
M(s)=\int_0^{s} \frac{2\sqrt{2}\ii z_{t-u}}{N^{3/2}}\sum_{k=1}^N\frac{z_k(u)}{(z_{t-u}-z_k(u))^2}\rd\hat B_k(u).
$$
From 
\begin{equation}\label{eqn:basic}
\frac{{\rm Re}\ m_u(z)}{|z|^2-1} = \frac{1}{N} \sum_j\frac{1}{|z-z_j(u)|^2},
\end{equation}
the quadratic variation of the martingale $(M(s\wedge\tau))_s$ (the sum of the quadratic variations of its real and imaginary parts) is bounded at time $s$ with 
\begin{align}
\frac{1}{N^3}\int_0^{s}\sum_{i} \frac{|z_{t-u}|^2\rd u}{|z_i(u)-z_{t-u}|^4}
\leq \frac{C}{N^2}\int_0^{s}\frac{|z_{t-u}|^2{\rm Re}(m_u(z_{t-u}))}{(1+|z_{t-u}|)\eta_{z_{t-u}}^3}\rd u
\leq \frac{C}{N^2\eta_{z_{t-s}}^2},
\end{align}\label{eqn:quadr}
where we have used $|{\rm Re}(m_u(z_{t-u}))-1|=\oo(1)$ because $u\leq \tau$ and $\eta_z>\varphi^{8/5}/N$. 
This classically implies (see e.g.  \cite[Appendix B.6,  equation (18)]{ShoWel1986}) that for any $D>0$ there exists $N_0$ such that $\mathbb{P}(\cap_{0\leq s\leq t}\{|M_{s\wedge\tau}|<\frac{\varphi^{1/20}}{N\eta_{z_{t-s}}}\} ) \geq 1-e^{-(\log N)^D}$ for any $N>N_0$.  More precisely we have $\mathbb{P}(|M_{s\wedge\tau}|<\frac{\varphi^{1/20}}{N\eta_{z_{t-s}}})\geq 1-e^{-2(\log N)^{D}}$ and uniformity in time follows from a grid argument similar to the second step of this proof detailed below.

On the event $\cap_{0\leq s\leq t}\{|M_{s\wedge\tau}|<\frac{\varphi^{1/20}}{N\eta_{z_{t-s}}}\}$, which has overwhelming probability, for any $s\leq t \wedge \tau$ from
(\ref{eqn:afterChar}) we have, denoting $h(s)=m_{s\wedge\tau}(z_{t-s\wedge\tau})-1$,
$$
|h(s)|\leq 
\int_{0}^{s}|z_{t-u}|\cdot|h(u)|\cdot |\partial_z m_{u}(z_{t-u})|\rd u+\frac{\varphi^{1/20}}{N\eta_{z_{t-s}}}+\frac{\varphi^{11/10}}{N\eta_{z_{t-s}}}$$
where we have used $|m_0(z)-1|\leq \varphi^{11/10}/(N\eta_z)$ by Riemann sum approximation because $\btheta(0)\in\widetilde{\mathscr{G}}$. Together with $|\partial_zm|\leq 2({\rm Re}m)\cdot(|z|^2-1)^{-1}$ from (\ref{eqn:basic}), this implies
$$ 
|h(s)|\leq 
\int_{0}^{s}|h(u)|\cdot \frac{|z_{t-u}|2{\rm Re}\ m_{u\wedge\tau}(z_{t-u\wedge\tau})}{(|z_{t-u}|^2-1)}\rd u+C\frac{\varphi^{11/10}}{N\eta_{z_{t-s}}}
\leq 
\left(1+\varphi^{-1/10}\right)\int_{0}^{s}\frac{|h(u)|}{\log|z_{t-u}|}\rd u+C\frac{\varphi^{11/10}}{N\eta_{z_{t-s}}},
$$
where we successively relied on the inequalities $x/(x^2-1)<1/(2\log x)$ for $x>1$,
$|{\rm Re}(m_u(z_{t-u}))-1|\leq \frac{\varphi^{3/2}}{N\eta_{z_{t-u}}}\leq \varphi^{-1/10}$ for $u\leq \tau$ and $\eta_z>\varphi^{8/5}/N$.  The integral form of Gronwall lemma then implies 
$$
|h(s)|\leq C\frac{\varphi^{11/10}}{N\eta_{z_{t-s}}}+C\varphi^{11/10}\int_0^s\frac{1}{N\eta_{z_{t-u}} \log |z_{t-u}|}e^{(1+\varphi^{-1/10})\int_u^s\frac{\rd r}{\log|z_{t-r}|}}\rd u.
$$
The antiderivative of $({\log|z_{t-r}|})^{-1}$ is $\log\log|z_{t-r}|$, so
$\exp(\int_u^s\frac{\rd r}{\log|z_{t-r}|})=\frac{\log|z_{t-u}|}{\log|z_{t-s}|}$. Moreover for our parameters we always have $(\frac{\log|z_{t-u}|}{\log|z_{t-s}|})^{\varphi^{-1/10}}\leq C$.  Thus we have obtained
$$|h(s)|\leq C \frac{\varphi^{11/10}}{N\eta_{z_{t-s}}}+C\varphi^{11/10}\int_0^s\frac{1}{N\eta_{z_{t-u}} \log |z_{t-s}|}\rd u\leq \frac{\varphi^{12/10}}{N\eta_{z_{t-s}}}.$$
This proves that 
$$
\mathbb{P}\left(|m_{\tau}(z_{t-\tau})-1|>\frac{\varphi^{3/2}}{N\eta_{z_{t-\tau}}} \right)\leq e^{-(\log N)^D}.
$$
By definition of $\tau$ this implies $\mathbb{P}(\tau=t)\geq 1-e^{-(\log N)^D}$, and therefore there exists $N_0$ such that for any $N>N_0$,  $1+\varphi^{8/5}/N<|z|<2$ and $0<t<\Delta$,
\begin{equation}\label{eqn:fixed point}
\mathbb{P}\left(|m_{t}(z)-1|>\frac{\varphi^{16/10}}{N\eta_z} \right)\leq e^{-(\log N)^D}.
\end{equation}

\noindent{\it Second step: Uniformity in space and time.}\ 
Let $D>0$ be fixed and large,  $M=e^{10(\log N)^D}$, $(z_i)_{1\leq i\leq M}$ (resp. $(t_j)_{1\leq j\leq M}$) be points in $|z|\in[1+\varphi^{8/5}/N,2]$ (resp. $[0,\Delta]$) such that for any such $|z|\in[1+\varphi^{8/5}/N,2]$ there exists $z_i$ with $|z-z_i|\leq N^{-4}$ 
(resp. $0=t_1<\dots<t_{M}=\Delta$,  $|t_{j+1}-t_j|\leq e^{-5(\log N)^D}$).  Then by union bound in  (\ref{eqn:fixed point}),   there exists $N_0$ such that for $N\geq N_0$ we have
\begin{equation}\label{eqn:fixed point2}
\mathbb{P}\left(\cap_{1\leq i,j\leq M}\{|m_{t_j}(z_i)-1|<\frac{\varphi^{17/10}}{N\eta_{z_i}} \}\right)\geq 1- e^{-2(\log N)^D}.
\end{equation}
Moreover, for any fixed $z_i$ and $t_j$,  a bracket calculation and again, for example \cite[Appendix B.6,  equation (18)]{ShoWel1986}), imply
\begin{equation}\label{eqn:incrementbound}
\mathbb{P}\left(\max_{t_j<t<t_{j+1}}|m_t(z_i)-m_{t_j}(z_i)|>N^{-3} \right)\leq e^{-100(\log N)^D}.
\end{equation}
Equations (\ref{eqn:fixed point2}), (\ref{eqn:incrementbound}) and a union bound give existence of $N_0$ such that, for $N\geq N_0$,
\begin{equation}\label{eqn:timeunif}
\mathbb{P}\left(\cap_{1\leq i\leq M,0<t<\Delta}\{|m_{t}(z_i)-1|<\frac{\varphi^{18/10}}{N\eta_{z_i}}\} \right)\geq 1 - e^{-(\log N)^D}.
\end{equation}
The function $z\mapsto m_{t}(z)$ is deterministically $N^2$-Lipschitz for $|z|>1+\varphi^{8/5}/N$.  Therefore from the previous equation, for some $N_0$,  $N>N_0$ implies
\begin{equation}\label{eqn:proba}
\mathbb{P}\left(\bigcap_{
\substack{1+\varphi^{8/5}/N<|z|<2},0<t<\Delta}
\{|m_{t}(z)-1|\leq \frac{\varphi^{19/10}}{N\eta_z}\}\right)\geq 1-e^{-(\log N)^D}.
\end{equation}

\noindent{\it Third step: Extension below microscopic scales.}\ 
We now consider $|z|\in[1,1+\varphi^{2}/N]$.  Let $z'$ have the same argument as $z$ and $\eta_{z'}=\varphi^2/N$.  The following always holds,  for some universal $C$:
\begin{equation}\label{eqn:monotonic}
{\rm Re}\,m_t(z)\leq C \frac{\eta_{z'}}{\eta_z}{\rm Re}\,m_t(z').
\end{equation}
Therefore, for any arc $I$ of length at most $\varphi^2/N$ centered at $|w|=1$, denoting $w_I=w(1+|I|)$,  under the event considered in (\ref{eqn:proba}) we have 
$$
\sum\mathds{1}_{z_i(t)\in I}
\leq C\sum_{i}\frac{\eta_{w_I}^2}{|w_I-z_i(t)|^2}
\leq C N\eta_{w_I} {\rm Re}m_t(w_I)\leq C N\eta_{w_I}\frac{\eta_{w'}}{\eta_{w_I}}{\rm Re}m_t(w')\leq C\varphi^2{\rm Re}m_t(w')\leq C\varphi^2,
$$
and, denoting $\eta_k=e^k\eta_z$, $\tilde{z}_k = z \frac{1+\eta_k}{1+\eta_z}$ so that  $\eta_{\tilde{z}_k} = \eta_k$,  we obtain by using $|\arg (z) - \arg (z_i)| \leq C |z - z_i|$ in the first inequality and $-\log (\varphi^2/N) \leq \varphi \leq \varphi^3/(N\eta_z)$ in the last one (with $k\geq 0$ in all series below),
\begin{multline}
| {\rm Im} m_t(z) |
\leq \frac{C}{N} \sum_i\frac{|{\rm arg}z-{\rm arg}z_i|}{|z-z_i|^2}
\leq  \frac{C}{N} \sum_{k \geq 0, e^k \eta_z \leq |z-z_i| \leq e^{k+1} \eta_z } \frac{1}{|z-z_i|} \\
\leq\frac{C}{N} \sum_{k \geq 0} \frac{\big|\{|z-z_i| \leq \varphi^2/N\}\big|}{e^k\eta_z }   + \frac{C}{N} \sum_{\varphi^2/N \leq e^k \eta_z \leq 1, i} \frac{\eta_k}{|\tilde{z}_k-z_i|^2} \leq 
\frac{C\varphi^2}{N\eta_z}+C\sum_{\varphi^2/N\leq e^k\eta_z\leq 1}{\rm Re } m_t(\tilde{z}_k)\leq \frac{\varphi^{3}}{N\eta_z}.\label{eqn:monotonic2}
\end{multline}
From (\ref{eqn:monotonic}) with $|\re m_t(z)-1| \leq 1 + \re m_t(z)$, (\ref{eqn:monotonic2}) and their analogue for $1/2<|z|<1$,  (\ref{eqn:proba}) extends into (we denote $s(z)=\mathds{1}_{|z|>1}-\mathds{1}_{|z|<1}$)
\begin{equation}\label{eqn:proba2}
\mathbb{P}\left(\bigcap_{
\substack{1/2<|z|<2},0<t<\Delta}
\{|m_{t}(z)-s(z)|\leq \frac{\varphi^{4}}{N\eta_z}\}\right)\geq 1-e^{-(\log N)^D}.
\end{equation}
\noindent{\it Fourth step: Rigidity of gaps.}\ The inclusion 
\begin{equation}\label{eqn:gaps}
\bigcap_{
\substack{1/2<|z|<2},0<t<\Delta}
\{|m_{t}(z)-m_0(z_{t})|\leq \frac{\varphi^{4}}{N\eta_z}\}\subset \bigcap_{\substack{0\leq t\leq \Delta\\1\leq i<j\leq N}}\{|\theta_i(t)-\theta_j(t)-(\gamma_i-\gamma_j)|\leq \frac{\varphi^{5}}{N}\}
\end{equation}
holds for large enough $N$ thanks to the following classical argument based on the Helffer-Sj{\"o}strand formula (\ref{eqn:HS-rep}).  
Here we follow  \cite[Section 6.2]{AdhLan2023}.
Let $g(z)=1$ for $\arg z\in[\gamma_i+\varphi^{4}/N,\gamma_j-\varphi^{4}/N]$,  $g(z)=0$ for $\arg z\in[\gamma_i,\gamma_j]^{\rm c}$, and $|g'|\leq C N / \varphi^{4}, |g''|\leq C(N / \varphi^{4})^2$. We pick $\chi$ from (\ref{eqn:HS-rep}) on scale 1.  On the set from  the left-hand side of (\ref{eqn:gaps}), we have,
denoting $s(z)=m_t(z)-(\mathds{1}_{|w|>1}-\mathds{1}_{|w|<1})$,
\begin{multline*}
\sum g(z_i(t))=\frac{N}{2\pi}\int_{\mathbb{C}}\partial_{\bar w}\tilde g(w) m_t(w)\frac{\rd m(w)}{w}=\frac{N}{2\pi}\int_{\mathbb{C}}\partial_{\bar w}\tilde g(w) (\mathds{1}_{|w|>1}-\mathds{1}_{|w|<1})\frac{\rd m(w)}{w}+
\int_{\mathbb{C}}\partial_{\bar w}\tilde g(w) s(w)\frac{\rd m(w)}{w}\\
=N\int g(e^{\ii\theta})\frac{\rd\theta}{2\pi}+\int_{\mathbb{C}}\partial_{\bar w}\tilde g(w) s(w)\frac{\rd m(w)}{w}.
\end{multline*}
As $\btheta(0)\in\widetilde{\mathscr{G}}$, we have $m_0(w_{t})=(\mathds{1}_{|w|>1}-\mathds{1}_{|w|<1})+\OO(\frac{\varphi}{N\eta_{w_{t}}})$
so on the set from  the left-hand side of (\ref{eqn:gaps}), $s(z)=\OO(\frac{\varphi}{N\eta_{w_{t}}})$.
Together with the decomposition (\ref{eqn:dbar}) this gives
\begin{multline*}
\sum g(z_i(t))=N\int g(e^{\ii\theta})\frac{\rd\theta}{2\pi}+
\OO(\varphi^{4})\cdot\int(|g(e^{\ii\theta})|+|g'(e^{\ii\theta})|)\cdot |\chi'(r)|r\rd r\rd\theta+2\int_{\theta,r>1} g''(e^{\ii\theta}){\rm Re}s(re^{\ii\theta})\chi(r)\log r\rd r\rd\theta\\
=N\int g(e^{\ii\theta})\frac{\rd\theta}{2\pi}+
\OO(\varphi^{4})+2\int_{\theta,r>1} g''(e^{\ii\theta}){\rm Re}s(re^{\ii\theta})\chi(r)\log r\rd r\rd\theta
\end{multline*}
With an integration by parts as in \cite[Equation (6.20)]{AdhLan2023}, the remaining integral term is also $\OO(\varphi^{4})$.

Similarly we  have $\sum h(z_i(t))=\frac{N}{2\pi}\int h+\OO(\varphi^{4})$  where $h$ has the same regularity as $g$ and $h(z)=1$ for $\arg z\in[\gamma_i,\gamma_j]$,  $h(z)=0$ for $\arg z\in[\gamma_i-\varphi^{4}/N,\gamma_j+\varphi^{4}/N]^{\rm c}$.  These estimates  on $\sum g(z_i(t))$ and $\sum h(z_i(t))$
prove (\ref{eqn:gaps}), which together with
 (\ref{eqn:proba}) gives
\begin{equation}\label{eqn:proba3}
\mathbb{P}\left(
\bigcap_{\substack{0\leq t\leq \Delta\\1\leq i<j\leq N}}\{|\theta_i(t)-\theta_j(t)-(\gamma_i - \gamma_j)|\leq \frac{\varphi^{5}}{N}\}
\right)\geq 1-e^{-(\log N)^D}.
\end{equation}

\noindent{\it Fifth step: rigidity of positions.}\ Let $\bar\theta(t)=\sum_i\theta_i (t)$.  Then \eqref{eqn:Dyson} gives
$
\rd \bar{\theta}(t)= \sum_{j} \sqrt{\frac{2}{N}} \rd B_j(t)= \sqrt{2} \rd B(t)
$
where $B$ is a standard Brownian motion. This implies that for any $D>0$ there exists $N_0$ such that for $N\geq N_0$
\begin{equation}\label{eqn:proba4}
\mathbb{P}\left(\cap_{0<t<\Delta} \{ |\bar\theta(t)-\bar\theta(0)|  \leq \varphi \} \right)\geq 1-e^{-(\log N)^D}.
\end{equation}
We now write
$$
\theta_i(t)-\gamma_i=\frac{1}{N}\sum_{j=1}^N\left((\theta_i(t)-\theta_j(t))-(\gamma_i-\gamma_j)\right)+\frac{1}{N}\sum_{j=1}^N(\theta_j(t)-\theta_j(0))+\frac{1}{N}\sum_{j=1}^N(\theta_j(0)-\gamma_j).
$$
With probability $1-e^{-(\log N)^D}$, the following holds.  For all $i$ and $t\in [0,\Delta]$
the first term is at most $\varphi^{5}/N$ (from (\ref{eqn:proba3})),  the second is at most  $\varphi/N$ (by  (\ref{eqn:proba4})),  and the last one is at most $C\varphi/N$ because $\btheta(0)\in{\tilde{\mathscr{G}}}$. This concludes the proof.
\end{proof}

\subsection{Finite rank projections}.\ 
The result below shows the following:  eigenvectors perturbations under mean field noise are simply given at the level of the resolvent by moving the spectral parameter through the characteristics.  It is a simple analogue of \cite[Theorem 2.1]{BouHuaYau2017}, which considers Hermitian perturbations out of equilibrium,  but our dynamical proof is different from \cite{BouHuaYau2017}, which  proceeds through the Schur complement formula.
Such estimates on arbitrary (finite rank) projections of the resolvent first appeared in the context of Wigner and covariance matrices, see e.g.  \cite{BloErdKnoYauYin} and references therein.

\begin{proposition}\label{prop:ani1}For any $D>0$ there exists $N_0$ such that for any $N\geq N_0$ and $q\in\mathbb{C}^N$ $\mathscr{F}_s$-measurable ($\mathscr{F}_s=\sigma(U_u,u\leq s)$), $|q|=1$, we have
$$
\mathbb{P}\left(\bigcap_{\substack{s<t<s+\Delta\\
|\eta_z|>\varphi^{20}/N}}\{
\left|\langle q, \frac{z+U_t}{z-U_t}q\rangle-\langle q, \frac{z_{t-s}+U_s}{z_{t-s}-U_s}q\rangle\right|\leq \frac{\varphi}{\sqrt{N\eta_z}}{\rm Re}\langle q, \frac{z_{t-s}+U_s}{z_{t-s}-U_s}q\rangle\}
\mid \btheta(s)\in{\mathscr{G}}\right)\geq 1-e^{-(\log N)^D}.
$$
\end{proposition}
Note that the above real part is always positive.

\begin{proof}
We choose $A=qq^*$ in Lemma  \ref{lem:StocAdv}.  This gives 
$$
\rd q_t(z) = m_t(z) z \partial_z q_t(z)+ 2 z q^* \frac{U}{z-U} \sqrt{2}\rd B_t\frac{1}{z-U}q \ {\rm where }\ q_t(z) = \langle q, \frac{z+U_t}{z-U_t}q\rangle.
$$
We can assume $s=0$
and first consider some $|z|>1+\varphi^{20}/N$.  Then, with orthonormal eigenvectors $u_j(s)$ diagonalizing $U_s$,
\begin{equation}\label{eqn:afterChar2}
\rd q_u(z_{t-u})= (m_u(z_{t-u})-1)z_{t-u}\partial_z q_u(z_{t-u})\rd u+ \frac{2\sqrt{2}\ii z_{t-u}}{N^{1/2}}\sum_{k,j}q^*u_j(u)\frac{z_j(u)}{z_{t-u}-z_j(u)}\rd {\hat B}_{jk}(u)\frac{1}{z_{t-u}-z_k(u)}{u_k(u)^*q},
\end{equation}
where the $\hat B_{jk}$ are independent Brownian motions defined before (\ref{eqn:sae}).
We define the stopping times
\begin{align}\label{eqn:tauq}
\tau_q:&=\inf\left\{u\in[0,t]: |q_0(z_t)-q_u(z_{t-u})|> \frac{\varphi^{1/10}}{\sqrt{N\eta_{z_{t-u}}}}\,{\rm Re}\, q_0(z_{t})\right\},\\
\tau:&=\inf\left\{u\in[0,t]: \exists k\in\llbracket 1,N\rrbracket, |\theta_k(u)-\gamma_k|>\frac{\varphi^8}{N}\right\},\\
\sigma:&=\tau\wedge\tau_q.
\end{align}
The quadratic variation of the martingale term in (\ref{eqn:afterChar2}) stopped at $\sigma$ is bounded with
\begin{align}
\frac{C}{N}\int_0^{\sigma}
\sum_{j,k}|z_{t-u}|^2\frac{|\langle q,u_j(u)\rangle|^2}{|z_{t-u}-z_j(u)|^2}\cdot\frac{|\langle q,u_k(u)\rangle|^2}{|z_{t-u}-z_k(u)|^2}\rd u
\leq \frac{C}{N}\int_0^{\sigma}\frac{|z_{t-u}|^2({\rm Re}(q_u(z_{t-u})))^2}{(1+|z_{t-u}|)^2\eta_{z_{t-u}}^2}\rd u
\leq C\frac{{\rm Re}(q_0(z_{t}))^2}{N\eta_{z_{t-\sigma}}},
\end{align}
where we have used
$
\sum_j\frac{|\langle q,u_j(u)\rangle|^2}{|z-z_j(u)|^2}=\frac{1}{|z|^2-1}{\rm Re}\ q_u(z).
$
Similarly to the estimate after (\ref{eqn:quadr}),  this implies that this martingale term is bounded with $\frac{\varphi^{1/10}}{\sqrt{N\eta_{z_{t-\sigma}}}}\,{\rm Re}\, q_0(z_{t})$ with probability $1-e^{-(\log N)^D}$.
Moreover,  the finite variation error term from (\ref{eqn:afterChar2}) is bounded with
$$
\int_{0}^{\sigma}|z_{t-u}|\cdot|m_u(z_{t-u})-1|\cdot |\partial_z q_u(z_{t-u})|\rd u\leq C
\int_{0}^{\sigma}\frac{\varphi^8|z_{t-u}|}{N\eta_{z_{t-u}}}\cdot \frac{{\rm Re}\ q_u(z_{t-u})}{\eta_{z_{t-u}}(1+|z_{t-u}|)}\rd u\leq 
\frac{C\varphi^8\,{\rm Re}\, q_0(z_{t})}{N\eta_{z_{t-\sigma}}}\leq \frac{C\,{\rm Re}\, q_0(z_{t})}{\sqrt{N\eta_{z_{t-\sigma}}}},
$$
where we have first used that for $u<\tau$ we have $m_u(z)-1=\OO(\varphi^8/(N\eta_z))$, and finally we have used $|\eta_{z_{t-\sigma}}|>\varphi^{20}/N$.
We have therefore proved that for any $D>0$ there is a $N_0$ such that for $N\geq N_0$ and $|z|>1+\varphi^{20}/N$ we have
$$
\mathbb{P}\left(|q_{\sigma}(z_{t-\sigma})-q_0(z_t)|>\frac{\varphi^{1/10}}{\sqrt{N\eta_{z_t-\sigma}}}{\rm Re}q_0(z_{t})\right)\leq e^{-(\log N)^D}.
$$
By definition of $\tau_q$ this implies $\mathbb{P}(\sigma=\tau)\geq 1-e^{-(\log N)^D}$. Moreover, from Proposition \ref{prop:rig},  $\mathbb{P}(\tau=t)\geq 1-e^{-(\log N)^D}$ (this proposition naturally also holds when replacing exponents $\varphi,\varphi^6$ defining $\mathscr{G}, {\tilde{\mathscr{G}}}$ with $\varphi^6,\varphi^8$), so we have proved
$$
\mathbb{P}\left(|q_{t}(z)-q_0(z_t)|>\frac{\varphi^{1/9}}{\sqrt{N\eta_{z_{t-\sigma}}}}{\rm Re}q_0(z_{t})\right)\leq e^{-(\log N)^D}.
$$
Uniformity in $t\in[0,\Delta]$ and $\eta_z\in[\varphi^{20}/N,1/2]$ follows easily by a grid argument similar to the second step in the proof of Proposition \ref{prop:rig}.

Finally,  for uniformity in $\eta_z>1/2$, denote $f(z)=\langle q, \frac{z+U_t}{z-U_t}q\rangle$, $g(z)=\langle q, \frac{z_t+U_0}{z_t-U_0}q\rangle$. We have proved that with overwhelming probability 
$|\frac{f}{g}(z)-1|\leq \frac{\varphi^{1/9}}{\sqrt{N\eta_z}}$. As $f/g-1\to 0$ as $|z|\to\infty$, the Cauchy integral formula for $z$ outside the contour $|w|=6/5$ gives, for $|z|>7/5$,  $f/g(z)-1=\OO(\frac{\varphi^{1/9}}{|z|\sqrt{N}})$, which concludes the proof.
\end{proof}

Polarization in Proposition \ref{prop:ani1} shows that if $u_a(s),u_b(s)$ are normalized eigenvectors of $U(s)$ and $a\neq b$, then for $|z|>1+\varphi^{20}/N$ we have
$$
|\langle u_a(s),\frac{z+U_t}{z-U_t}u_b(s)\rangle|\leq \varphi \frac{\eta_{z_{t-s}}(|z_{t-s}|+1)}{\sqrt{N\eta_{z}}}
\left(\frac{1}{|z_{t-s}-z_{a}(s)|^2}+
\frac{1}{|z_{t-s}-z_{b}(s)|^2}\right)
$$
with overwhelming probability. This error term is not enough for Proposition \ref{prop:avIso} in the next subsection,  so we first obtain the following essentially optimal bound.

\begin{proposition}\label{prop:offDiag}For any $D,\varepsilon>0$ there exists $N_0$ such that for any $N\geq N_0$ and $u_a(s)$ $u_b(s)\in\mathbb{C}^N$  eigenvectors of $U(s)$ associated to distinct eigenvalues ($|u_a|=|u_b|=1$) we have
$$
\mathbb{P}\left(\bigcap_{\substack{s<t<s+\Delta\\
\eta_z>N^\varepsilon/N}}\{
|\langle u_a(s),\frac{z+U_t}{z-U_t}u_b(s)\rangle|\leq \frac{\eta_{z_{t-s}}(1+|z_{t-s}|)N^\varepsilon}{\sqrt{N\eta_{z}}}
\frac{1}{|z_{t-s}-z_{a}(s)|}
\frac{1}{|z_{t-s}-z_{b}(s)|}\}
\mid \btheta(s)\in{\mathscr{G}}\right)\geq 1-e^{-(\log N)^D}.
$$
\end{proposition}

\begin{proof}
We choose $A=u_b(s)u_a(s)^*$ in Lemma  \ref{lem:StocAdv}, and we abbreviate $a=u_a(s)$, $b=u_b(s)$.  Defining
$$
p_t(z)=p_t^{a,b}(z)  = \langle a, \frac{z+U_t}{z-U_t}b\rangle,
$$
this gives 
$$
\rd p_t(z) = m_t(z) z \partial_z p_t(z)+ 2 z u_a(s)^* \frac{U}{z-U} \sqrt{2}\rd B_t\frac{1}{z-U}u_b(s).
$$
We can assume $s=0$. Note that $p_0(z_t)=0$ and we want to bound $p_t(z)$.
We first consider some $|z|>1+\varphi^{30}/N$.  Then 
\begin{equation}\label{eqn:afterChar3}
\rd p_u(z_{t-u})= (m_u(z_{t-u})-1)z_{t-u}\partial_z p_u(z_{t-u})\rd u+ \frac{2\sqrt{2}\ii z_{t-u}}{N^{1/2}}\sum_{k,j}a^*u_j(u)\frac{z_j(u)}{z_{t-u}-z_j(u)}\rd {\hat B}_{jk}(u)\frac{1}{z_{t-u}-z_k(u)}{u_k(u)^*b},
\end{equation}
where the $\hat B_{jk}$ are independent Brownian motions and $\hat B_{jj}=\hat B_{j}$ from (\ref{eqn:afterChar}),  and $u_k(s)$ are  orthonormal eigenvectors diagonalizing $U_s$.
The quadratic variation of the martingale term in (\ref{eqn:afterChar3}) is bounded with
$$
\frac{C}{N}\int_0^t\frac{|z_{t-u}|^2}{\eta_{z_{t-u}}^2(1+|z_{t-u}|^2)}
{\rm Re}\langle a,\frac{z_{t-u}+U(u)}{z_{t-u}-U(u)}a\rangle
{\rm Re}\langle b,\frac{z_{t-u}+U(u)}{z_{t-u}-U(u)}b\rangle\rd u.
$$
From Proposition \ref{prop:ani1},  with probability $1-e^{-3(\log N)^D}$ this is bounded with
$$
\frac{C}{N}\int_0^t\frac{1}{\eta_{z_{t-u}}^2}
{\rm Re}\langle a\frac{z_{t}+U(0)}{z_{t}-U(0)}a\rangle
{\rm Re}\langle b\frac{z_{t}+U(0)}{z_{t}-U(0)}b\rangle\rd u
\leq
\frac{C}{N\eta_{z}}
\frac{\eta_{z_t}(1+|z_t|)}{|z_t-z_{a}(0)|^2}
\frac{\eta_{z_t}(1+|z_t|)}{|z_t-z_{b}(0)|^2},
$$
so that with probability $1-e^{-2(\log N)^D}$ the martingale term in (\ref{eqn:afterChar3}) is bounded with 
$\frac{\varphi}{\sqrt{N\eta_{z}}}
\frac{\eta_{z_t}(1+|z_t|)}{|z_t-z_{a}(0)|\cdot |z_t-z_{b}(0)|}$, which is the expected error.

A new difficulty comes from the finite variation error term in (\ref{eqn:afterChar3}): for $a\neq b$, ${\rm Re}\, p^{a,b}$ has no a priori sign.  We therefore first simply bound
$
|\partial_z p^{a,b}_u|\leq \frac{C}{\eta_z(1+|z|)}({\rm Re}\, p^{a,a}_u+{\rm Re}\,  p^{b,b}_u)
$
and use Proposition \ref{prop:ani1} and its proof to obtain
\begin{multline*}
\int_{0}^{t}|z_{t-u}|\cdot|m_u(z_{t-u})-1|\cdot |\partial_z p^{a,b}_u(z_{t-u})|\rd u\leq 
\int_{0}^{t}\frac{\varphi^8|z_{t-u}|}{N\eta_{z_{t-u}}}\cdot \frac{{\rm Re}\ p^{a,a}_u(z_{t-u})+{\rm Re}\ p^{b,b}_u(z_{t-u})}{\eta_{z_{t-u}}(1+|z_{t-u}|)}\rd u\\
\leq 
\frac{\varphi^8}{N\eta_{z}}\cdot\left({\rm Re}\ p^{a,a}_0(z_{t})+{\rm Re}\ p^{b,b}_0(z_{t})\right)\leq\frac{\varphi^8\eta_{z_t}(1+|z_t|)}{N\eta_{z}}\cdot\left(\frac{1}{|z_t-z_a(0)|^2}+\frac{1}{|z_t-z_b(0)|^2}\right).
\end{multline*}
We have therefore proved, that, for any $D>0$ there exists $N_0$ such that for any $N\geq N_0$,  with probability $1-e^{-(\log N)^D}$ we have 
$$
|p_t^{a,b}(z)|\leq \frac{\varphi^8\eta_{z_t}(1+|z_t|)}{\sqrt{N\eta_{z}}}
\frac{1}{|z_{t}-z_{a}(0)|\cdot|z_{t}-z_{b}(0)|}
+
\frac{\varphi^8\eta_{z_t}(1+|z_t|)}{N\eta_{z}}\cdot\left(\frac{1}{|z_t-z_a(0)|^2}+\frac{1}{|z_t-z_b(0)|^2}\right)
$$
for any $\eta_z>\varphi^{30}/N$ and $t\in[0,\Delta]$ (uniformity in $z,t$ requires (1) an omitted  grid argument identical to the second step in the proof of Proposition \ref{prop:rig} for $\eta_z\in[\varphi^{30}/N,1/2],t\in[0,\Delta]$, (2) a contour integral argument similar to the end of the proof of Proposition \ref{prop:ani1} to extend to $\eta_z>1/2$).

We now iterate by injecting this estimate in the finite variation term from (\ref{eqn:afterChar3}). More precisely, consider the following induction hypothesis $({\rm P}_n)$:  For any $D>0$ there exists $N_0=N_0(n,D)$,  such that for any $N\geq N_0$,  $a,b\in\llbracket 1,N\rrbracket$, the following holds with probability $1-e^{-(\log N)^D}$:  for any $0<t<\Delta$ and $\eta_z>\varphi^{30n}/N$ we have
$$
|p_t^{a,b}(z)|\leq \frac{\varphi^{8n}\eta_{z_t}(1+|z_t|)}{\sqrt{N\eta_{z}}}
\frac{1}{|z_{t}-z_{a}(0)|\cdot|z_{t}-z_{b}(0)|}
+
\frac{\varphi^{8n}\eta_{z_t}(1+|z_t|)}{(N\eta_{z})^n}\cdot\left(\frac{1}{|z_t-z_a(0)|^2}+\frac{1}{|z_t-z_b(0)|^2}\right).
$$
We have just proved $(P_1)$,  and to prove that  $(P_n)$ implies $(P_{n+1})$ we just need to improve on the finite variation term.  By Cauchy's formula,
\begin{multline*}
\int_{0}^{t}|z_{t-u}|\cdot|m_u(z_{t-u})-1|\cdot |\partial_z p^{a,b}_u(z_{t-u})|\rd u\leq 
\int_{0}^{t}\frac{\varphi^8|z_{t-u}|}{N\eta_{z_{t-u}}}\cdot \frac{\max_{|w-z_{t-u}|=\eta_{z_{t-u}/10}}{|p^{a,b}_u(w)|}}{\eta_{z_{t-u}}}\rd u\\
\leq 
\int_{0}^{t}\frac{|z_{t-u}|}{N\eta_{z_{t-u}}^2}  \left(
\frac{\varphi^{8(n+1)}\eta_{z_t}(1+|z_t|)}{\sqrt{N\eta_{z_{t-u}}}}
\frac{1}{|z_{t}-z_{a}(0)|
\cdot|z_{t}-z_{b}(0)|}
+
\frac{\varphi^{8(n+1)}\eta_{z_t}(1+|z_t|)}{(N\eta_{z_{t-u}})^n}\cdot\left(\frac{1}{|z_t-z_a(0)|^2}+\frac{1}{|z_t-z_b(0)|^2}\right)
\right) \rd u\\
\leq  \frac{\varphi^{8(n+1)}\eta_{z_t}(1+|z_t|)}{\sqrt{N\eta_{z}}}
\frac{1}{|z_{t}-z_{a}(0)|\cdot|z_{t}-z_{b}(0)|}
+ \frac{\varphi^{8(n+1)}\eta_{z_t}(1+|z_t|)}{(N\eta_{z})^{n+1}}\cdot\left(\frac{1}{|z_t-z_a(0)|^2}+\frac{1}{|z_t-z_b(0)|^2}\right).
\end{multline*}
This completes the induction and the proof of the proposition by choosing $n=100/\varepsilon$ (now $\eta_z > N^{\eps}/N$).
\end{proof}

\subsection{Full rank projections}. We now prove the main estimate to reach optimal scales for multi-time loop equations, concerning the following resolvent projection,
\begin{equation}\label{eqn:start1}
{\rm Tr}\left(\frac{w+U_t}{w-U_t}\cdot \frac{v+U_s}{v-U_s}\right)=\sum_{k}\frac{v+z_k(s)}{v-z_k(s)}\langle u_k(s), \frac{w+U_t}{w-U_t}u_k(s)\rangle.
\end{equation}
If we add the error estimates from Proposition \ref{prop:ani1} on the above right-hand side, for example for $\eta_v\sim 1$,  the obtained bound is
$
\sqrt{N/\eta_w},
$
far worse than the bound $1/\eta_w$ below. The key source of improvement to achieve the optimal result below is Proposition \ref{prop:offDiag}.
The {\it averaged} and {\it multi-time} local law below seems to be new,  including in the context of Hermitian random matrices.

In the following statement,  we use the notation $d(v,w)=\max(|v-w|,|v-\frac{w}{|w|^2}|)$.

\begin{proposition}\label{prop:avIso} For any $D,\varepsilon>0$ there exists $N_0$ such that for any $N\geq N_0$ we have
\begin{multline*}
\mathbb{P}\Big(\bigcap_{\substack{s<t<s+\Delta\\
\eta_v,\eta_w\in(0,\frac{1}{2}]}}\{\left|{\rm Tr}\left(\frac{w+U_t}{w-U_t}\cdot \frac{v+U_s}{v-U_s}\right)-{\rm Tr}\left(\frac{w_{t-s}+U_s}{w_{t-s}-U_s}\cdot \frac{v+U_s}{v-U_s}\right)\right|\\
\leq \frac{N^\varepsilon(1+|w_{t-s}|)}{ d(v,w_{t-s})\,\min(1,N\eta_v)}(\frac{1}{\eta_w}+\frac{1}{\sqrt{\eta_w\min(\eta_{w_{t-s}},\eta_v)}})\nc\}
\mid \btheta(s)\in\mathscr{G}\Big)\geq 1-e^{-(\log N)^D}.
\end{multline*}
\end{proposition}

\begin{proof}
We can again assume $s=0$, and first consider the case
$|w|\in[1+N^{\e}/N,3/2]$, $|v|\in(1,3/2]$.
Lemma \ref{lem:StocAdv} with $A=\frac{v+U_0}{v-U_0}$ gives
\begin{multline}\label{eqn:stochint}
{\rm Tr}\left(\frac{w+U_t}{w-U_t}\cdot A\right)-{\rm Tr}\left(\frac{w_{t}+U_0}{w_{t}-U_0}\cdot A\right)=\int_0^t
w_{t-u} (m_{u}(w_{t-u})-1)\partial_w m_{u,A}(w_{t-u})\rd u\\+2\int_0^t w_{t-u} {\rm Tr}\left(\frac{1}{w_{t-u}-U_u}A\frac{U_u}{w_{t-u}-U_u}\sqrt{2}\rd B_u\right).
\end{multline}
The above stochastic integral can also be written 
$$
\frac{\sqrt{2}}{\sqrt{N}}\int_{0}^t\sum_{j,k}\frac{w_{t-u}}{w_{t-u}-z_j(u)}\frac{z_k(t)}{w_{t-u}-z_k(u)}\langle u_j(u),A u_k(u)\rangle \rd \hat B_{jk}(u)
$$
where the $\hat B_{jk}$ are independent, standard Brownian motions.  Abbreviating $\ell=u_{\ell}(0)$ and using the spectral decomposition $A=\sum_{\ell} \frac{v+z_\ell(0)}{v-z_\ell(0)}\ell \ell^*$,  the bracket of the above stochastic integral is (we denote, in this proof, $\langle x,y\rangle=x^*y$)
\begin{align}
&\frac{C}{N}\int_0^t\sum_{j,k}\frac{|w_{t-u}|^2}{|w_{t-u}-z_j(u)|^2}\frac{1}{|w_{t-u}-z_k(u)|^2}|\langle u_j(u),A u_k(u)|^2\rd u\notag\\
=&\frac{C}{N}\int_0^t\sum_{j,k}\frac{|w_{t-u}|^2}{|w_{t-u}-z_j(u)|^2}\frac{1}{|w_{t-u}-z_k(u)|^2}\left|\sum_\ell\langle u_j(u),\ell\rangle \langle \ell,u_k(u)\rangle\frac{v+z_\ell(0)}{v-z_\ell(0)}\right|^2\rd u\notag\\
=&\frac{C}{N}\int_0^t\sum_{\ell_1,\ell_2,j,k}
\frac{v+z_{\ell_1}(0)}{v-z_{\ell_1}(0)}\frac{\overline{v+z_{\ell_2}(0)}}{\overline{v-z_{\ell_2}(0)}}
\frac{|w_{t-u}|^2}{|w_{t-u}-z_j(u)|^2}\frac{1}{|w_{t-u}-z_k(u)|^2}\langle u_j(u),\ell_1\rangle \langle \ell_1,u_k(u)\rangle\overline{\langle u_j(u),\ell_2\rangle }\overline{\langle\ell_2,u_k(u)\rangle}\rd u\notag\\
=&\frac{C}{N}\int_0^t\sum_{\ell_1,\ell_2}
\frac{v+z_{\ell_1}(0)}{v-z_{\ell_1}(0)}\frac{\overline{v+z_{\ell_2}(0)}}{\overline{v-z_{\ell_2}(0)}}|w_{t-u}|^2
\left|\sum_j\frac{\langle\ell_2,u_j(u)\rangle\langle u_j(u),\ell_1\rangle}{|w_{t-u}-z_j(u)|^2} \right|^2\rd u\notag\\
\leq&\frac{C}{N}\int_0^t \frac{1}{\eta_{w_{t-u}}^2}
\sum_{\ell_1,\ell_2}
\frac{1}{|v-z_{\ell_1}(0)|\cdot|v-z_{\ell_2}(0)|}
\left|\langle \ell_2,\Big({\rm Re}\frac{w_{t-u}+U(u)}{w_{t-u}-U(u)}\Big)\ell_1\rangle \right|^2\rd u.\label{eqn:interm}
\end{align}
Note that $2{\rm Re}\frac{w+U}{w-U}=\frac{w+U}{w-U}-\frac{w^*+U}{w^*-U}$ where $w^*={\bar w}^{-1}$, so that we can apply
Proposition \ref{prop:offDiag} to bound the above:  With probability $1-e^{-4(\log N)^D}$
the contribution from $\ell_1\neq \ell_2$ in the above sum leads to evaluating
$$
\sum_{\ell_1,\ell_2}
\frac{N^\e}{|v-z_{\ell_1}(0)|\cdot|v-z_{\ell_2}(0))|}
\frac{\eta_{w_{t}}^2(1+|w_t|^2)}{N\eta_{w_{t-u}}}
\frac{1}{|w_{t}-z_{\ell_1}(0)|^2}
\frac{1}{|w_{t}-z_{\ell_2}(0)|^2}
=
\frac{N^{1+\e}}{\eta_{w_{t-u}}}\left(\frac{1}{N}\sum_{\ell}\frac{(1+|w_t|)\eta_{w_t}}{|v-z_{\ell}(0)|\cdot|w_{t}-z_{\ell}(0)|^2}\right)^2
$$
As $\eta_w>N^{-1+\e}$, by Proposition \ref{prop:rig} the following holds with overwhelming probability:
\begin{align*}
&\frac{1}{N}\sum_{\ell:|v-z_\ell(0)|>\frac{N^\e}{N}}\frac{\eta_{w_t}}{|v-z_{\ell}(0)|\cdot|w_{t}-z_{\ell}(0)|^2}\leq C\int_{|v-\lambda|>\frac{N^\e}{N}} \frac{\eta_{w_t}}{|v-\lambda|\cdot|w_{t}-\lambda|^2}\rd\lambda
\leq \frac{C}{d(v,w_{t})}\mathds{1}_{\eta_{w_t}<\eta_v}+ \frac{C\log N}{d(v,w_{t})}\mathds{1}_{\eta_{w_t}\geq\eta_v}\\
&\frac{1}{N}\sum_{\ell:|v-z_\ell(0)|\leq\frac{N^\e}{N}}\frac{\eta_{w_t}}{|v-z_{\ell}(0)|\cdot|w_{t}-z_{\ell}(0)|^2}
\leq 
\frac{N^{2\e}}{N\eta_v}\frac{\eta_{w_t}}{d(v,w_t)^2}\leq \frac{N^{2\e}}{N\eta_v}\frac{1}{d(v,w_t)}.
\end{align*}
so that we have obtained
\begin{equation}
\label{eqn:rig11}
\frac{1}{N}\sum_{\ell}\frac{\eta_{w_t}}{|v-z_{\ell}(0)|\cdot|w_{t}-z_{\ell}(0)|^2}\leq C\int \frac{\eta_{w_t}}{|v-\lambda|\cdot|w_{t}-\lambda|^2}\rd\lambda
\leq \frac{C}{d(v,w_{t})}\mathds{1}_{\eta_{w_t}<\eta_v}+ \frac{N^{2\e}}{d(v,w_{t})\min(1,N\eta_v)}\mathds{1}_{\eta_{w_t}\geq\eta_v}.
\end{equation}
Moreover, the contribution from the diagonal terms in (\ref{eqn:interm}) leads to a sum evaluated with Proposition \ref{prop:ani1}:
$$
\sum_{\ell}
\frac{1}{|v-z_{\ell}(0)|^2}\left({\rm Re}\langle \ell,\frac{w_{t-u}+U(u)}{w_{t-u}-U(u)}\ell\rangle \right)^2\leq
\sum_{\ell}
\frac{C}{|v-z_{\ell}(0)|^2}\left({\rm Re}\frac{w_t+z_{\ell}(0)}{w_t-z_{\ell}(0)} \right)^2\\
\leq \sum_{\ell}
\frac{C(1+|w_t|^2)\eta_{w_t}^2}{|v-z_{\ell}(0)|^2\cdot|w_t-z_{\ell}(0)|^4},$$
and similarly to (\ref{eqn:rig11}) we obtain
$$
\frac{1}{N}\sum_{\ell}
\frac{\eta_{w_t}^2}{|v-z_{\ell}(0)|^2\cdot|w_t-z_{\ell}(0)|^4}
\leq \frac{C}{\eta_{w_t}d(v,w_{t})^2}\mathds{1}_{\eta_{w_t}<\eta_v}+ \frac{C}{\eta_v d(v,w_{t})^2\min(1,N\eta_v)^2}\mathds{1}_{\eta_{w_t}\geq\eta_v}.
$$
Using the previous four estimates in (\ref{eqn:interm}), with probability $1-e^{-3(\log N)^D}$  the bracket of (\ref{eqn:stochint}) is bounded with
$$
\frac{N^{2\e}}{\min(1,N\eta_v)^2}\int_0^t \frac{1}{\eta_{w_{t-u}}^2}
\left(
\frac{1+|w_t|^2}{\eta_{w_{t-u}}d(v,w_t)^2}
+
\frac{1+|w_t|^2}{\min(\eta_{w_t},\eta_v)d(v,w_t)^2}
\right)
\rd u
\leq \frac{N^{2\e}(1+|w_t|^2)}{ d(v,w_{t})^2 \min(1,N\eta_v)^2}\cdot(\frac{1}{\eta_w^2}+\frac{1}{\eta_w\min(\eta_{w_t},\eta_v)})\nc,
$$
so, with probability $1-e^{-2(\log N)^D}$, (\ref{eqn:stochint}) is smaller than $\frac{N^\e(1+|w_t|)}{d(v,w_t)\min(1,N\eta_v)}\cdot(\frac{1}{\eta_w}+\frac{1}{\sqrt{\eta_w\min(\eta_{w_t},\eta_v)}})$.
We now consider the error term due to the finite variation term, based on (\ref{eqn:start1}):
\begin{multline*}
|\partial_w m_{u,A}(w_{t-u}))|\leq \sum\frac{1}{|v-z_k(0)|}\frac{1}{\eta_{w_{t-u}}(1+|w_{t-u}|)}{\rm Re}\langle k,  \frac{w_{t-u}+U_u}{w_{t-u}-U_u}k \rangle\\
\leq
\sum\frac{C}{|v-z_k(0)|}\frac{1}{\eta_{w_{t-u}}(1+|w_{t-u}|)}\frac{1}{|w_{t}-z_k(0)| }\leq \frac{CN^{1+\e}}{\eta_{w_{t-u}}(1+|w_{t-u}|)d(v,w_t)\min(1,N\eta_v)}
\end{multline*}
so
$$
\int_{0}^{t}|w_{t-u}|\cdot|m_u(w_{t-u})-1|\cdot |\partial_w m_{u,A}(w_{t-u}))|\rd u\leq \frac{N^\e}{\eta_wd(v,w_t)\min(1,N\eta_v)}.
$$
This concludes the proof of the proposition for $|w|\in[1+N^{\e}/N,3/2]$, $|v|\in(1,3/2]$.  The proof for $|w|\in[1/2,1-N^{\e}/N]$, $|v|\in[1/2,1)$ is strictly similar, and uniformity in $v,w$ and $t\in[0,\Delta]$ follows from
the same grid argument as in the second step in the proof of Proposition \ref{prop:rig}.

We now consider the case $|w|\in(1,1+N^\e/N]$ and $|v|\in(1,3/2]$, relying on the following analogue of (\ref{eqn:monotonic}), where we now denote $w'$ with the same argument as $w$ such that $\eta_{w'}=N^{-1+\e}$:
\begin{equation}\label{eqn:monotonicBis}
{\rm Re}\langle q, \frac{w+U_t}{w-U_t}q\rangle\leq C \frac{\eta_{w'}}{\eta_w}{\rm Re}\langle q \frac{w'+U_t}{w'-U_t}q\rangle.
\end{equation}
In the sequence below we start with (\ref{eqn:start1}), use 
(\ref{eqn:monotonicBis}) and proceed similarly to (\ref{eqn:monotonic2})
to bound the contribution from ${\rm Im}\langle q, \frac{w+U_t}{w-U_t}q\rangle$, denoting $w_k$ with the same argument as $w$ such that $\eta_{w_k}=e^k\eta_w$:
\begin{align*}
\left|{\rm Tr}\left(\frac{w+U_t}{w-U_t}\cdot \frac{v+U_0}{v-U_0}\right)\right|&\leq
\sum_{k}\frac{C}{|v-z_k(0)|}\left(\re\langle u_k(0), \frac{w+U_t}{w-U_t}u_k(0)\rangle+\Big|\im\langle u_k(0), \frac{w+U_t}{w-U_t}u_k(0)\rangle\Big|\right)\\
&\leq \sum_{k}\frac{C}{|v-z_k(0)|}\left(\frac{N^\e}{N\eta_w}{\rm Re}\langle u_k(0), \frac{w'+U_t}{w'-U_t}u_k(0)\rangle+\sum_{N^\e/N\leq e^j\eta_w\leq 1}{\rm Re}\langle u_k(0), \frac{w_j+U_t}{w_j-U_t}u_k(0)\rangle\right)\\
&\leq
\sum_{k}\frac{C}{|v-z_k(0)|}\left(\frac{N^\e}{N\eta_w}{\rm Re}\langle u_k(0), \frac{w'_{t}+U_0}{w'_{t}-U_0}u_k(0)\rangle+\sum_{N^\e/N\leq e^j\eta_w\leq 1}{\rm Re}\langle u_k(0), \frac{(w_j)_{t}+U_0}{(w_j)_{t}-U_0}u_k(0)\rangle\right)\\
&=\sum_{k}\frac{C}{|v-z_k(0)|}\left(\frac{N^\e}{N\eta_w}{\rm Re}\frac{w'_{t}+z_k(0)}{w'_{t}-z_k(0)}+\f
\sum_{N^\e/N\leq e^j\eta_w\leq 1}{\rm Re}\frac{(w_j)_{t}+z_k(0)}{(w_j)_{t}-z_k(0)}\right)\\
&\leq \sum_{k}\frac{C}{|v-z_k(0)|}\left(\frac{N^\e}{N\eta_w}\frac{\eta_{w'_t}(1+|w'_t|)}{|w'_{t}-z_k(0)|^2}+\sum_{N^\e/N\leq e^j\eta_w\leq 1}\frac{\eta_{(w_j)_t}(1+|(w_j)_t|)}{|(w_j)_{t}-z_k(0)|^2}\right).
\end{align*}
As in (\ref{eqn:rig11}), we have
$$
\sum_{k}\frac{\eta_{w'_t}}{|v-z_k(0)|\cdot|w'_{t}-z_k(0)|^2}\leq \frac{N^{1+\e}}{d(v,w'_{t})\min(1,N\eta_v)},
$$
and similarly for the terms involving $(w_j)_t$, which gives, as $w_t'$ is close to $w_t$,
\begin{equation}\label{eqn:keyEstimateExt}
\left|{\rm Tr}\left(\frac{w+U_t}{w-U_t}\cdot \frac{v+U_0}{v-U_0}\right)\right|\leq \frac{N^\e(1+|w_t|)}{\eta_w d(v,w_{t})\min(1,N\eta_v)}.
\end{equation}
The analogous estimate with $U_t$ replaced with $U_0$, and $w$ replaced with $w_t$ gives 
\begin{equation}\label{eqn:extTriv}
\left|{\rm Tr}\left(\frac{w_t+U_0}{w_t-U_0}\cdot \frac{v+U_0}{v-U_0}\right)\right|\leq \frac{N^\e(1+|w_t|)}{\eta_{w_t} d(v,w_t)\min(1,N\eta_v)}.
\end{equation}
This concludes the proof for  $|w|\in[1,1+ N^\e/N]$ and $|v|\in(1,3/2]$. The proof when $|w|\in[1- N^\e/N,1]$ follows the same argument.
\end{proof}

\section{Loop equations via stochastic analysis on the unitary group}

\label{sec:loop}

Integration by parts at the level of matrix process (and not at the level of the two-dimensional interacting particle systems constituted by its eigenvalues) has a particularly simple form in the case of the Dyson dynamics for the Gaussian Unitary Ensemble:  $M(t)$ is distributed according to  $e^{-t} M(0) + \sqrt{1-e^{-2t}} G$ where $G$  is a  GUE matrix of size $N$, whose density is proportional to $e^{- N \Tr(H^2)} D H$, and $G$ is independent of $M(0)$, the initial condition. Here, $DH$ is the Lebesgue measure on Hermitian matrices. The explicit potential in $e^{- N \Tr(H^2)} D H$ makes integration by parts tractable and this has been used for instance in \cite[Lemma 4.1]{DuiJoh2018} in the context of mesoscopic equilibrium for linear statistics in the GUE Dyson's Brownian motion. However, this very nice structure does not extend to the unitary Brownian motion and we use instead stochastic calculus on this Lie group, in particular Girsanov theorem and exact solutions of some matrix SDEs that characterize Fr\'echet derivatives as an alternative (Section 5.1 below).
%

Such integration by parts often carry the name loop equations in random matrix theory \cite{Joh1998}, where they traditionally relate correlation functions of particle systems (see \cite{For,Gui,Ser}),  i.e.  only eigenvalues in the context of random matrices. In our multitime and singular setting, the integration by parts formula (see Proposition \ref{prop:Girsanov}) encodes information/correlations about eigenvalues but also eigenvectors.

\subsection{Fr\'echet derivatives as explicit solutions of matrix SDEs}.\  As in the previous sections,  
 the Brownian motion $(U_t)$ on the unitary group is defined through (\ref{eqn:Unitbis}), i.e.
$$
\rd U_t = \sqrt{2}U_t \rd B_t -  U_t \rd t
$$
where $(B_t)$ is a Brownian motion on the space of skew Hermitian matrices. Note that if $M$ is Hermitian and $N$ is skew-Hermitian, then $\langle M ,  N \rangle_{\mf{R}} :=  \re (\Tr (\overline{M}^T  N )) = 0$.  
\begin{lemma}[Representation of UBM derivatives]\label{lem:der}
Consider a predictable bounded and continuous skew Hermitian valued process $(f_s)$ and set $F_t := \int_0^t f_s ds$. Then in $L^2(\mb{P})$ and almost surely,
\begin{equation}
\label{eq:frechet-ubm}
D_F U_t := \lim_{\eps \to 0} \eps^{-1} ( U(B+\eps F)_t - U(B)_t ) = \sqrt{2} \left( \int_0^t U_s f_s U_s^{-1} \rd s \right) U_t.
\end{equation}
\end{lemma}

\begin{proof}
First, we show that $V_t := D_F U_t$ exists and solves, in integral form,
$$
V_0 = 0, \qquad \rd  V_t =\sqrt{2} V_t \rd B_t  -  V_t \rd t  + \sqrt{2}U_t f_t \rd t.
$$
Indeed, with $U^{(\eps)} := U(B+\eps F)$ which solves
$$
\rd  U^{(\eps)}_t = \sqrt{2}U^{(\eps)}_t \rd (B_t + \eps F_t) - U^{(\eps)}_t \rd t =\sqrt{2} U^{(\eps)}_t \rd B_t  -  U^{(\eps)}_t \rd t + \eps \sqrt{2}U^{(\eps)}_t f_t \rd t 
$$
and $V^{(\eps)} := \eps^{-1} (U^{(\eps)}  - U)$, which satisfies $V^{(\eps)}_0 = 0$ and 
$$
\rd  V^{(\eps)}_t =\sqrt{2} V^{(\eps)}_t \rd B_t  - V^{(\eps)}_t \rd t +\sqrt{2} (U^{(\eps)}_t - U_t ) f_t \rd t + \sqrt{2}U_t f_t \rd t
$$
we obtain (by an $L^2$ estimate, Gronwall lemma and a continuity estimate), when $\eps \downarrow 0$,
\begin{equation}
\rd  V_t =\sqrt{2} V_t \rd B_t  -  V_t \rd t  + \sqrt{2}U_t f_t \rd t.
\end{equation}
Most importantly,  this equation has an explicit solution. Recalling that $\rd U_t = \sqrt{2}U_t \rd B_t - U_t \rd t$, taking the conjugate transpose and using that $\rd B_t$ is skew Hermitian, we have
$$
\rd U^{-1}_t = -\sqrt{2} \rd B_t U^{-1}_t-  U^{-1}_t \rd t.
$$ 
An application of It\^o's formula gives 
\begin{align*}
\rd V_t U^{-1}_t &= (\sqrt{2}V_t \rd B_t - V_t \rd t +\sqrt{2} U_t f_t \rd t) U_t^{-1} + V_t ( -\sqrt{2} \rd B_t U_t^{-1} -U_t^{-1} \rd t)  +2 V_t \rd B_t (- \rd B_t U^{-1}_t ) \\
& = \sqrt{2}U_t f_t U_t^{-1} \rd t
\end{align*}where we used $\rd B_t \rd B_t = - I$ to obtain the second equality, hence \eqref{eq:frechet-ubm}. 
\end{proof}

In the case of $1$d Brownian motion, the Cameron-Martin's formula implies, for deterministic shift $(f_t)$
\begin{align*}
\E \left( \int_0^1 f_s \rd B_s\cdot  \Phi (B)  \right) & =  \frac{\rd}{\rd \eps}_{|\eps = 0} \E \left( e^{\eps \int_0^1 f_s \rd B_s - \frac{\eps^2}{2} \int f_s^2 \rd s} \Phi (B) \right)  \\
& = \frac{\rd}{\rd \eps}_{|\eps = 0}  \int \Phi(B) e^{ - \frac{1}{2} \int_0^1 \rd (B_s - \eps F_s) \rd ( B_s - \eps F_s)} DB \\ 
& =  \frac{\rd}{\rd \eps}_{|\eps = 0} \E \left( \Phi ( B + \eps \int_0^{\cdot} f_s \rd s) \right) = \E ( D_F \Phi ( B))
\end{align*}
where $F = \int_0^{\cdot} f_s \rd s$. The calculation above is formal but can be made rigorous ($DB$ stands for the ``Lebesgue measure" on the space of continuous paths, which does not exist). The generalization to the Brownian motion on skew Hermitian matrices is straightforward and we have,  
\begin{align}
\int D_F \Phi (B) e^{ - \frac{N}{2} \int_0^1 \| \rd B_s \|_{\mathfrak{R}}^2} \mc{D}B  & = - \int  \Phi (B) D_F \left( e^{ - \frac{N}{2} \int_0^1 \| \rd B_s \|_{\mathfrak{R}}^2} \right) \mc{D}B \nonumber \\
&  = N \int \Phi(B) \int_0^1 \langle f_s, \rd B_s \rangle_{\mathfrak{R}}  e^{ - \frac{N}{2} \int_0^1 \| \rd B_s \|_{\mathfrak{R}}^2} \mc{D}B 
\end{align}
where $\mc{D}B$ formally stands for the Lebesgue measure on skew Hermitian valued continuous paths. The necessity of $N = \sigma^{-2}$  in the potential $V(B) = \frac{1}{2\sigma^2} \int_0^1 \| \rd B_s \|_{\mathfrak{R}}^2$ can be checked by computing, with $F_t = \ii t I$ and recalling \eqref{def:shbm},
$$
\sigma^2 N t = \sigma^2 \int_0^t \| F_s' \|_{\mathfrak{R}}^2 \rd s = \Var \int_0^t \langle \rd F_s, \rd B_s \rangle_{\mathfrak{R}}   = \Var \int_0^t \re ( \Tr (\overline{\ii I} \rd B_s)) = t.
$$
The Girsanov theorem gives an extension to predictable processes.
\begin{lemma}[Integration by parts for $(B_t)$]\label{lem:IBP} Consider a predictable bounded and continuous skew Hermitian valued process $(f_s)$ and set $F_t := \int_0^t f_s \rd s$. Suppose that $\Phi(B) \in L^2(\mb{P})$ is measurable with respect to $B$ and that $D_F\Phi(B)$ exists almost surely and in $L^2(\mb{P})$. Then, 
$$
\E \left[ D_F \Phi(B) \right] = N \E \left[ \Phi(B) \int_0^t \langle f_s, \rd B_s \rangle_{\mathfrak{R}} \right].
$$
\end{lemma}

\begin{proposition}[Integration by parts for $(U_t)$]
\label{prop:Girsanov}
With $F = \int_0^{\cdot} f_s ds$ and $\Phi$ as above, we have
\begin{equation}\label{eqn:gir1}
\E \left[ \Phi(B) \int_0^t \langle f_s, \rd B_s \rangle_{\mathfrak{R}} \right] = \frac{1}{N} \E \left[ D_F \Phi(B) \right], \qquad \text{and} \qquad D_F U_t = \sqrt{2} \left( \int_0^t U_s f_s U_s^{-1} \rd s \right) U_t.
\end{equation}
Furthermore, for a matrix valued bounded  and continuous predictable process $(h_s)$  (not necessarily skew Hermitian), and with a finite number of positive times $t_j$ and $\mathscr{C}^{1}$  functions $g_j$ on the unit circle, we have
$$
\E \left[  \int_0^t \Tr (h_s \rd B_s)  \prod_i e^{\Tr g_i(U_{t_i})} \right] =   - \frac{\sqrt{2}}{N} \sum_j \E \left[ \Tr \left( g_j'(U_{t_j}) U_{t_j} \int_0^{\min (t, t_j)} U_s h_s U_s^{-1}  \rd s \right) \prod_i e^{\Tr g_i(U_{t_i})} \right].
$$
\end{proposition}

\begin{proof}
The first statement is immediate from the lemmas \ref{lem:IBP} and \ref{lem:der}. For the second statement, 
we denote by $p_S$ (resp. $p_H$) the projection on skew Hermitian (resp. Hermitian) matrices. Since these spaces are orthogonal for $\langle \cdot, \cdot \rangle_{\mf{R}}$ and $dB$ is skew Hermitian,
$$
\re ( \Tr (h dB)) = - \re (\Tr (p_S(h)^* dB)) + \re \Tr (p_H(h) dB)  = - \langle p_S(h) , dB \rangle_{\mf{R}} + 0.
$$
Given that $\ii : M \mapsto \ii M$ maps Hermitian matrices to skew Hermitian ones,  and skew Hermitian matrices to Hermitian ones,  we have
$$
\im ( \Tr (h dB)) = \re ( - \ii \Tr(h dB)) = \re ( \Tr (- \ii p_S(h) dB)) + \re ( \Tr (- \ii p_H(h) dB))  = 0 + \langle \ii p_H(h), dB \rangle_{\mf{R}}.
$$
We suppose that the product $\prod_i$ reduces to one term since the generalization is straightforward. We have
\begin{multline*}
\E  \int_0^t \Tr (h_s \rd B_s)  e^{\Tr g(U_t)}  = \E  \int_0^t  \langle - p_S(h_s) , \rd B_s \rangle_{\mf{R}} e^{\Tr g(U_t)} + \ii \E  \int_0^t   \langle \ii p_H(h_s), \rd B_s \rangle_{\mf{R}}  e^{\Tr g(U_t)} \\
 = \frac{\sqrt{2}}{N}  \E \Tr \left( g'(U_t) U_t \int_0^t U_s (- p_S(h_s) ) U_s^{-1}  \rd s \right)e^{\Tr g(U_t)} + \frac{\sqrt{2} }{N} \ii  \E\Tr \left( g'(U_t) U_t \int_0^t U_s \ii p_H(h_s) U_s^{-1} \rd s \right) e^{\Tr g(U_t)}  \\
= -  \frac{\sqrt{2}}{N} \E \Tr \left( g'(U_t) U_t \int_0^t U_s h_s U_s^{-1} \rd s e^{\Tr g(U_t)} \right).
\end{multline*}
In the second equality, we used (\ref{eqn:gir1}) and the third equality follows from $-p_S(h) + \ii^2 p_H(h) = -h$.
\end{proof}

\subsection{Biased measures and error terms}. In this section, we will use the following a priori estimate by Johansson as an input.
\begin{lemma}[{\cite[Lemma 2.9]{Joh97}}]
\label{lem:Johansson}
If $f$ is real and $\|f\|_{{\rm H}}<\infty$, then
$$
\E\left[e^{{\rm Tr}f(U)}\right]\leq e^{N\hat f_0+ \frac{1}{2} \|f\|_{{\rm H}}^2}.
$$
\end{lemma}
\noindent  An immediate consequence with $f$ chosen as $g$ below (\ref{eqn:gaps}) is 
\begin{equation}\label{eqn:rigidity}
\mathbb{P}(\btheta(s)\in\mathscr{G})>1-e^{-(\log N)^{10}},
\end{equation}
where $\mscr{G}$ is defined as the rigidity event at the beginning of Section \ref{subsec:rig}.

Let $\ell^h$ be a regularization of $\log$ (around the singularity $h$) on scale $\iota>N^{-2}$.
Let 
\begin{equation}\label{eqn:D}
\mc{D}=\mc{D}_{J,\delta,C}
\end{equation}
 denote the family of laws biased by $ e^{\sum_{1 \leq j \leq J} \Tr f_j(U_{t_j})}$ where 
$J$ is fixed,
$f_j$ is either an element in  $\mathscr{S}_{\delta,C}$ (see Definition \ref{subsec:fctSpaces}) or $f_j=\lambda\, \ell^h$, $0\leq\lambda\leq C$ and $h\in[0,2\pi)$,
$t_j\in[0,C]$.

For any $\mb{P}$ in $\mc{D}$ we denote by $\mb{E}$ the expected value under $\mb{P}$,  and the dependence in $f_j,t_j$ will sometimes be emphasized trough  $\mathbb{P}_{\bf{f}}, \mathbb{E}_{\bf{f}}$. 

We will  use the following a priori estimates without systematically referencing them,  when transferring an estimate for a biased measure from the Haar measure: Under the Haar measure there exists $C'$ such that, uniformly in $f\in\mathscr{S}_{\delta,C}$ and $0\leq \lambda\leq C$, we have
\begin{equation}\label{eqn:polybound}
\E[e^{{\rm Tr} f-\E({\rm Tr}f)}]\leq e^{C'(\log N)^2},\ \E[e^{\lambda \rm Tr \ell^{h}-\lambda \E({\rm Tr}\ell^{h})}]\leq e^{C'\log N}.
\end{equation}
Both inequalities follows from Lemma \ref{lem:Johansson} and $\|f\|_{{\rm H}}^2\leq C(\log N)^2$ (due to our assumption $f\in\mathscr{S}_{\delta,C}$), $\|\ell^{h}\|_{{\rm H}}^2\leq C\log N$ (see Lemma \ref{lem:regularity}).  As an example of application,  Equations  (\ref{eqn:rigidity}),  (\ref{eqn:polybound}),  H\"older and Jensen inequalities imply $\mathbb{P}_{\bf f}(\btheta(s)\in\mathscr{G})>1-e^{-(\log N)^5}$.

\medskip

For the following lemma,  we recall the notation $d(v,w)=\max(|v-w|,|v-\frac{w}{|w|^2}|)$ and  $R = (\log N)^{1+c}$.

\begin{lemma}[Application of the full rank projection estimate]
\label{lem:loclaw-rig}
Let $\eps \in (0,1)$ be arbitrary. Uniformly in $\eta_v,\eta_w\in[0,1/2]$,  $\mb{P}_{\bf f} \in \mc{D}$,  $-R \leq s \leq C$ and $\max (0,s) \leq t \leq C$, we have
\begin{multline*}
\frac{1}{N}\E\left[{\rm Tr}\left(
\frac{v+U_s}{v-U_s}\frac{U_t}{(w-U_t)^2}
\right)\right]=\frac{1}{N}\frac{w_{t-s}}{w}\E\left[
{\rm Tr}\left(
\frac{v+U_s}{v-U_s}\frac{U_{s}}{(w_{t-s}-U_s)^2}
\right)  \right]\\
+
\frac{\OO(N^\varepsilon)(1+|w_{t-s}|)}{ \eta_w\,d(v,w_{t-s})\,\min(1,N\eta_v)}(\frac{1}{\eta_w}+\frac{1}{\sqrt{\eta_w\min(\eta_{w_{t-s}},\eta_v)}})
\end{multline*}
\end{lemma}

\begin{proof}
We first prove the result under the unbiased measure, i.e.  ${\bf f}=0$. We have
$$
\frac{1}{N}\E\left[{\rm Tr}\left(
\frac{v+U_s}{v-U_s}\frac{U_t}{(w-U_t)^2}
\right) \right]=-\frac{1}{2N}\partial_w\E\left[{\rm Tr}\left(
\frac{v+U_s}{v-U_s}\frac{w+U_t}{w-U_t}
\right) \right].
$$
With the Cauchy formula, under the event $E$ from Proposition \ref{prop:avIso} we have
\begin{multline*}
\left|\frac{1}{N}\E\left[{\rm Tr}\left(
\frac{v+U_s}{v-U_s}\frac{U_t}{(w-U_t)^2}
\right)  \mathds{1}_E-\frac{w_{t-s}}{w}
{\rm Tr}\left(
\frac{v+U_s}{v-U_s}\frac{U_{s}}{(w_{t-s}-U_s)^2}
\right) \mathds{1}_E \mid \btheta(s)\in\mathscr{G} \right]\right|\\
\leq
\frac{N^\varepsilon(1+|w_{t-s}|)}{\eta_w\,d(v,w_{t-s})\,\min(1,N\eta_v)}(\frac{1}{\eta_w}+\frac{1}{\sqrt{\eta_w\min(\eta_{w_{t-s}},\eta_v)}}).
\end{multline*}
%
%
Furthermore, the same result holds without $\mathds{1}_E$ nor the conditioning, by using the trivial estimate on the integrand (e.g.  $N^6/(\eta_v\eta_w^2)$), the rigidity estimate (\ref{eqn:rigidity}) and  the isotropic law from Proposition  \ref{prop:avIso}: $\mb{P} (E \mid \btheta(s)\in\mathscr{G})\geq 1-e^{-(\log N)^D}$.  This completes the proof for the equilibrium measure. 

We now consider $\mb{P}_{{\bf f}} \in \mc{D}$.  On the event $E$, the same estimates hold for $\mathbb{P}_{{\bf f}}(\btheta(s)\in\mathscr{G})$ and $\mathbb{P}_{{\bf f}}(E \mid \btheta(s)\in\mathscr{G})$,   as explained before the statement of the lemma, so the above proof applies to the biased measures.
\end{proof}

\subsection{Asymptotics of the loop equations}. The following Lemma will be the main tool for the ``gluing'' operation mentioned in subsection \ref{outline}. It relies on the integration by parts formula from Proposition \ref{prop:Girsanov}, the consequences of the local law and rigidity estimates for biased measures Lemma \ref{lem:loclaw-rig}, the first section in the appendix to express our result in terms of Fourier coefficients, and various smoothings.

\begin{lemma}
\label{lem:loopEqn}
Let $C>0$ and $\delta\in(0,1)$ be arbitrary.  Consider $h_r\in \mathscr{S}_{\delta,C}$   (see Definition \ref{subsec:fctSpaces})
where $r\in I$, $I$ a set of at most $C$ times in $[0,C]$, possibly $N$-dependent.
Let $\mb{P}_{\bf{f}}$ denote the law of the unitary Brownian motion at equilibrium  biased by $ e^{\sum_{t\in\mathscr{B}} \Tr f_t(U_{t})+\sum_{z=t+\ii\theta\in\mathscr{A}}\gamma_z \Tr \ell^\theta_+(U_{t})}$ where $\mathscr{A}$,  $\mathscr{B}$ and $f_s$
are as defined in Theorem \ref{thm:FH} and $\ell_{+}$ is as defined in Definition \ref{def:SmoothLog} with regularization scale $N^{-1-\alpha}$, $\alpha=\delta/6$.

We also assume that 
for any $t+\ii\theta\in\mathscr{A}$ we have (see \eqref{eqn:Hilbert} for the definition of ${\rm H}$)
\begin{equation}\label{eqn:assumed}
\sum_{r\in I}{\rm H}\,{\rm P}_{|r-t|} h_r(e^{\ii\theta})=0.
\end{equation}
Then for any small $\e>0$ we have
\begin{multline}\label{eqn:loopEqn}
\sum_{r\in I}\E_{\bf{f}} \left[ \Tr h_r(U_{r}) -N \dashint h_r\right]  =  \sum_{r\in I,t+\ii\theta\in\mathscr{A},k\in\mathbb{Z}}  e^{- |k| |t-r|}  |k| {(\widehat{\ell_+^\theta})}_{k} \hat{(h_r)}_{-k}  +
\sum_{r\in I,t\in\mathscr{B},k\in\mathbb{Z}}  e^{- |k| |t-r|}  |k| {{\hat{(f_t)}}}_{k} \hat{(h_r)}_{-k}  \\
+ \OO_{\delta,C}(N^{-\delta/4+\e}).
\end{multline}
\end{lemma}

\begin{remark}
In the above statement, every $\ell_{+}$ can be replaced by $\ell_{-}$ without any change to the proof.
\end{remark}

\begin{proof} 
To simplify the notations we assume that $\mathscr{B}=\varnothing$,  as the functions $f_s$ 
are more regular that the logarithmic singularities $\ell_+^\theta$, which represent the key difficulty. The contribution of the $f_s$'s can be included following the method below with only notational changes.
We also abbreviate $\ell_+$ into $\ell$ along the proof.\\

\noindent{\it First step: integral representation}. 
First,  from (\ref{eqn:afterChar}) we have
$$
m_r(w)=  m_{-R}(w_{r+R}) + \int_{-R}^r\frac{2\sqrt{2}}{N}{\rm Tr}\left(\frac{ w_{r-s} U_s}{(w_{r-s}-U_s)^2}\rd B(s)\right) +\int_{-R}^r(m_s(w_{r-s})-s(w))w_{r-s}\partial_z m_s(w_{r-s})\rd s.
$$
Remember the decomposition of $h_r$ into $O(\log N)$ many functions from the sets ${\rm A}_{\xi,C}$,   $(\xi\in[\delta,1])$. We start evaluating the trace of a function $h$ from $\rm A_{\xi,C}$ for an arbitrary $\xi\in[\delta,1]$. Together with the representation (\ref{eqn:HS-rep}),  this implies
 (remember we denote   $s(z)=\mathds{1}_{|z|>1}-\mathds{1}_{|z|<1}$)
that for any $h$
\begin{align}
\label{eqn:start}
 \E_{\bf{f}}( \Tr h(U_r) )- N \dashint h&= -\frac{N}{2\pi}
\int\partial_{\bar w}\tilde h(w)\cdot \E_{\bf{f}} [ m_{-R}(w_{r+R})-s(w) ] \frac{\rd m(w)}{w}\\
-
& \frac{N}{2\pi}\label{eqn:start2}
\int\partial_{\bar w}\tilde h(w)\cdot  \int_{-R}^r \E_{\bf{f}} \left[ (m_s(w_{r-s})-s(w))w_{r-s}\partial_z m_s(w_{r-s}) \right] \rd s  \frac{\rd m(w)}{w}
\\ 
\label{eqn:start3}
& -\frac{2\sqrt{2}}{2\pi}
\int\partial_{\bar w}\tilde h(w)\cdot \E_{\bf{f}} \int_{-R}^r \Tr \left(\frac{ w_{r-s} U_s}{(w_{r-s}-U_s)^2}\rd B(s)\right)  \frac{\rd m(w)}{w}
\end{align}
where we have chosen the second order quasi-analytic extension $\tilde h$ (contrary to the first order in (\ref{eqn:dbar})) and $N^{-1+\xi}$ as the scale of the associated bump function $\chi$ in 
\begin{align} \label{eqn:2nd_order_quasi}
\tilde{h}(re^{i\theta})=\left(h(e^{i\theta})-\ii h'(e^{i\theta})\log r - h^{\prime\prime}(e^{i\theta})\frac{(\log r)^2}{2}\right)\chi(r).
\end{align}
 The first term (\ref{eqn:start}) above is easily shown to be subpolynomial in $N$ because $w_{r+R}$ is either superpolynomially large or close to $0$. The second term (\ref{eqn:start2}) is also negligible by the bound 
\begin{equation}\label{eqn:HS3}
|\partial_{\bar w}\tilde h|\leq (|h|+\eta_w|h'|+\eta_{w}^{2}|h^{\prime\prime}|)\cdot |\chi'|+|h^{\prime\prime\prime}|\eta_w^2 \cdot|\chi|.
\end{equation}
Indeed denoting $\eta_0=N^{-1+\xi}$, with (\ref{eqn:proba2}) we obtain
\[
|\eqref{eqn:start2}|\lesssim N^{1+\e}\iint_{[0,\eta_0]^2}\left((1+\frac{\eta}{\eta_0}+\frac{\eta^2}{\eta_0^2})\cdot \frac{\mathds{1}_{\eta_0/2<\eta}}{\eta_0}+\frac{\eta^2}{\eta_0^3} \right)\cdot\int_{0}^\infty\frac{1}{N(\eta+s)}\frac{1}{N(\eta+s)^2}\rd s\rd\eta\rd\theta\lesssim \frac{N^{\varepsilon}}{N\eta_0}.
\]

\noindent{\it Second step: injecting the integration by parts formula}. To evaluate (\ref{eqn:start3}) we rely on the integration by parts formula from Proposition \ref{prop:Girsanov}:
\begin{equation}\label{eqn:inter11}
\E_{\bf{f}}\left[
\int_{-R}^r \Tr \left(\frac{ w_{r-s} U_s}{(w_{r-s}-U_s)^2}\rd B(s)\right) 
\right]
= -  \frac{1}{N} \sum_{t+\ii\theta\in\mathscr{A}}
\E_{\bf{f}}\left[{\rm Tr}\left(
{\ell^{\theta}}'(U_{t}) V_r^{t}(w)
\right)\right]
\end{equation}
where, for $1 \leq j \leq N$,
$$
V_r^{t}(w)=\sqrt{2}\int_{-R}^{t \wedge r} U_s \frac{w_{r-s} U_s}{(w_{r-s}-U_s)^2} U_s^{-1 }ds \cdot U_{t}=\sqrt{2}\int_{-R}^{t \wedge r} \frac{w_{r-s} U_s}{(w_{r-s}-U_s)^2}\rd s\cdot U_{t}.
$$

 We can write ${\ell^\theta}=\sum g_m$ where the sum is over $\OO(\Delta)$ terms (recall $\Delta$ in \eqref{eq:phi-delta}) and $g_m$ supported on an arc of length  $1/(Ne^m)$, $-\log N\leq m\leq \alpha\log N$,   $\sum_{k=0}^3 (Ne^{m})^k\|g_m^{(k)}\|_\infty\leq C\log N$.  
The number of considered $g_m$'s is $\OO((\log N)^2)$ thanks to the initial smoothing on scale $N^{-1-\alpha}$ in Section \ref{sec:subsmooth}. We can therefore assume, until further notice, without loss of generality that ${\ell^\theta}$ term inside the $\Tr$ coincides with such a $g_m$, and we define $\tilde\varepsilon=e^{-m}/N$.

  From (\ref{eqn:HS-rep}) ($|v|=1$ in $v {{\ell^\theta}}'(v)$ below) we can write
\begin{align}
\frac{1}{N}\E_{\bf{f}}\left[{\rm Tr}\left(
 {{\ell^\theta}}'(U_{t}) V_r^{t}(w))
\right)\right]& =\frac{\sqrt{2}}{N}\E_{\bf{f}}\left[{\rm Tr}\left(
U_{t} {{\ell^\theta}}'(U_{t})\int_{-R}^{t \wedge r} \frac{w_{r-s} U_s}{(w_{r-s}-U_s)^2}\rd s
\right)\right]\notag\\
& =
-\frac{\sqrt{2}}{2 \pi N}\int_{\mathbb{C}}
\int_{-R}^{t \wedge r}\E_{\bf{f}}{\rm Tr}\left(
\frac{v+U_{t}}{v-U_{t}}\frac{w_{r-s} U_s}{(w_{r-s}-U_s)^2}
\right)\rd s
\partial_{\bar v}\widetilde{ v {{\ell^\theta}}'}(v)\frac{\rd m(v)}{v},\label{eqn:interm100}
\end{align}
where for further error estimates it will be pertinent to choose $c=\tilde\varepsilon$ for $\chi_c$ in the definition of the above $\widetilde{v{{\ell^\theta}}'}$, here we use the first order quasi-analytic extension given in \eqref{eqn:HS-exp}.

We first use reversibility in the above expectation. 
In this way, instead of changing the very singular function $\frac{v+U_{t}}{v-U_{t}}$ into 
$\frac{v_{t-s}+U_{s}}{v_{t-s}-U_{s}}$ and collecting a problematic error term in $\eta_v^{-1}$
(the $v$ variable corresponds to the singular function ${\ell^\theta}$, while 
the $w$ variable corresponds to the smooth $h$), we will change the more regular function in $w$ and collect 
an error term $\eta_w^{-1}$, as in Lemma \ref{lem:loclaw-rig}.

More precisely, by reversibility the above expectation is
\begin{equation}\label{eqn:interm101}
\frac{\E e^{\sum_{t'+\ii\theta\in\mathscr{A}} \Tr \ell^\theta(U_{t'})}{\rm Tr}\left(
\frac{v+U_{t}}{v-U_{t}}\frac{w_{r-s} U_s}{(w_{r-s}-U_s)^2}
\right)}{\E e^{\sum_{t'+\ii\theta\in\mathscr{A}} \Tr \ell^\theta(U_{t'})}}
=
\frac{\E e^{\sum_{t'+\ii\theta\in\mathscr{A}} \Tr \ell^\theta(U_{t-t'})}{\rm Tr}\left(
\frac{v+U_{0}}{v-U_{0}}\frac{w_{r-s} U_{t-s}}{(w_{r-s}-U_{t-s})^2}
\right)}{\E e^{\sum_{t'+\ii\theta\in\mathscr{A}} \Tr \ell^\theta(U_{t-t'})}}.
\end{equation}
We denote the above right-hand side as  a biased measure $\E_{j,\bm{t}}{\rm Tr}\left(
\frac{v+U_{0}}{v-U_{0}}\frac{w_{r-s} U_{t-s}}{(w_{r-s}-U_{t-s})^2}
\right)$.  

With (\ref{eqn:start3}),  (\ref{eqn:inter11}), (\ref{eqn:interm100}) and (\ref{eqn:interm101}), we have proved that
$ \E_{\bf{f}}( \Tr h(U_r) )- N \dashint h$ is equal to
\[
  -\frac{1}{N\pi^2}\sum_{t+\ii\theta\in\mathscr{B}}
\iint\partial_{\bar w}\tilde h(w)\cdot \partial_{\bar v}\widetilde{ v {{\ell^\theta}}'}(v)\
\int_{-R}^{t \wedge r}
\E_{j,\bm{t}}{\rm Tr}\left(
\frac{v+U_{0}}{v-U_{0}}\frac{w_{r-s} U_{t-s}}{(w_{r-s}-U_{t-s})^2}
\right)
\rd s\frac{\rd m(v)}{v}
  \frac{\rd m(w)}{w}+\OO(N^{-\delta+\varepsilon}).
\]

\noindent {\it Third step: injecting resolvent estimates}. 
With the key Lemma \ref{lem:loclaw-rig} in the above expectation,  noting that $\frac{w_{t-s}}{w}w_{r-s}=w_{r+t-2s}$ we obtain
\begin{align*}
&\E_{\bf{f}}( \Tr h(U_r) )- N \dashint h= \sum_{t+\ii\theta\in\mathscr{A}} A_r(t+\ii\theta)+\OO(N^\e\sum_{t+\ii\theta\in\mathscr{B}} \mathcal{E}_r(t+\ii\theta))+\OO(N^{-\delta+\e}),\\
&A_r(t+\ii\theta):=-\frac{1}{N\pi^2}\iint\partial_{\bar w}\tilde h(w)
\int_{-R}^{t\wedge r}
\E_{j,\bm{t}}{\rm Tr}\left(
\frac{v+U_{0}}{v-U_{0}}\frac{w_{r+t-2s} U_{0}}{(w_{r+t-2s}-U_{0})^2}
\right)
\rd s
\partial_{\bar v}\widetilde{ v {{\ell^\theta}}'}(v)\frac{\rd m(v)}{v}\frac{\rd m(w)}{w},\\
&\mathcal{E}_r(t+\ii\theta):=\frac{1}{N}\iint|\partial_{\bar w}\tilde h(w)|
\int_{-R}^{t\wedge r}
\frac{(1+|w_{r+t-2s}|)}{ \eta_{w_{r-s}}\,d(v,w_{r+t-2s})\,\min(1,N\eta_v)}\\
&\hspace{5cm}\times(\frac{1}{\eta_{w_{r-s}}}+\frac{1}{\sqrt{\eta_{w_{r-s}}\min(\eta_{w_{r+t-2s}},\eta_v)}})
\rd s
|\partial_{\bar v}\widetilde{ v {{\ell^\theta}}'}(v)|\frac{\rd m(v)}{|v|}\frac{\rd m(w)}{|w|}.
\end{align*}
We start with the evaluation of the main terms,  $A_r(t+\ii\theta)$.
Importantly the biased measures $\E_{j,\bm{t}}$ don't depend on $s$, so that 
for any $D>0$,  uniformly in $w$ we can integrate
\[
\int_{-R}^{t\wedge r}
\frac{w_{r+t-2s} U_{0}}{(w_{r+t-2s}-U_{0})^2}
\rd s=\int_{-\infty}^{t\wedge r}
\frac{w_{r+t-2s} U_{0}}{(w_{r+t-2s}-U_{0})^2}
\rd s+\OO(N^{-D})
=
-\frac{1}{2}\frac{U_0}{w_{|r-t|}-U_0}\cdot(\mathds{1}_{|w|>1}-\mathds{1}_{|w|<1})
+\OO(N^{-D}).
\]
Noting that $\frac{U}{w-U}=\tfrac{1}{2}\frac{w+U}{w-U}-\tfrac{1}{2}$ and $\int_{\mathbb{C}}\partial_{\bar w}\tilde h(w)
(\mathds{1}_{|w|>1}-\mathds{1}_{|w|<1})
\frac{\rd m(w)}{w}=0$ (this follows from (\ref{eqn:upper}) and (\ref{eqn:lowerbis}) when $z=0$ or $z\to+\infty$),  we have obtained
\[
A_r(t+\ii\theta)=\frac{1}{4N\pi^2}
\E_{j,\bm{t}}{\rm Tr}\left[\left(
\int_{\mathbb{C}}\partial_{\bar w}\tilde h(w)
\frac{w_{|r-t|}+U_0}{w_{|r-t|}-U_0}\cdot(\mathds{1}_{|w|>1}-\mathds{1}_{|w|<1})
\frac{\rd m(w)}{w}\right)\left(
\int_{\mathbb{C}}
\partial_{\bar v}\widetilde{ v {{\ell^\theta}}'}(v)\frac{v+U_{0}}{v-U_{0}}\frac{\rd m(v)}{v}\right)
\right]+\OO(N^{-D}).
\]
The first parenthesis is equal to ${\rm H}\,{\rm P}_{|r-t|} h(U_0)$ from (\ref{eqn:HilbertPoisson}),  and the second one is
simply $U_0{{\ell^\theta}}'(U_0)$ from (\ref{eqn:HS-rep}). This gives
\[
\sum_{r\in I,t+\ii\theta\in\mathscr{A}} A_r(t+\ii\theta)=
\frac{1}{N}\E_{j,\bm{t}}
\sum_{t+\ii\theta\in\mathscr{A}}{\rm Tr}\left(\sum_{r\in I}{\rm H}\,{\rm P}_{|r-t|} h(U_0)
U_0 {{\ell^\theta}}'(U_0)\right)+\OO(N^{-D}).
\]
We now return to the original regularized logarithm functions $\ell^{\theta}$'s and the functions $h_r$'s by summing over all components of the decompositions. Because the logarithmic singularity $1/x$ of ${\ell^\theta}'$ (up to scale $N^{-1-\alpha}$) is compensated by the vanishing assumption \eqref{eqn:assumed}, we claim that
\begin{multline} \label{eq:hlbrt_sing}
\frac{1}{N}\E_{j,\bm{t}}
\sum_{t+\ii\theta\in\mathscr{A}}{\rm Tr}\left(\sum_{r\in I}{\rm H}\,{\rm P}_{|r-t|} h_r(U_0)
U_0 {{\ell^\theta}}'(U_0)\right)=\sum_{t+\ii\theta\in\mathscr{A}}\frac{1}{2\pi}\int_{0}^{2\pi}\sum_{r\in I} {\rm H}\,{\rm P}_{|r-t|} h_r(e^{\ii\omega}) e^{\ii\omega} {{\ell^\theta}}'(e^{\ii\omega})\rd \omega +\OO(N^{-\delta+\varepsilon}) 
\\
=\sum_{r\in I,t+\ii\theta\in\mathscr{A},k\in\mathbb{Z}}  e^{- |k| |t-r|}  |k| {\widehat{\ell^\theta}}_{k} \hat{(h_r)}_{-k} +\OO(N^{-\delta+\varepsilon}) 
\end{multline}
where the second equality follows by equations \eqref{eqn:PoissonK} and \eqref{eqn:Hilbert}.
For the first equality,  
denoting $F(\omega)=\sum_r{\rm H}\,{\rm P}_{|r-t|} h_r(e^{\ii\omega})e^{\ii\omega} {{\ell^\theta}}'(e^{\ii\omega})$
by rigidiy of the eigenangles we first have $|N^{-1}\sum F(\theta_k)-\int F|\lesssim N^{-1+\e}\int|F'|$ with overwhelming probability, so we just need to bound $\int|F'|$. Writing $G(\omega)=\sum_r{\rm H}\,{\rm P}_{|r-t|} h_r(e^{\ii\omega})e^{\ii\omega}$ and using the key assumption (\ref{eqn:assumed}), we have (assuming $\theta=0$ without loss of generality)
\[
\int|F'|\lesssim \int |(G(\omega)-G(0))\ell^{''}(\omega)+G'(\omega)\ell^{'}(\omega)|\rd\omega\lesssim \int |(G(\omega)-G(0)-\omega G'(\omega))\ell^{''}(\omega)|\rd\omega+\int |G'(\omega)(\ell^{'}(\omega)+\omega\ell''(\omega))|\rd\omega
\]
Using Taylor with integral remainder and  the bounds $|\ell''|\lesssim |\omega|^{-2}$,  $|\ell^{'}(\omega)+\omega\ell''(\omega)|\lesssim (|\omega|+N^{-1-\alpha})^{-1}$,  we obtain
$
\int|F'|\lesssim \int|G''|+|G'(0)|\log N.
$
Note that $\left|\left({\rm H}\,{\rm P}_{|r-t|} h_r\right)'\right|\leq \sum_{k\in\Z}|k\hat{(h_r)}_k|\leq (\log N)^2N^{1-\delta}$ as $h_r\in\mathscr{S}_{\delta,C}$ so $G'(0)\leq (\log N)^2N^{1-\delta}$. 
Moreover,  from  the LHS representation of the Hilbert transform in (\ref{eqn:Hilbert}) we easily obtain $\int|{\rm H}f|\lesssim (\log N)\int|f|+N^{-10}\sup|f'|$ for any function $f$,  so 
 $\int |G''|\lesssim (\log N)\sum_r\int |{\rm P}_{r-t}h_r''|\lesssim (\log N)\sum_r\int |h_r''|\lesssim (\log N)^2 N^{1-\delta}$.
This concludes the proof of  \eqref{eq:hlbrt_sing}.

We now estimate the error term $\mathcal{E}_r(t+\ii\theta)$,
\begin{align*} 
&\mathcal{E}_r(t+\ii\theta)\leq \frac{1}{N}\iint|\partial_{\bar w}\tilde h(w)| \frac{1}{\eta_{w}}(1+\frac{1}{N\eta_{v}}) (\frac{1}{\eta_{w}}+\frac{1}{\sqrt{\eta_{w}\eta_v}})|\partial_{\bar v}\widetilde{ v {{\ell^\theta}}'}(v)|\int_{-R}^{t\wedge r}\frac{(1+|w_{r+t-2s}|)}{ d(v,w_{r+t-2s})}\rd s \rd m(v)\rd m(w).
\end{align*}
Note that $\int_{-R}^{t\wedge r}\frac{(1+|w_{r+t-2s}|)}{ d(v,w_{r+t-2s})}\rd s\leq \varphi \,|\log d(v,w)|$. The contribution from $d(v,w)\leq N^{-10}$ is negligible (e.g. it is $\OO(N^{-3})$) by volume estimate.  Hence,
\[
\mathcal{E}_r(t+\ii\theta)\leq \frac{\varphi^2}{N}\iint|\partial_{\bar w}\tilde h(w)| \frac{1}{\eta_{w}}(1+\frac{1}{N\eta_{v}}) (\frac{1}{\eta_{w}}+\frac{1}{\sqrt{\eta_{w}\eta_v}})|\partial_{\bar v}\widetilde{ v {{\ell^\theta}}'}(v)| \rd m(v)\rd m(w)+\OO(N^{-3}).
\]
From (\ref{eqn:HS3}),
denoting $N^{-1+\xi}$ by $\epsilon$, as $\chi$ is supported on $\exp([-2\varepsilon,2\varepsilon])$, constant equal to $1$ on $\exp([-\varepsilon,\varepsilon])$, we obtain (for some points $a,b$ on the unit circle) that $|\partial_{\bar w}\tilde h|\leq \eta_w^2\varepsilon^{-3}\mathds{1}_{|w-a|<4\varepsilon}$, and similarly $|\partial_{\bar v}\widetilde{ v{\ell^\theta}'}|\leq \eta_v{\tilde\varepsilon}^{-3}\mathds{1}_{|v-b|<4\tilde\varepsilon}$ (remember $\tilde h$ is defined from the second order expansion (\ref{eqn:2nd_order_quasi}) and $\widetilde{ v{\ell^\theta}'}$ from the first order (\ref{eqn:HS-exp})).  Substituting these, we obtain
\begin{align*} 
\int \frac{|\partial_{\bar w}\tilde h(w)|}{\eta_{w}^c}\rd m(w)\leq \frac{1}{\epsilon^{c-1}},\ \textnormal{for }c<3; \quad \int \frac{|\partial_{\bar v}\widetilde{ v {{\ell^\theta}}'}(v)|}{\eta_{v}^{c}}\rd m(v)\leq \frac{1}{\tilde\varepsilon^{c}},\ \textnormal{for }c<2.
\end{align*}
Therefore,
\begin{align*} 
\mathcal{E}_r(t+\ii\theta)\leq \frac{\varphi^2}{N}\left(\frac{1}{\epsilon^{1/2}\tilde{\epsilon}^{1/2}}+\frac{1}{\epsilon}\right)+\frac{\varphi^2}{N^2}\left(\frac{1}{\epsilon^{1/2}\tilde{\epsilon}^{3/2}}+\frac{1}{\epsilon\tilde{\epsilon}}\right)\leq N^{-\delta/4}
\end{align*}
where we used $\alpha=\delta/6$ in the last inequality.
Combining the estimates for  $A_r$ and $\mc{E}_r$ concludes the proof.\end{proof}

\section{Proof of the theorems}

\label{sec:proofmain}

\subsection{Theorem \ref{thm:FH}}. \label{subsec:proofmain}\ 
The local decoupling (Section \ref{sec:DPP}) and the asymptotics of the loop equations (Section \ref{sec:loop}) allow to prove Theorem \ref{thm:FH} through the following surgery.  The Selberg  formula is a base point.

\begin{lemma}[One singularity]
\label{lem:one-sing}
 For any $\theta \in [0,2\pi]$, as $N \to \infty$,
$$
\E( | \det (U_0 - e^{\ii \theta})|^{\gamma})  = N^{ \frac{\gamma^2}{4}} \frac{{\rm G}(1+\frac{\gamma}{2})^2}{{\rm G}(1+\gamma)} (1 + \OO(1/N)).
$$
\end{lemma}
\begin{proof} We use the exact expression of the expected value of powers of the characteristic polynomials derived by Keating and Snaith \cite[(6)]{KeaSna2000} (and based on Weyl's and  Selberg's formulas) to calculate
\begin{align*}
\lim_{N \to \infty} N^{-\frac{\gamma^2}{4}} \E | \det (U_0 - e^{\ii \theta} )|^{\gamma}  & = \lim_{N \to \infty} N^{-\frac{\gamma^2}{4}} \prod_{j=1}^N \frac{\Gamma(j) \Gamma(j+\gamma)}{\Gamma(j+\frac{\gamma}{2})^2} \\
& = \lim_{N \to \infty} N^{-\frac{\gamma^2}{4}} \frac{{\rm G}(N+1) {\rm G}(N+1+\gamma)}{{\rm G}(N+1+\frac{\gamma}{2})^2} \frac{{\rm G}(1+\frac{\gamma}{2})^2}{{\rm G}(1) {\rm G}(1+\gamma)} =   \frac{{\rm G}(1+\frac{\gamma}{2})^2}{{\rm G}(1+\gamma)}.
\end{align*}
The second equality followed from the relation ${\rm G}(z+1) = \Gamma(z) {\rm G}(z)$ and the last one from ${\rm G}(1) = 1$ and the following asymptotics  (see, e.g., Barnes' original paper on the G function \cite[page 269]{Bar1900}):
$$
\log {\rm G}(z+1) = \frac{z^2}{2} \log z - \frac{3z^2}{4} + \frac{z}{2} \log 2\pi - \frac{1}{12} \log z + C + \OO_{z \to \infty}( \frac{1}{z})
$$
Indeed, it gives $\log {\rm G}(N+\gamma+1)  = \log {\rm G}(N+1) + \gamma N \log N - \gamma N  + \frac{\gamma^2}{2} \log N  +  \OO(\frac{1}{N})$ hence only the quadratic term in $\gamma$ contributes to $\log \frac{{\rm G}(N+1) {\rm G}(N+1+\gamma)}{{\rm G}(N+1+\frac{\gamma}{2})^2}  = \frac{\gamma^2}{4} \log N  + \OO( \frac{1}{N} )$.
\end{proof}
In what follows, we use the notations $f_t$ to denote the pair $(f,t)$ where $f$ is a function and $t$ a real number and we set for any $s,t$,
\begin{equation}
\label{eq:def-opera}
\mscr{C}(f_s,g_t) := \lim_{N \to \infty} \Cov ( \Tr f(U_s), \Tr g(U_t)) = \sum_{k \in \mb{Z}} |k| \hat{f}_k \hat{g}_{-k} e^{-|k| |t-s|} = (f, P_{|t-s|} g)_{ \rm H}.
\end{equation}
We extend it to finite linear combination, $\mscr{C}(f_s,\lambda g_t+ h_r) = \lambda \mscr{C}(f_s, g_t) +  \mscr{C}(f_s, h_r)$ and set $\mscr{C}(f_s) := \mscr{C}(f_s,f_s)$, which does not depend on $s$. With $L^x := \gamma_x \log |e^{\ii x} - \cdot |$ and $t>0$, we record  the following identities, obtained by using  $\widehat{(f_x)}_k = - \frac{e^{- \ii k x}}{2|k|}$ where $f_x(\theta) = \log |e^{\ii x} - e^{\ii \theta}|$,
\begin{align}
\mscr{C}(f_0, L_t^x) & = \gamma_x \sum_{k \neq 0} |k| \hat{f}_k \cdot ( - \frac{e^{\ii k x}}{2|k|} 1_{k \neq 0} ) e^{- |k| |t|}  = - \frac{\gamma_x}{2} ( {\rm P}_{t}f(e^{\ii x}) - \dashint f )   = - \frac{\gamma_x}{2} ({\rm P}_{t}-{\rm P}_{\infty}) f (e^{\ii x}), \label{eq:cov-id-1}\\
\mscr{C}(L_0^x, L_t^y) & = \gamma_x \gamma_y \sum_{k \neq 0} |k| \frac{e^{\ii k x}}{2|k|} \frac{e^{-\ii k y}}{2|k|} e^{- |k| t}    =   \frac{\gamma_x \gamma_y}{2} \sum_{k \geq 1} \frac{\cos(k(x-y))}{k} e^{- k t}  = \gamma_x \gamma_y {\rm P}_t C(x-y), \label{eq:cov-id-2}
\end{align}
where the function $C$ is defined in (\ref{def:cov-C}).

\begin{lemma}[One singularity $\&$ one smooth function] Let $t\geq 0$ and $f \in \mscr{S}_{\delta,C}$ for $\delta\in (0,1)$, then  
$$
\E( | \det (U_t - e^{\ii x}) |^{\gamma_x} e^{\Tr f(U_0)})  = N^{ \frac{\gamma_x^2}{4}} \frac{{\rm G}(1+\frac{\gamma_x}{2})^2}{{\rm G}(1+\gamma_x)} e^{ \mscr{C} ( f_0, L_t^{x})} e^{N \dashint f_0 + \frac{1}{2} \mscr{C}(f_0)} (1+\OO(N^{-\delta/9}) ).
$$
\end{lemma}
\begin{proof}
Without loss of generality, we suppose $\dashint f =0$. We start with replacing the logarithmic singularity with $\ell_+^{e^{\ii x}}$ (recall Definition \ref{def:SmoothLog}, here we set the submicroscopic smoothing parameter $\alpha=\delta/15$) and for simplicity we write $L_+^x=\gamma_x\ell_+^{e^{\ii x}}$. Define a function $q(e^{i\omega})=\chi_r(\omega)-\chi_l(\omega)$ where $\chi_r$ (resp. $\chi_l$) is a smooth bump function that is equal to $1$ on $[x+N^{-1+9\delta/20},x+N^{-1+\delta/2}]$ (resp. $[x-N^{-1+\delta/2},x-N^{-1+9\delta/20}]$) and supported in $[x+N^{-1+9\delta/20}/2,x+2N^{-1+\delta/2}]$ (resp.  $[x-2N^{-1+\delta/2},x-N^{-1+9\delta/20}/2]$). By the principal value definition of Hilbert transform from \eqref{eqn:Hilbert}, it's easy to see that ${\rm H} q (e^{ix})\asymp\log N$. On the other hand using the Fourier space definition of Hilbert transform from \eqref{eqn:Hilbert},  ${\rm H P}_tf=\OO (\log N)$. Hence, there exist an $\OO(1)$ constant $\alpha$ such that ${\rm HP}_tf(e^{\ii x})-\alpha{\rm H}q(e^{\ii x})=0$. We choose 
\begin{equation}\label{eqn:comp}
p=\alpha q, 
\end{equation}
calling the function $p$ compensator.  

Now we apply Lemma \ref{lem:loopEqn}, adding the local compensator $p$ at the singularity in order to satisfy the Hilbert transform condition \eqref{eqn:assumed},
\begin{multline*}
\E( e^{\Tr L_+^x(U_t)+\Tr f(U_0)})  = \E( e^{\Tr L_+^x(U_t)+\Tr p(U_t)}) \exp \left( \int_0^1 \frac{\rd}{\rd\nu} \log \E( e^{ \Tr L_+^x(U_t)+ \nu \Tr f(U_0)+(1-\nu)\Tr p(U_t)})\rd\nu \right) \\
 =  \E( e^{\Tr L_+^x(U_t)+\Tr p(U_t)}) \exp \left( \int_0^1 \mscr{C} (   f_0-p_t, L_{+,t}^x + \nu f_0+(1-\nu)p_t) \rd\nu \right) (1+\OO( N^{-\delta/9}) )
\end{multline*}
where by integrating we have $\int_0^1 \mscr{C} (   f_0-p_t, L_{+,t}^x + \nu f_0+(1-\nu)p_t) \rd\nu = \mscr{C} (  f_0, L_{+,t}^x) + \frac{1}{2} \mscr{C}(f_0,f_0)-\mscr{C}(p_t,L_{+,t}^x)-\frac{1}{2}\mscr{C}(p_t,p_t)$. Notice that the expectation is now just a single-time expression and once we prove the following:
\begin{align} \label{eqn:single_loop_step}
\log\E( e^{\Tr L_+^x+\Tr p(U_t)})=\log\E( e^{\Tr L_+^x(U_t)})+\mscr{C}(p_t,L^x_{+,t})+\frac{1}{2}\mscr{C}(p_t,p_t)+\OO(N^{-\delta/9})
\end{align}
with the approximation $\mscr{C}(f_0, L_{+,t}^x)=\mscr{C}(f_0, L_{t}^x)+\OO(N^{-\delta/9})$, which can be shown easily by the Fourier coefficient expression in \eqref{eq:def-opera}, we obtain
\begin{align*} 
\E( | \det (U_t - e^{\ii x}) |^{\gamma_x} e^{\Tr f(U_0)})  \leq N^{ \frac{\gamma_x^2}{4}} \frac{{\rm G}(1+\frac{\gamma_x}{2})^2}{{\rm G}(1+\gamma_x)} e^{ \mscr{C} ( f_0, L_t^{x})} e^{N \dashint f_0 + \frac{1}{2} \mscr{C}(f_0)} (1+\OO(N^{-\delta/9}) )
\end{align*}
by Lemma \ref{lem:single_smoothing}. Repeating the same steps with lower regularization $\ell_-$ finishes the proof. The proof of equation \eqref{eqn:single_loop_step} is just a straightforward application of the single-time loop equation,  it is given in the Appendix.
\end{proof}

Let $\lambda=N^{-1+\kappa}$ be the mesoscopic scale for the small fixed positive constant $\kappa\ll \delta$ where $\delta$ is the separation parameter for the singularities of $\mscr{A}$.
We denote the mesoscopic regularization of logarithm around $e^{\ii x}$ by $\ell_\lambda^{e^{\ii x}}$ so that $\log f_j+\gamma_x\log (-2\lambda)=\ell_+^{e^{\ii x}}-\ell_{\lambda}^{e^{\ii x}}$ where $f_j$ is defined as in equation \eqref{eq:truncated-singularities} for $z=e^{\ii x}$. Now, we introduce the notation $L^x_+ =\gamma_x\ell_+^{e^{ix}}= L^{x,{\rm loc}}_+ + L^{x,{\rm reg}}_+$ where $L^{x,{\rm loc}}_+$ and $L^{x,{\rm reg}}_+$ stands for local submicroscopically regularized logarithm $\gamma_x(\ell_+^{e^{ix}}-\ell_{\lambda}^{e^{ix}})$ and mesoscopic log -regularization $\gamma_x \ell_{\lambda}^{e^{ix}}$ respectively. An application of the lemma above (with Lemmas \ref{lem:smoothing} and \ref{lem:single_smoothing}) gives
\begin{equation}
\label{eq:asymp-loc-sing}
\E (e^{ \Tr L^{x,\rm{loc}}_+(U_t)})  = N^{\frac{\gamma_x^2}{4}} \frac{{\rm G}(1+\frac{\gamma_x}{2})^2}{{\rm G}(1+\gamma_x)} e^{-\frac{1}{2}(\mscr{C}(L_{+,t}^{x}) - \mscr{C}(L_{+,t}^{x,\rm{loc}}) )}  (1+\OO(N^{-\kappa/9}))
\end{equation}
since  $-\mscr{C}( L_{+,t}^{x,{\rm reg}}, L_{+,t}^{x})+ \frac{1}{2} \mscr{C} (L_{+,t}^{x,{\rm reg}} ) = -\frac{1}{2}(\mscr{C}(L_{+,t}^{x}) - \mscr{C}(L_{+,t}^{x,{\rm loc}}) )$. We are now ready to prove our main theorem. 


\begin{proof}[Proof of Theorem \ref{thm:FH}]  Throughout the proof, $\eps$ that is used for a small positive constant, may change line by line. Again, we start with $\ell_+$ regularization of $\log$, i.e., we consider the expression
\begin{align*} 
\E\Big[e^{\sum_{s \in \mathscr{B}} \Tr f_s(U_s)+\sum_{z = t + \ii \theta\in\mathscr{A}} \Tr \ell_{+}^{e^{\ii\theta}}(U_t)}\Big].
\end{align*} 
We choose the submicroscopic regularization parameter $\alpha$ (used in the definition of $\ell_+$) and the mesoscopic scale parameter $\kappa$ (in the definition of $\ell_\lambda$) so that $\alpha, \kappa \ll \delta$ (where $\ll$ depends on the value of $C$), where $\delta$ is the separation parameter for the singularities at $\mathscr{A}$ appearing in the statement of the theorem. For convenience, we enumerate the singularities as $\ell_+^{e^{\ii x_j}}(U_{t_j})$ and write $|\mathscr{A}| = J$. We denote by $\mc{L} = \mc{L}_{\rm{loc}}+\mc{L}_{\rm{reg}}$ the decomposition of the regularized log-singularity sum into submicroscopic localized terms and mesoscopic smoothing parts (see above equation \eqref{eq:asymp-loc-sing}), and let $\mc{S}$ denote the remaining smooth contributions. To satisfy the Hilbert transform vanishing assumption \eqref{eqn:assumed} in the loop equation, we introduce local compensator functions $p_j = \alpha_j q_j$ (see above \ref{eqn:comp} for the definition of $q$) around each singularity $\ell_+^{x_j}$, and denote their sum by $\mc{P}$, where $\alpha_j$ is to be determined. The Hilbert transform condition \eqref{eqn:assumed} then takes the form:

\begin{multline*} 
\begin{bmatrix}
{\rm HP}_{|t_1-t_1|}q_1(e^{\ii x_1}) & {\rm HP}_{|t_1-t_2|}q_2(e^{\ii x_1}) & \cdots & {\rm HP}_{|t_1-t_J|}q_J(e^{\ii x_1})
\\
{\rm HP}_{|t_2-t_1|}q_1(e^{\ii x_2}) & {\rm HP}_{|t_2-t_2|}q_2(e^{\ii x_2}) & \cdots & {\rm HP}_{|t_2-t_J|}q_J(e^{\ii x_2})
\\
\vdots & \vdots & \ddots & \vdots
\\
{\rm HP}_{|t_J-t_1|}q_1(e^{\ii x_J}) & {\rm HP}_{|t_J-t_2|}q_2(e^{\ii x_J}) & \cdots & {\rm HP}_{|t_J-t_J|}q_J(e^{\ii x_J})
\end{bmatrix} 
\begin{bmatrix}
\alpha_1 \\ \alpha_2 \\ \vdots \\ \alpha_{J}
\end{bmatrix}
\\
=
\begin{bmatrix}
\sum_{z = t + \ii \theta\in\mathscr{A}} {\rm HP}_{|t_1-t|}\ell_{\lambda}^{e^{\ii \theta}}(e^{\ii x_1})+\sum_{s \in \mathscr{B}} {\rm HP}_{|t_1-s|} f_s(e^{\ii x_1})
\\
\vdots
\\
\sum_{z = t + \ii \theta\in\mathscr{A}} {\rm HP}_{|t_J-t|}\ell_{\lambda}^{e^{\ii \theta}}(e^{\ii x_J})+\sum_{s \in \mathscr{B}} {\rm HP}_{|t_J-s|} f_s(e^{\ii x_J}).
\end{bmatrix}
\end{multline*}
The square matrix on the left-hand side has  diagonal entries $\asymp\log N$ and off-diagonal entries $\OO(N^{-\delta+\epsilon})$. On the other hand, every entry of the right-hand side vector is $\OO(\log N)$. Hence, there exist $\OO(1)$ constants $\alpha_j$'s satisfying this system of linear equations.

 We suppose without loss of generality that the smooth functions are centered. Our starting point is the identity
$$
\E (e^{ \mc{S} + \mc{L}}) = \E(e^{\mc{L}_{\rm{loc}}+\mc{P}}) \exp \left( \int_0^1 \frac{\rd}{\rd\nu} \log  \E (e^{\nu (\mc{S}+\mc{L}_{\rm{reg}}-\mc{P}) + \mc{L}_{\rm{loc}}+\mc{P}} ) \rd\nu \right).
$$
Then, by Proposition \ref{Prop:decoupling} we have
\begin{align}
\E(e^{\mc{L}_{\rm{loc}}+\mc{P}}) & =  \prod_j \E (e^{ \Tr L^{x_j,\rm{loc}}_+(U_{t_j})+\Tr p_j(U_{t_j}) } ) \cdot (1 + \OO(N^{-\delta/3}))  \nonumber\\
&=\prod_j \E (e^{ \Tr L^{x_j,\rm{loc}}_+(U_{t_j})} ) e^{\mscr{C}(p_{j,t_j},L^{x_j,\rm{loc}}_{+,t_j})+\frac{1}{2}\mscr{C}(p_{j,t_j})}\cdot (1 + \OO(N^{-\eps})) \label{eqn:single_loop_second_app}
\\
&= \prod_j N^{\frac{\gamma_{x_j}^2}{4}} \frac{{\rm G}(1+\frac{\gamma_{x_j}}{2})^2}{{\rm G}(1+\gamma_{x_j})}  e^{-\frac{1}{2}(\mscr{C}(L_{+,t_j}^{x_j})-\mscr{C}(L_{+,t}^{x_j,{\rm{loc}}}) )+\mscr{C}(p_{j,t_j},L^{x_j,\rm{loc}}_{+,t_j})+\frac{1}{2}\mscr{C}(p_{j,t_j})}  (1 + \OO(N^{-\eps}))\nonumber
\end{align}
where $\epsilon$ is a small positive constant depending only on $C$ and $\delta$; the second equality follows from single-time loop equation similarly to \eqref{eqn:single_loop_step} and the third equality follows from \eqref{eq:asymp-loc-sing}. Furthermore, by Lemma \ref{lem:loopEqn},
\begin{align*}
\int_0^1 \frac{\rd}{\rd\nu} &\log  \E (e^{\nu (\mc{S}+\mc{L}_{\rm{reg}}-\mc{P}) + \mc{L}_{\rm{loc}}+\mc{P}} ) \rd\nu   = \int_0^1 \mscr{C}(\mc{S}+\mc{L}_{\rm{reg}}-\mc{P},\nu (\mc{S}+\mc{L}_{\rm{reg}}-\mc{P})+\mc{L}_{\rm{loc}}+\mc{P})\rm{d}\nu + \OO( N^{-\eps}) ) 
\\
&= \mscr{C}(\mc{S}+\mc{L}_{\rm{reg}}, \mc{L}_{\rm{loc}}) + \frac{1}{2} \mscr{C}(\mc{S}+\mc{L}_{\rm{reg}})-\frac{1}{2}\mscr{C}(\mc{P})+\frac{1}{2}\mscr{C}(\mc{P},\mc{S})-\mscr{C}(\mc{P},\mc{L}_{{\rm loc}})+\frac{1}{2}\mscr{C}(\mc{P},\mc{L}_{{\rm reg}})  + \OO( N^{-\eps}) 
\\
&= \frac{1}{2}\mscr{C}(\mc{S}) + \mscr{C}(\mc{S},\mc{L}) + \frac{1}{2} \mscr{C}(\mc{L}) -  \frac{1}{2} \mscr{C}(\mc{L}_{\rm{loc}})-\frac{1}{2}\mscr{C}(\mc{P})+\frac{1}{2}\mscr{C}(\mc{P},\mc{S})-\mscr{C}(\mc{P},\mc{L}_{{\rm loc}})+\frac{1}{2}\mscr{C}(\mc{P},\mc{L}_{{\rm reg}})+ \OO( N^{-\eps}) .
\end{align*}
Altogether, we obtain
\begin{multline*}
\E (e^{ \mc{S} + \mc{L}})   =e^{\frac{1}{2} \mscr{C}(\mc{S})} e^{\mscr{C}(\mc{S}, \mc{L})}  e^{\frac{1}{2}\mscr{C}(\mc{L})-\frac{1}{2}\sum_{\mscr{A}}\mscr{C}(L_{+,t}^{x})} \prod_{\mathscr{A}} N^{ \frac{\gamma_z^2}{4}} \frac{{\rm G}(1+\frac{\gamma_z}{2})^2}{{\rm G}(1+\gamma_z)} e^{-\frac{1}{2} \mscr{C}(\mc{L}_{\rm{loc}})+\frac{1}{2}\sum_{\mscr{A}}\mscr{C}(L_{+,t}^{x,{\rm loc}})}
\\
e^{-\mscr{C}(\mc{P},\mc{L}_{{\rm loc}})+\sum_{1}^{|\mscr{A}|}\mscr{C}(p_j,L_{+,t_j}^{x_j,{\rm loc}})}e^{-\frac{1}{2}\mscr{C}(\mc{P})+\frac{1}{2}\sum_{1}^{|\mscr{A}|}\mscr{C}(p_j)}(1 + \OO(N^{-\eps})).
\end{multline*}
It can be easily seen by Fourier space definition of $\mscr{C}$ at \eqref{eq:def-opera} that the submicroscopic regularizations in the second and third exponentials can be replaced by pure logarithmic singularities with a negligible error. Moreover, due to the separation condition on the singularities, the exponents of the last three exponential terms in the right-hand side are negligible because the cross terms are negligible. For the sake of brevity we only discuss it for the first exponential out of those there, i.e. we prove that
$$
\mscr{C} (L_0^{x,{\rm loc}}, L_t^{y,{\rm loc}})  = \OO( N^{-\epsilon})
$$
for any pair of distinct singularities. Indeed,  since $\langle f , g \rangle_{L^2( \frac{\lambda}{2\pi} )} = \sum_k \hat{f}_k \hat{g}_{-k}$, $\frac{d}{dt} \langle {\rm P}_{t} f, g\rangle_{L^2( \frac{\lambda}{2\pi} )}   = -\langle {\rm P}_{t} f, g \rangle_{\rm H}$. Note that
$$
\int_{\mb{U}} {\rm P}_t f(w) g(w) \lambda(dw) = \int_{\mb{U}^2} g(w) \frac{1}{2\pi} \re \left( \frac{w'+w e^{-t}}{w'-w e^{-t}} \right) f(w') \lambda(\rd w) \lambda(\rd w')
$$
and $\frac{\rd}{\rd t} \frac{w'+w e^{-t}}{w'-w e^{-t}} = \frac{\rd}{\rd t} \frac{2 w'}{w'-we^{-t}} = - 2 \frac{w' w e^{-t}}{(w'-w e^{-t})^2}$, so, denoting $d=N \min_{i \neq j} (|(e^{\ii x_i},t_i)- (e^{\ii x_j}, t_j) | \geq N^{\delta}$, we bound from above 
\begin{equation}
\label{eq:trunc-sing-indep}
| \mscr{C} (L_0^{x,\lambda}, L_t^{y,\lambda}) | \leq C ( \lambda/N)^2 \cdot \log(\lambda/N)^2\,  \max(1, \min(t^{-2}, |e^{\ii x} - e^{\ii y}|^{-2} ))
\leq N^\kappa \lambda^2/d^2
\end{equation}
by using $\sup_{|w'-e^{\ii x}| < \lambda/N, |w-e^{\ii y}|<\lambda/N} \frac{1}{|w'-w e^{-t} |^2} \leq C \max(1, \min(t^{-2}, |e^{\ii x} - e^{\ii y}|^{-2} ))$.

To conclude,  we use \eqref{eq:def-opera}, \eqref{eq:cov-id-1}, \eqref{eq:cov-id-2} combined with \eqref{eq:cov-field}, 
 This concludes the proof of
 \begin{multline}
 \E\Big[e^{\sum_{s \in \mathscr{B}} \Tr f_s(U_s)  } \prod_{z = t + \ii \theta\in\mathscr{A}} | \det (U_t - e^{\ii \theta}) |^{\gamma_z} \Big]\leq e^{N \sum_{\mathscr{B}}  \dashint f_s + \frac{1}{2} \sum_{\mathscr{B}^2} (f_s, {\rm P}_{|s-s'|} f_{s'})_{\rm H} - \sum_{z \in \mathscr{A}, s\in\mathscr{B}} \frac{\gamma_{z}}{2} ( {\rm P}_{|t-s|} - {\rm P}_{\infty}) f_s (e^{\ii \theta})} \\
\times \prod_{\mathscr{A}} N^{ \frac{\gamma_z^2}{4}} \frac{{\rm G}(1+\frac{\gamma_z}{2})^2}{{\rm G}(1+\gamma_z)} \prod_{z,w \in \mathscr{A}, z \neq w} \left( \frac{\max(|e^z|,|e^w|)}{|e^z-e^w|} \right)^{\frac{1}{4} \gamma_z \gamma_w} (1+{\rm O}( N^{-\eps})).
\end{multline}
For the other direction of the inequality, it suffices to replace all $\ell_+$ smoothings with $\ell_-$'s from the beginning, the steps are identical.
\end{proof}

\subsection{Theorem \ref{thm:gmc-cv}}.
In this proof, we prefer simplicity/brevity to generality and present only the details for the $L^2$ phase (namely $\gamma \in (0, 2)$). The parameter $\gamma$ is fixed throughout the proof so we drop it from the notation. In the Gaussian setup, \cite{Ber2017} gave an elementary approach for the convergence of GMC measures for a natural class of approximations, including the $L^1$ phase (corresponding here to $\gamma \in [2, 2\sqrt{2})$  using barrier estimates. In random matrix theory, the works \cite{LamOstSim2018} and \cite[Section 3]{ NikSakWeb20} explain how barrier estimates and in particular the convergence in the $L^1$ phase follow from Theorem \ref{thm:FH}. 

Let $\mu_{N}^{(\eps)}$  be the $2d$ GMC measure with parameter $\gamma$ associated to the field $h_N^{(\eps)}(t,\cdot) := {\rm P}_{\eps} h_N(t, \cdot)$. For any continuous function $f$ on $[0,1] \times \mb{U}$, the $L^2$ norm of $\int_{[0,1] \times \mb{U}} f (\rd\mu_{N}^{(\eps)} - \rd\mu_{N})$ vanishes when taking $N \to \infty$ and then $\eps \to 0$ (details on this are given below). Furthermore, $h_N^{(\eps)}$ converges to a smooth Gaussian field $h^{(\eps)}$ whose covariance kernel is given by $\E({\rm P}_{\eps} h (s,x) {\rm P}_{\eps} h (t,y)) = \frac{1}{2} \sum_{k \geq 1} \frac{\cos(k (x-y))}{k} e^{-|k| |t-s|} e^{- 2 \eps |k|} = {\rm P}_{2\eps + |t-s|} C(x-y)$ where $C(x-y) = \E(h_0(x) h_0(y))$. Finally, the GMC $e^{\gamma h^{(\eps)}}$ converges to the GMC $e^{\gamma h}$ by \cite[Theorems 3, 25]{Sha2016}. Altogether, this concludes the proof of Theorem \ref{thm:gmc-cv} for $\gamma \in (0,2)$.

Now, we provide some details on the $L^2$ estimates. Three terms arise:
$$
\frac{\E( e^{\gamma h_N^{(\eps)}(s,x)} e^{\gamma h_N^{(\eps)}(t,y)} )}{\E (e^{\gamma h_N^{(\eps)}(s,x)} ) \E (e^{\gamma h_N^{(\eps)}(t,y)} )}, 
\qquad 
\frac{\E( e^{\gamma h_N^{(\eps)}(s,x)} e^{\gamma h_N(t,y)} )}{\E (e^{\gamma h_N^{(\eps)}(s,x)} ) \E (e^{\gamma h_N(t,y)} )},
\quad 
\text{and} 
\qquad 
\frac{\E( e^{\gamma h_N(s,x) } e^{\gamma h_N(t,y)} )}{\E (e^{\gamma h_N(s,x)} ) \E (e^{\gamma h_N(t,y)} )}.
$$
Set $f_x = \log | e^{\ii \cdot} - e^{\ii x} |$ and $f_x^{(\eps)} = {\rm P}_\eps f_x$. By applying Theorem \ref{thm:FH}  (with one singularity or one smooth function), we obtain the asymptotics of the normalizing constants: $\lim_{N \to \infty} \E (e^{\gamma h_N^{(\eps)}(s,x)} ) = e^{\frac{\gamma^2}{2} \| f_x^{(\eps)} \|_{ \rm H}^2}$  and $\lim_{N \to \infty} N^{-\frac{\gamma^2}{4}} \E (e^{\gamma h_N(s,x)} ) = \frac{{\rm G}(1+\gamma/2)^2}{{\rm G}(1+\gamma)}$. Still with Theorem \ref{thm:FH} (and this time only pairwise terms contribute), we obtain the $2$-point asymptotics
\begin{align*}
\frac{\E( e^{\gamma h_N^{(\eps)}(s,x)} e^{\gamma h_N^{(\eps)}(t,y)} )}{\E (e^{\gamma h_N^{(\eps)}(s,x)} ) \E (e^{\gamma h_N^{(\eps)}(t,y)} )} & \underset{N\to\infty}{\sim} e^{\frac{\gamma^2}{2} \times 2 (f_x^{(\eps)}, {\rm P}_{|t-s|} f_y^{(\eps)})_H }  = e^{\gamma^2 {\rm P}_{|t-s|+2\eps} C(x-y)},\\
 \frac{\E( e^{\gamma h_N^{(\eps)}(s,x)} e^{\gamma h_N(t,y)} )}{\E (e^{\gamma h_N^{(\eps)}(s,x)} ) \E (e^{\gamma h_N(t,y)} )} &\underset{N\to\infty}{\sim} e^{- \frac{\gamma^2}{2} {\rm P}_{|t-s|} f_x^{(\eps)}(e^{\ii y}) } = e^{\gamma^2 {\rm P}_{|t-s|+\eps} C(x-y)},\\
 \frac{\E( e^{\gamma h_N(s,x) } e^{\gamma h_N(t,y)} )}{\E (e^{\gamma h_N(s,x)} ) \E (e^{\gamma h_N(t,y)} )}  & \underset{N\to\infty}{\sim} (e^{2 {\rm P}_{|t-s|} C(x-y)})^{\frac{\gamma^2}{4} \times 2},
\end{align*}
where we used \eqref{eq:cov-field} for the last equality.  
Note that the above asymptotics  hold {\it uniformly} in the domain allowed for   $(x,s), (y,t)$ in Theorem \ref{thm:FH} .

With $f_x(y) = - \sum_{k \geq 1} \frac{1}{2k}  (e^{\ii k(x-y)}+e^{-\ii k(x-y)})$, we find $(f_x, f_y)_{\rm H} = \frac{1}{2} \sum_{k \geq 1} \frac{\cos(k (x-y))}{k}$ and $(f_x^{(\eps)}, {\rm P}_{|t-s|} f_y^{(\eps)})_H = \frac{1}{2} \sum_{k \geq 1} \frac{\cos(k(x-y))}{k} e^{-|t-s|k} e^{-2 \eps k}$. 

For small mesoscopic contributions, we use the Cauchy-Schwarz inequality and obtain (again from \eqref{eq:asymptotics} but with a $2\gamma$ singularity) as $N \to \infty$, 
$$
\frac{\E(e^{2\gamma h_M^{(\eps)}(0,0)})}{\E(e^{\gamma h_M^{(\eps)}(0,0)})^2} \asymp \frac{N^{\frac{(2\gamma)^2}{4}}}{N^{\frac{\gamma^2}{2}\times 2}} = N^{\frac{\gamma^2}{2}}
$$
so, for $\eps$ small enough, the contributions to the $L^2$ norm of the points $z,w \in [0,1] \times \mb{U}$ with $|z-w| < N^{-1+\eps}$ vanishes. Therefore, $\lim_{\eps \to 0} \lim_{N \to \infty} \E (\int_{[0,1] \times \mb{U}} f (\rd\mu_{N}^{(\eps)} - \rd\mu_{N}))^2 $ is equal to
$$
\lim_{\eps \to 0}  \int_{([0,1] \times \mb{U} )^2} f(s,x) f(t,y)  (e^{\gamma^2 {\rm P}_{|t-s|+2\eps} C(x-y)} - 2 e^{\gamma^2 {\rm P}_{|t-s|+\eps} C(x-y)}  + e^{\gamma^2 {\rm P}_{|t-s|} C(x-y)} ) = 0,
$$
hence the aforementioned $L^2$ estimate.
To justify the above limit, let us denote $A_\omega:=\int f(s,x) f(t,y) e^{\gamma^2 {\rm P}_{|t-s|+\omega} C(x-y)}$. Because the integrand is non-negative we have $A_{2\epsilon}-2A_{\epsilon}+A_0\geq 0$ and by Fatou's lemma we have $\liminf_{\epsilon\to 0} (2A_\epsilon-A_{2\epsilon})\geq A_0$; combining these two completes the proof.

This paper is focused on the measures as in our framework the limiting $2d$ LQG measure is connected with many topics of $2d$ random geometry, as outlined above. However, Theorem \ref{thm:FH} has other direct consequences which we list below.

\begin{remark}\label{rem:logcor} Theorem \ref{thm:FH} implies directly the pointwise convergence of $h_N(z)=\log|\det(e^{\ii\theta}-U_t)|$ (where $z=t+\ii\theta$) to a Gaussian logarithmically-correlated field: $(\frac{1}{2}\log N)^{-1/2}(h_N(z),h_N(z'))$ converges in distribution to $(\mathscr{N}_z,\mathscr{N}_{z'})$ where these standard Gaussians  have asymptotic covariance $-\log|z-z'|/\log N$ for $|z-z'|$ on mesoscopic scale. 
\end{remark}

\begin{remark}\label{rem:max} For $\Omega$  any fixed compact set in $\mathbb{R}\times \mathbb{U}$ with non-empty interior, yet
another corollary is the asymptotics 
$$(\log N)^{-1}\max_{z\in\Omega}|h_N(z)|\to \sqrt{2}$$ 
in probability,  i.e.  the space-time analogue of the main result in \cite{ArgBelBou2017}.  For fixed time this maximum is known up to second order \cite{PaqZei2018},  tightness \cite{ChaMadNaj2018} and distribution \cite{PaqZei2022}; it is an interesting question whether Theorem \ref{thm:FH} can help to approach this precision on $\Omega$, or if our $2d$ framework is useful to study fine properties of the maximum of the $1d$ restriction of the field.

In the same vein as equation (\ref{eq:LQGconvergence2}),  Theorem \ref{thm:FH} (more precisely its natural analogue for $\im\log$) also captures the maximum deviation of the eigenvalues along trajectories.  Indeed, ordering the initial eigenangles at equilibrium $0\leq \theta_1(0)\leq\dots\leq \theta_N(0)\leq 2\pi$,  and denoting $\gamma_k=\frac{2\pi k}{N}$, $t=N^{-1+\beta}$ ($0\leq \beta\leq 1$), we have (in probability),
\begin{equation}\label{eqn:maxDyn}
\frac{N}{\log N}\max_{0\leq s\leq t,1\leq k\leq N}|\theta_k(s)-\gamma_k|\to 2\sqrt{1+\beta}.
\end{equation}
\end{remark}

Finally,  we note two interesting questions related to our results.  First,  in the context of random tilings, (\ref{eqn:maxDyn}) is the analogue of the maximal  deviation of the height function from the hydrodynamic limit.  Asymptotics of this maximum and convergence to LQG are not known in this context.
Moreover, instead of considering an infinite volume surface, the unitary Brownian bridge with same (Haar-distributed) starting ($t=0$) and ending point ($t=1$) provides a natural framework in Random Matrix Theory to generate the LQG measure on a finite volume surface without boundary, the torus $\mb{R}/2\pi \mb{Z} \times \mb{R} / \mb{Z}$.  The general surgery and some methods developed in this work may apply to these problems.

\nc
\setcounter{equation}{0}
\setcounter{theorem}{0}
\renewcommand{\theequation}{A.\arabic{equation}}
\renewcommand{\thetheorem}{A.\arabic{theorem}}
\appendix
\setcounter{secnumdepth}{0}
\section[Appendix]{Appendix}

In the following paragraphs, we present some standard formulas, accompanied with a proof to be self-contained and some identities that are used in the manuscript as well as an extension of Theorems \ref{thm:gmc-cv} and \ref{thm:FH} to incorporate jump singularities.\\

\noindent{\bf Helffer-Sj{\H o}strand formula, Hilbert transform and Poisson kernel}.  This paragraph presents the natural unitary analogue of the classical 
Helffer-Sj{\H o}strand formula, originally used to develop an alternative functional calculus for self-adjoint operators \cite{Dav} and of great use in random matrix theory, see \cite{ErdYau2017}, and its interplay with the Poisson kernel.

  Let $\tilde g=g(w)$ be a quasi-analytic extension of $g$, i.e. 
$g$ and $\tilde g$ coincide on the unit circle and $\partial_{\bar w}\tilde g(w)={\rm O}(||w|-1|)$ (We could also impose $\partial_{\bar w}\tilde g(w)={\rm O}(||w|-1|^p)$ for arbitrary fixed $p\geq1$).  
In practice we often use the following natural analogue of the Hermitian formulas from \cite{Dav,ErdYau2017}, with representation in polar coordinates ($w=r e^{\ii\theta}$), as in \cite{AdhLan2023}: 
\begin{equation}
\label{eqn:HS-exp}
\tilde g(w)=(g(e^{\ii\theta})-\ii g'(e^{\ii\theta}) \log r)\chi(r),
\end{equation}
where $\chi=\chi_c=1$ on $\exp([-c,c])$,  $0$ on $\exp([-2c,2c]^{\rm c})$,  and $|\chi'|\leq 10c^{-1}$, $|\chi''|\leq 10c^{-2}$.  Furthermore, we used the notation $g'(e^{\ii\theta})$  for the differential of $\theta\mapsto g(e^{\ii\theta})$, and similarly for $g''$. Note that for this specific form of $\tilde g$  we have
\begin{equation}\label{eqn:dbar}
\partial_{\bar w} \tilde g(w)=\frac{e^{\ii\theta}}{2}(g(e^{\ii\theta})-\ii g'(e^{\ii\theta}) \log r)\chi'(r)+\frac{e^{\ii\theta}}{2r}   g''(e^{\ii\theta}) \chi(r) \log r.
\end{equation}

Let $m$ denote the Lebesgue measure on $\mathbb{C}$.  Assume also that $\tilde g$ is compactly supported.  Green's theorem  in complex coordinates can be written (in the case of outer boundary)
$$
\int_D \partial_{\bar w}f(w) \rd m(w)=\frac{1}{2}\int_D (\partial_x f-\partial_y (-\ii f)) \rd m(w)=\frac{1}{2}\int_{\partial D}\left(-\ii f\rd x+f\rd y\right)=-\frac{\ii}{2}\int_{\partial D}f(w)\rd w.
$$
This gives, for any $|z|< 1$, (note that we have a sign change due to inner boundary)
\begin{multline}\label{eqn:upper}
\frac{1}{\pi}\int_{|w|>1}\partial_{\bar w}\tilde g(w)\cdot \frac{z+w}{z-w}\frac{\rd m(w)}{w}=\frac{1}{\pi}\int_{|w|>1}\partial_{\bar w}\left(\tilde g(w) \frac{z+w}{z-w}\right)\frac{\rd m(w)}{w}\\
=\frac{\ii}{2\pi}\int_{|w|=1}g(w) \frac{z+w}{z-w}\frac{\rd w}{w}=-\int_0^{2\pi}g(e^{\ii\theta})\frac{z+e^{\ii\theta}}{z-e^{\ii\theta}}\frac{\rd\theta}{2\pi}.
\end{multline}
Similarly,  for any $|z|> 1$,
\begin{multline}\label{eqn:lowerbis}
\frac{1}{\pi}\int_{|w|<1}\partial_{\bar w}\tilde g(w)\cdot \frac{z+w}{z-w}\frac{\rd m(w)}{w}=\frac{1}{\pi}\int_{|w|<1}\partial_{\bar w}\left(\tilde g(w) \frac{z+w}{z-w}\right)\frac{\rd m(w)}{w}\\
=-\frac{\ii}{2\pi}\int_{|w|=1}g(w) \frac{z+w}{z-w}\frac{\rd w}{w}=\int_0^{2\pi}g(e^{\ii\theta})\frac{z+e^{\ii\theta}}{z-e^{\ii\theta}}\frac{\rd\theta}{2\pi}.
\end{multline}
Defining $w_t=we^{t} \mathds{1}_{|w| > 1} + w e^{- t} \mathds{1}_{|w|<1}$,  from (\ref{eqn:upper}) and (\ref{eqn:lowerbis}) we obtain,  for any 
$|z|=1,z=e^{\ii\varphi}$ and $t>0$ we have
\begin{multline}\label{eqn:HSPoisson}
\frac{1}{2\pi}\int_{\mathbb{C}}\partial_{\bar w}\tilde g(w)\cdot \frac{z+w_t}{z-w_t}\frac{\rd m(w)}{w}=
-\frac{1}{2}\int_0^{2\pi}g(e^{\ii\theta})\frac{ze^{-t}+e^{\ii\theta}}{ze^{-t}-e^{\ii\theta}}\frac{\rd\theta}{2\pi}
+\frac{1}{2}\int_0^{2\pi}g(e^{\ii\theta})\frac{ze^{t}+e^{\ii\theta}}{ze^{t}-e^{\ii\theta}}\frac{\rd\theta}{2\pi}\\
=
\int_0^{2\pi}g(e^{\ii\theta})\frac{1}{2}\left(
\frac{1+e^{-t}e^{\ii(\varphi-\theta)}}{1-e^{-t}e^{\ii(\varphi-\theta)}}
+
\frac{1+e^{-t}e^{\ii(\theta-\varphi)}}{1-e^{-t}e^{\ii(\theta-\varphi)}}
\right)\frac{\rd\theta}{2\pi}
=
\int_0^{2\pi}g(e^{\ii\theta}){\rm Re}
\frac{1+e^{-t}e^{\ii(\varphi-\theta)}}{1-e^{-t}e^{\ii(\varphi-\theta)}}
\frac{\rd\theta}{2\pi}={\rm P}_t g(z).
\end{multline}
Recall that the Hilbert transform of a function on $\partial\mathbb{D}$ can be defined through a principal value or in Fourier space:
\begin{equation}\label{eqn:Hilbert}
{\rm H} f(e^{\ii\theta})=\int_0^{2\pi}\frac{f(e^{\ii\varphi})-f(e^{\ii\theta})}{\tan\frac{\varphi-\theta}{2}}\frac{1}{2\pi}\rd\varphi,\ \ \ 
{\rm H} f(z)=\sum_{k\in\mathbb{Z}} \ii (\mathds{1}_{k\geq1}-\mathds{1}_{k\leq-1})\hat f_k z^k.
\end{equation}
From (\ref{eqn:upper}) and (\ref{eqn:lowerbis}), with  converging series expansion,  we easily obtain
\begin{equation}\label{eqn:HilbertPoisson}
\frac{1}{2\pi}\left(\int_{|w|>1}-\int_{|w|<1}\right)\partial_{\bar w}\tilde g(w)\cdot \frac{z+w_t}{z-w_t}\frac{\rd m(w)}{w}=\ii\, {\rm H}\, {\rm P}_tf(z).
\end{equation}
Finally, for general $z$ we have
\begin{multline}\label{eqn:HS-rep}
\frac{1}{2\pi}\int_{\mathbb{C}}\partial_{\bar w}\tilde g(w)\cdot \frac{z+w}{z-w}\frac{\rd m(w)}{w}=\frac{1}{2\pi}\lim_{\varepsilon\to 0}\int_{D(z,\varepsilon)^{\rm c}}\partial_{\bar w}\left(\tilde g(w) \frac{z+w}{z-w}\right)\frac{\rd m(w)}{w}\\
=\frac{\ii}{4\pi}\lim_{\varepsilon\to 0}\int_{C(z,\varepsilon)}\tilde g(w) \frac{z+w}{z-w}\frac{\rd w}{w}=\frac{\ii}{2\pi}\tilde g(z)\lim_{\varepsilon\to 0}\int_{C(z,\varepsilon)} \frac{\rd w}{z-w}=\tilde g(z).
\end{multline}

\noindent{\bf Poisson summation}.
We denote by $p_t(x)$ the one-dimensional heat kernel on the real line, i.e., $p_t(x) = \frac{e^{-x^2/2t}}{\sqrt{2\pi t}}$. The formula below is a generalization of the usual Poisson summation formula and is related with the transformation formula of the theta function.

\begin{lemma}
\label{lem:Poisson-Summation-Formula}
For every $\delta \in \mb{R}$, $x \in \mb{R}$ and $t >0$, $\sum_{ k \in \mb{Z}} e^{2 \ii \pi k \delta}  p_t(x + 2k \pi)  = \frac{1}{2\pi} \sum_{n \in \mb{Z}}  e^{\ii x (n + \delta)} e^{- \frac{(n+\delta)^2 t}{2}}$.
\end{lemma}
\begin{proof} This  follows by writing, with $B_t$ distributed as a centered Gaussian variable with variance $t$,
$$
e^{- \frac{(n+\delta)^2 t}{2}} = \E ( e^{ i (n+\delta) B_t}) = \sum_{ k \in \mb{Z}} \int_{2k\pi}^{2(k+1)\pi} e^{\ii (n +\delta) y} p_t(y) \rd y =   \sum_{ k \in \mb{Z}} \int_0^{2\pi} e^{\ii (n+\delta) (u+2k\pi)} p_t(u + 2k \pi) \rd u
$$
multiplying it by $e^{- \ii y(n+\delta)} $ and by summation, i.e. 
$$
\sum_{n \in \mb{Z}} e^{- \ii y(n+\delta)} e^{- \frac{(n+\delta)^2 t}{2}}  = \lim_{N \to \infty} \int_0^{2\pi}  \left(\sum_{n = -N}^N e^{\ii n (u-y)} \right) e^{\ii \delta (u-y) }  \sum_{k \in \mb{Z}} e^{2 \ii k \pi \delta} p_t(u+2k\pi) \rd u.
$$
The limit follows from basic properties of the Dirichlet kernel.
\end{proof}

\noindent{\bf Single-time loop equation}. 
In this subsection we prove (\ref{eqn:single_loop_step}).

Loop equations, also called Schwinger–Dyson equations, are functional identities that have long been a central analytic method in random matrix theory, widely used to study global and local statistics of eigenvalues, they underpin many modern results on fluctuations, large deviations, etc. The following lines outline the main steps in proving the equations \eqref{eqn:single_loop_step} and \eqref{eqn:single_loop_second_app},  closely following \cite{lambert2021mesoscopic}.

In this subsection, all the functions are assumed to have domain $\T=[-\pi,\pi]$ for convenience. Given function $f:\T\to\R$, define the centered linear statistics $S(f)=\sum_{k=1}^{N}f(\theta_k)-\dashint_{\T} f(\theta)\rd \theta$ and let $\nu$ denote the centered empirical spectral measure, i.e., $\int f\rd\nu=S(f)$. The single-time loop equation for the CUE (Lemma 2.1 in \cite{lambert2021mesoscopic}) is given by
 \begin{align}\label{eqn:single_loop}
\mathbb{E}_{f}(S(-{\rm H} h))=&\int_{\T} h(x)f'(x)\frac{1}{2\pi}\rd x+\frac{1}{N}\mathbb{E}_{f}(S(h f'))+\frac{1}{2N}\mathbb{E}_{f}\Big(\int_{\T}\int_{\T} \frac{h(x)-h(y)}{\tan(\frac{x-y}{2})}\rd\nu(x)\rd\nu(y)\Big)
\end{align}
where the biased measure is given by $\mathbb{E}_{f}[\cdot]=\mathbb{E}[\cdot\frac{e^{\Tr(f)}}{\mathbb{E}[e^{\Tr(f)}]}]$ and $\rm H$ is defined as in \eqref{eqn:Hilbert}. By the Fourier space representation of Hilbert transform one can directly see that for all $f\in L^2(\T)$, $-{\rm H}( {\rm H}f)=f-\hat{f}_0$ a.e. Moreover, if $f\in \mathscr{C}^{1}(\T)$, the equality holds everywhere on $\T$. 

Now, we can discuss the proof of \eqref{eqn:single_loop_step}. Since it's a single-time problem from now on, we omit the time indices and also write $\Tr f$ for $\Tr f(U)$. Moreover, because the problem is rotationally symmetric, we can assume that the singularity is at $0$, and so we omit the $x$ superscript in $L_+^x$ as well. Applying the loop equation \eqref{eqn:single_loop} for $f=L_++\nu p$ and $h={\rm H}p-{\rm H}p(0)$, assuming the second and the third terms in the loop equation \eqref{eqn:single_loop} are negligible we obtain,
\begin{multline*} 
\E( e^{\Tr L_++\Tr p})=\E( e^{\Tr L_+})\exp \left( \int_0^1 \frac{\rd}{\rd\nu} \log  \E (e^{\Tr L_++\nu \Tr p} ) \rd\nu \right)=\E( e^{\Tr L_+})\exp \left( \int_0^1 \E_{(L_++\nu p)} ( \Tr p) \rd\nu \right)
\\
=\E( e^{\Tr L_+})\exp \left( \int_0^1 \mscr{C}(L_++\nu p,p) \rd\nu \right)(1+\OO(N^{-\delta/9})) =\E( e^{\Tr L_+})e^{\mscr{C}(p,L_{+})+\frac{1}{2}\mscr{C}(p,p)}(1+\OO(N^{-\delta/9})).
\end{multline*}
The key difference between the multi-time loop equation asymptotics (Lemma \ref{lem:loopEqn}) and the single-time loop equation, which leads us using the latter one at the end of the proof of Theorem \ref{thm:FH}, is the fact that in the single-time loop equation the freedom of choice for the constant $c$ in $h={\rm H}p+c$ eliminates the need for the compensator. Thus, the only things left are to prove the following two estimations
\begin{align}\label{eqn:two_est}
\mathbb{E}_{L_++\nu p}\big(S\big(({\rm H}p-{\rm H}p(0)) (L_++\nu p)'\big)\big)=\OO(N^{1-\delta/9})
,\quad
\mathbb{E}_{L_++\nu p}\Big(\int_{\T}\int_{\T} \frac{{\rm H}p(x)-{\rm H}p(y)}{\tan(\frac{x-y}{2})}\rd\nu(x)\rd\nu(y)\Big)=\OO(N^{1-\delta/9})
\end{align}
for $\nu\in[0,1]$. For simplicity, we consider the case $\nu=1$, the proof is identical for the other values of $\nu$. 

Recall that the rigidity holds under this biased measure as discussed as a result of Lemma \ref{lem:Johansson} and the rigidity can also be expressed as follows: Given the centered eigenangle counting function on $[-\pi,\pi]$, $g(\theta)=\sum_{k=1}^N \mathds{1}_{\theta_k\in[-\pi,\theta]}-\frac{N(\theta+\pi)}{2\pi}$, we have $\mathbb{P}_{L_++p}(\sup|g|\leq(\log N)^{10})\geq 1-N^{100}$. Thus, by integration by parts (e.g. see Proposition 1.3 in \cite{lambert2021mesoscopic})
\begin{align*} 
\mathbb{E}_{L_++ p}\big(\big|S\big(({\rm H}p-{\rm H}p(0)) (L_++ p)'\big)\big|\big)\leq (\log N)^{10}\int \Big|({\rm H}p-{\rm H}p(0)) (L_++ p)''+({\rm H}p)' (L_++ p)'\Big|+ \oo(N^{-1}).
\end{align*}
Substituting the following estimations which can be simply obtained by Taylor expansion, the required result follows easily
\begin{align} \label{eqn:hlbrt_estimate}
|({\rm H}p)^{(k)}(x)|\leq\begin{cases} N^{k(1-2\delta/5)}(\log N)^2, & |x|\leq 3N^{-1+\delta/2} \\ \frac{N^{-1+\delta/2}}{|x|^{k+1}}(\log N)^2, &{\rm otherwise} \end{cases}, \quad {\rm for }\ k=0,1,2,3.
\end{align}
Here, for the sake of brevity we discuss this upper bound only for $({\rm H}p)'(x)$ when $|x|\leq 3N^{-1+\delta/2}$, the rest follows by the same way. In the following, the inequalities hold up to an absolute constant factor,
\begin{multline*} 
|({\rm H}p)'(x)|\leq \int_\T \Big|\frac{-p'(x)}{\tan(\frac{y-x}{2})}+\frac{1}{2}\frac{p(y)-p(x)}{\sin^2(\frac{y-x}{2})}\Big|\rd y
\leq  \int_\T \Big|-p'(x)\big(\frac{1}{y-x}+\OO(|y-x|))\big)+(p(y)-p(x))\big(\frac{1}{(y-x)^2}+\OO(1)\big)\Big|\rd y
\\
\leq \OO(1)+\int_\T \Big|-\frac{p'(x)}{y-x}+\frac{p(y)-p(x)}{(y-x)^2}\Big|\rd y \leq \OO(1)+\int_{|y|< 4N^{-1+\delta/2}} \|p''\|_\infty \rd y +\big(\|p'\|_\infty\log N+\|p\|_\infty N^{1-\delta/2}\big)
=\OO(N^{1-2\delta/5})
\end{multline*}
where we have used the Taylor expansion for $p$ while evaluating the integral over the region $|y|< 4N^{-1+\delta/2}$.

The second inequality in \eqref{eqn:two_est} follows similarly: first, applying integration by parts with rigidity (Proposition 1.3 in \cite{lambert2021mesoscopic}) we obtain 
\begin{multline*} 
\mathbb{E}_{L_++ p}\Big(\Big|\int_{\T}\int_{\T} \frac{{\rm H}p(x)-{\rm H}p(y)}{\tan(\frac{x-y}{2})}\rd\nu(x)\rd\nu(y)\Big|\Big)
\\
\leq (\log N)^{20} \int_\T\int_\T \underbrace{\Big|\frac{({\rm H}p)'(x)+({\rm H}p)'(y)}{(x-y)^2}-2\frac{({\rm H}p)(x)-({\rm H}p)(y)}{(x-y)^3}\Big|}_{=:(\ast)}\rd x\rd y+ \oo(N^{-1}).
\end{multline*}
We evaluate this double integral in four separate regions as follows:
\begin{align*}
\int\int_{\substack{|x|<3N^{-1+\delta/2}\\ |y|<4N^{-1+\delta/2} }} (\ast)\rd x\rd y&\leq \|({\rm H}p)'''\|_{\infty}(N^{-1+\delta/2})^2\log N=\OO(N^{1-\delta/5}\log N)
\\
\int\int_{\substack{|x|<3N^{-1+\delta/2}\\ |y|>4N^{-1+\delta/2} }} (\ast)\rd x\rd y &\leq \log N \int_{|x|<3N^{-1+\delta/2}}\|({\rm H}p)'\|_{\infty}N^{1-\delta/2}+\|{\rm H}p\|_{\infty}(N^{1-\delta/2})^2 \rd x=\OO(N^{1-2\delta/5}\log N)
\\
\int\int_{\substack{|y|>|x|>3N^{-1+\delta/2}\\ |x-y|<|x|/2 }} (\ast)\rd x\rd y &\leq \log N\int\int_{\substack{|y|>|x|>3N^{-1+\delta/2}\\ |x-y|<|x|/2 }} \|({\rm H}p)'''\|_{L^\infty(B_{|x|/2}(x))}\rd x \rd y =\OO(N^{1-\delta/2}\log N)
\\
\int\int_{\substack{|y|>|x|>3N^{-1+\delta/2}\\ |x-y|>|x|/2 }} (\ast)\rd x\rd y&\leq \int_{|x|>3N^{-1+\delta/2}}\frac{\|({\rm H}p)'\|_{L^\infty(B_{|x|}(0)^c)}}{|x|}+\frac{\|{\rm H}p\|_{L^\infty(B_{|x|}(0)^c)}}{|x|^2} \rd x=\OO(N^{1-\delta/2}\log N)
\end{align*}
which completes the proof of \eqref{eqn:single_loop_step}. The only change for the proof of \eqref{eqn:single_loop_second_app} is to obtain similar estimations to \eqref{eqn:hlbrt_estimate} for ${\rm H}L_+^{\rm loc}$ which can be shown similarly by Taylor expansions. \\

\noindent{\bf Jump Singularities}. With a few minor changes, the proof of Theorem \ref{thm:FH} applies to other singularities: the discontinuities from $\im \log$. 
We only treated the logarithmic singularity from $\re\log$ for the sake of conciseness.
Indeed, define $\im\log\det(1-e^{-\ii\theta}U_t)=\sum_k\im\log (1-e^{\ii(\theta_k(t)-\theta)})$, with the branch choice $\im \log (1-e^{\ii\varphi})=(\varphi-\pi)/2$ if $\varphi\in[0,\pi)$,  $(\varphi+\pi)/2$ if $\varphi\in(-\pi,0)$. For later convenience, define $\arg^\theta=\im\log (1-e^{\ii(\cdot-\theta)})$. Denoting $z=t+\ii x$ and $w=s+\ii y$ and assuming $\beta_z\in [-C,C]$ in addition to the conditions of Theorem \ref{thm:FH}, when the jump singularities are involved the asymptotic formula in the theorem reads as follows:
\begin{multline}
\label{eq:asymptotics_withjump}
\E\Big[e^{\sum_{s \in \mathscr{B}} \Tr f_s(U_s)  } \prod_{z = t + \ii \theta\in\mathscr{A}} | \det (U_t - e^{\ii \theta}) |^{\gamma_z} e^{\beta_z \im\log\det(1-e^{-\ii\theta}U_t)}\Big]
\\
= e^{N \sum_{\mathscr{B}}  \dashint f_s + \frac{1}{2} \sum_{\mathscr{B}^2} (f_s, {\rm P}_{|s-s'|} f_{s'})_{\rm H} - \sum_{z \in \mathscr{A}, s\in\mathscr{B}} \big( \frac{\gamma_{z}}{2} ( {\rm P}_{|t-s|} - {\rm P}_{\infty}) f_s (e^{\ii \theta})+\frac{\beta_z}{2}{\rm P}_{|t-s|}{\rm H}f_s(e^{\ii \theta})\big)} \\
\times \prod_{\mathscr{A}} N^{\frac{\gamma_z^2+\beta_z^2}{4}}\frac{G(1+\frac{\gamma_z}{2}+\ii\frac{\beta_z}{2})G(1+\frac{\gamma_z}{2}-\ii\frac{\beta_z}{2})}{G(\gamma_z+1)} \prod_{z,w \in \mathscr{A}, z \neq w} \left( \frac{\max(|e^z|,|e^w|)}{|e^z-e^w|} \right)^{\frac{1}{4} \gamma_z \gamma_w+\frac{1}{4}\beta_z\beta_w} 
\\
\prod_{z,w \in \mathscr{A}, z \neq w} e^{\gamma_z\beta_w\mscr{C}(\ell^x_t,\arg_{s}^{y})} 
(1+{\rm O}( N^{-\eps}))
\end{multline}
where we have substituted \cite[(71)]{KeaSna2000} for the asymptotics with both jump and log-type singularities and we have calculated (here we denote $\ell^x(\theta)=\log|e^{\ii x}-e^{\ii \theta}|$)
\begin{align} 
\mscr{C}(f_0, \arg_t^x) & = \sum_{k \neq 0} |k| \hat{f}_k \cdot ( -\ii \frac{e^{\ii k x}}{2k} 1_{k \neq 0} ) e^{- |k| |t|}  = -\frac{1}{2}{\rm P}_{|t|}{\rm H}f(e^{\ii x}), \label{eq:cov-id-3}
\\
\mscr{C}(\ell^x_0, \arg_t^y) & = \sum_{k \neq 0} |k|  (-\frac{1}{2|k|}e^{-\ii kx})(-\frac{\ii}{2k}e^{\ii ky}) e^{- |k| |t|}=\frac{1}{2}\sum_{k=1}^{\infty}\frac{e^{-k|t|}}{k}\sin(k(x-y)), \label{eq:cov-id-4}
\\
\mscr{C}(\arg^x_0, \arg_t^y) & = \sum_{k \neq 0} |k|  (\frac{\ii}{2k}e^{-\ii kx})(-\frac{\ii}{2k}e^{\ii ky}) e^{- |k| |t|} =\frac{1}{2}\log \frac{\max( |e^z|,|e^w|) }{|e^z-e^w|}. \label{eq:cov-id-5}
\end{align}
The explicit computations to obtain \eqref{eq:asymptotics_withjump} require regularizing the singularities at a submicroscopic scale and performing detailed error term estimates which will not be carried out here. 

Given $\gamma^2+\beta^2<8$, 
(\ref{eq:asymptotics_withjump})
 allows us to state the following 2$d$ analogue of Theorem 2.5 of \cite{Web2015},
\begin{align} 
\label{eq:LQGconvergence_withjump}
\lim_{N \to \infty }\frac{|\det(U_{t}-e^{\ii \theta})|^\gamma e^{\beta\Im\log\det(1-e^{-\ii \theta }U_t)}}{\E( |\det(U_{t}-e^{\ii \theta})|^{\gamma}e^{\beta\Im\log\det(1-e^{-\ii \theta }U_t)} )} \rd t \rd\theta = e^{\sqrt{\gamma^2+\beta^2} h(z)} \rd z
\end{align}
and in particular for $\gamma=0$ this gives (\ref{eq:LQGconvergence2}).
In the above equation
the GMC  $e^{\sqrt{\gamma^2+\beta^2} h(z)} \rd z$ is as defined in Theorem \ref{thm:gmc-cv}, i.e. associated with 
the Gaussian free field $h$ on the cylinder $\mb{R} \times \mb{R}/2\pi \mb{Z}$, $\E(h(z)h(w))= \pi (-\Delta_{\mc{C}})^{-1}(z,w)$.

When establishing the analogue of Theorem \ref{thm:gmc-cv} as a result of the asymptotics \eqref{eq:asymptotics_withjump}, the key thing to notice is that given $\gamma_z=\gamma$ and $\beta_z=\beta$ for all $z\in\mscr{A}$, the terms involving $\mscr{C}(\ell^x_0, \arg_t^y)$ will vanish due to:
\begin{align*} 
\mscr{C}(\ell^x_0, \arg_t^y)+\mscr{C}(\ell^y_t, \arg_0^x)=0
\end{align*}
which can be seen by \eqref{eq:cov-id-4}.

\tableofcontents

\end{document}